\documentclass[11pt,reqno]{amsart}
\usepackage[T1]{fontenc}
\usepackage{lmodern}
\usepackage{textcomp}
\usepackage{amsmath,amssymb,mathtools}
\usepackage[margin=1.08in]{geometry}
\usepackage{microtype}
\usepackage{enumitem}
\usepackage{tikz-cd}
\usepackage{xcolor}
\usepackage[backend=biber,style=numeric,sorting=nty]{biblatex}
\usepackage[colorlinks=true,linkcolor=blue!40!black,citecolor=blue!40!black,urlcolor=blue!40!black]{hyperref}
\hypersetup{pdftitle={The generalized Dixmier conjecture and the two-dimensional Jacobian conjecture},pdfauthor={Alexey Kanel-Belov and Timofey Fedorov}}
\DeclareFieldFormat{labelnumberwidth}{#1\adddot}
\DeclareNameAlias{default}{given-family}
\DeclareFieldFormat{titlecase}{#1}
\DeclareFieldFormat[article,incollection,misc,online,thesis]{title}{\mkbibemph{#1}}
\DeclareFieldFormat{journaltitle}{#1}
\DeclareFieldFormat[article]{volume}{\mkbibbold{#1}}
\DeclareFieldFormat[article]{number}{no.\addnbspace#1}
\DeclareFieldFormat[article]{pages}{#1}
\DeclareFieldFormat[incollection]{volume}{vol.\addnbspace#1}
\DeclareFieldFormat[incollection]{pages}{pp.\addnbspace#1}
\DeclareFieldFormat{usera}{\href{https://mathscinet.ams.org/mathscinet/article?mr=#1}{MR\,#1}}
\DeclareSourcemap{
  \maps[datatype=bibtex]{
    \map{\step[fieldsource=mrnumber,fieldtarget=usera]}
  }
}
\newbibmacro*{mr:links}{%
  \iffieldundef{doi}
    {\iffieldundef{eprint}{\printfield{url}}{\usebibmacro{eprint}}}
    {\printfield{doi}}%
  \newunit\printfield{usera}}
\DeclareBibliographyDriver{article}{%
  \usebibmacro{bibindex}\usebibmacro{begentry}%
  \printnames{author}\setunit{\addcomma\space}%
  \printfield{title}\setunit{\addcomma\space}%
  \printfield{journaltitle}\setunit*{\addspace}%
  \printfield{volume}\setunit{\addspace}%
  \printtext[parens]{\printfield{year}}\setunit{\addcomma\space}%
  \printfield{number}\setunit{\addcomma\space}%
  \printfield{pages}\newunit%
  \printfield{note}\newunit%
  \usebibmacro{mr:links}\usebibmacro{finentry}}
\DeclareBibliographyDriver{incollection}{%
  \usebibmacro{bibindex}\usebibmacro{begentry}%
  \printnames{author}\setunit{\addcomma\space}%
  \printfield{title}\setunit{\addcomma\space}%
  \printfield[default]{booktitle}\setunit{\addcomma\space}%
  \printfield{series}\setunit{\addcomma\space}%
  \printfield{volume}\setunit{\addcomma\space}%
  \printlist{publisher}\setunit{\addcomma\space}%
  \printlist{location}\setunit{\addcomma\space}%
  \printfield{year}\setunit{\addcomma\space}%
  \printfield{pages}\newunit%
  \usebibmacro{mr:links}\usebibmacro{finentry}}
\DeclareBibliographyDriver{misc}{%
  \usebibmacro{bibindex}\usebibmacro{begentry}%
  \printnames{author}\setunit{\addcomma\space}%
  \printfield{title}\setunit{\addcomma\space}%
  \printfield{note}\setunit{\addcomma\space}%
  \printfield{year}\newunit%
  \usebibmacro{mr:links}\usebibmacro{finentry}}
\DeclareBibliographyAlias{online}{misc}
\DeclareBibliographyDriver{thesis}{%
  \usebibmacro{bibindex}\usebibmacro{begentry}%
  \printnames{author}\setunit{\addcomma\space}%
  \printfield{title}\setunit{\addcomma\space}%
  \printfield{type}\setunit{\addcomma\space}%
  \printlist{institution}\setunit{\addcomma\space}%
  \printfield{year}\newunit%
  \usebibmacro{mr:links}\usebibmacro{finentry}}

\newtheorem{theorem}{Theorem}[section]
\newtheorem*{theorem*}{Theorem}
\newtheorem{proposition}[theorem]{Proposition}
\newtheorem{lemma}[theorem]{Lemma}
\newtheorem{corollary}[theorem]{Corollary}
\theoremstyle{definition}
\newtheorem{definition}[theorem]{Definition}
\newtheorem{conjecture}[theorem]{Conjecture}

\theoremstyle{remark}
\newtheorem{remark}[theorem]{Remark}
\newtheorem{example}[theorem]{Example}
\newcommand{\C}{\mathbb C}
\newcommand{\Z}{\mathbb Z}
\newcommand{\Q}{\mathbb Q}
\newcommand{\F}{\mathbb F}
\newcommand{\A}{\mathbb A}
\newcommand{\UU}{\mathcal U}
\newcommand{\PP}{\mathbb P}

\DeclareMathOperator{\Res}{Res}
\DeclareMathOperator{\Tr}{Tr}
\newcommand{\op}{\mathrm{op}}
\DeclareMathOperator{\GK}{GKdim}
\DeclareMathOperator{\Spec}{Spec}

\DeclareMathOperator{\End}{End}
\DeclareMathOperator{\Aut}{Aut}
\DeclareMathOperator{\Hom}{Hom}
\DeclareMathOperator{\Mat}{Mat}
\DeclareMathOperator{\Frac}{Frac}
\DeclareMathOperator{\rank}{rank}
\DeclareMathOperator{\gr}{gr}
\DeclareMathOperator{\id}{id}
\DeclareMathOperator{\Pic}{Pic}
\DeclareMathOperator{\Br}{Br}
\DeclareMathOperator{\Ext}{Ext}

\DeclareMathOperator{\Nrd}{Nrd}
\DeclareMathOperator{\ord}{ord}
\setlist[enumerate]{label=\textup{(\roman*)},leftmargin=2.2em,itemsep=3pt}
\allowdisplaybreaks[1]
\title[The generalized $\operatorname{DC}_1$ and $\operatorname{JC}_2$]{The generalized Dixmier conjecture and the two-dimensional Jacobian conjecture}
\author{Alexey Kanel-Belov}
\address{(Alexey Kanel-Belov) Bar--Ilan University, Building 216, 5290002 Ramat-Gan, Israel, MIPT}
\author{Timofey Fedorov}
\address{(Timofey Fedorov) National Research University Higher School of Economics, Moscow, \mbox{Russia}.}
\email{tsfedorov@edu.hse.ru}
\thanks{Corresponding author: Timofey Fedorov.}
\subjclass[2020]{Primary 16S32, 14R15; Secondary 14F10, 16D90, 13P10}
\keywords{Weyl algebra, holonomic $\mathcal D$-modules, arithmetic support, Jacobian conjecture, Dixmier conjecture}
\begin{document}
\begin{abstract}
We prove that the two-dimensional Jacobian conjecture $\mathrm{JC}_2$
implies the generalized Dixmier conjecture $\mathrm{DC}'_1$ for the
first Weyl algebra $A_1$, which states, in one of several equivalent
formulations, that every endomorphism of the two-by-two matrix ring
$\operatorname{Mat}_2(A_1)$ is invertible.
To study the converse, we associate to each endomorphism of
$\mathbb A^2_{\mathbb C}$ with Jacobian one a family of holonomic
$A_2(k)$-modules over finite fields $k$. Kontsevich's conjecture predicts
the existence of a $\mathcal D$-module in characteristic zero,
unique up to isomorphism, whose reductions to positive
characteristic are isomorphic to the prescribed modules.
We prove uniqueness whenever such a lift exists.
If these lifts exist for every such endomorphism, then
$\mathrm{DC}'_1$ implies $\mathrm{JC}_2$.
For every $n\ge2$, we construct a Lagrangian subvariety of
$T^*\A^n$ that provides a counterexample to Kontsevich's conjecture.
We also study an ultraproduct version of the lifting problem:
whether these finite-field modules can be recovered as reductions
of a single module in characteristic zero along a chosen ultrafilter.
We prove that this is possible precisely when the Berest--Wilson
invariants of these modules, regarded as left $A_1(k)$-modules,
are bounded along that ultrafilter.
Finally, we study two other conjectures of Kontsevich on arithmetic
support, giving a counterexample to one and partial results on
the other.
\end{abstract}
\maketitle
\tableofcontents
\clearpage

\section{Introduction}\label{sec:intro}
\subsection*{Notation and conventions}

For a commutative ring $R$, write
\[
 A_n(R)=
 R\langle x_1,\ldots,x_n,y_1,\ldots,y_n\rangle/
 \bigl([y_i,x_j]-\delta_{ij},[x_i,x_j],[y_i,y_j]\bigr)
\]
for the $n$th Weyl algebra over $R$. We abbreviate $A_n(\C)$
to $A_n$. We use the standard symplectic form
$\omega_n=\sum_{i=1}^n dx_i\wedge dy_i$ on $\A_R^{2n}$.

Algebra homomorphisms are unital and linear over the specified
ground field, and Morita equivalences are linear over that field.
Modules are left modules unless a right-module
structure is specified.

For a field $k$, put $W=A_n(k)$ and let $D_W$ be its division
ring of fractions. The rank of a finitely generated left $W$-module
$M$ is
\[
 \operatorname{rk}_W M
   =\dim_{D_W}(D_W\otimes_W M).
\]
For right modules, rank is defined analogously using
$M\otimes_W D_W$.

The Bernstein filtration on $A_n(k)$ assigns degree one to
every generator $x_i,y_i$. The order filtration assigns degree
zero to the $x_i$ and degree one to the $y_i$.

For a polynomial map $f=(f_1,\ldots,f_m)$, put
$\deg f=\max_i\deg f_i$, where the degrees on the right are total degrees.

For a scheme $Y$ over a field $k$ of characteristic $p>0$, write
\[
 Y^{(1)}=Y\times_{\Spec k,F_k}\Spec k
\]
for its Frobenius twist, where $F_k$ is the absolute Frobenius
of $\Spec k$. We denote the relative Frobenius by
$F_Y:Y\to Y^{(1)}$ and the corresponding base change of a
morphism $f$ by $f^{(1)}$.

In positive characteristic, $\mathcal D_X$ denotes the sheaf
of crystalline differential operators; see
Section~\ref{subsec:crystalline}.
\subsection{The classical conjectures and the remaining cases}\label{subsec:intro-classical}
We begin by recalling three classical families of conjectures whose
higher-dimensional cases have now largely been disproved.
\begin{conjecture}\label{conj:classical}
For each positive integer $n$, consider the following assertions:
\begin{enumerate}[leftmargin=5.4em,labelwidth=4.9em,labelsep=.5em]
\item[$\mathrm{JC}_n$:] every endomorphism
$f:\A^n_\C\to\A^n_\C$ with Jacobian one is invertible;
\item[$\mathrm{DC}_n$:] every endomorphism of $A_n$ is an automorphism;
\item[$\mathrm{Symp}_{2n}$:] every morphism
$f:\A^{2n}_\C\to\A^{2n}_\C$ satisfying $f^*\omega_n=\omega_n$ is invertible.
\end{enumerate}
\end{conjecture}

Alp\"oge's counterexample to $\mathrm{JC}_3$ can be written, after
normalizing its Jacobian to one, as follows
\cite{UlamJacobian}. With $h=1+xy$, set
\[
\begin{aligned}
 F_1&=-\tfrac12\bigl(h^3z+y^2h(4+3xy)\bigr),\\
 F_2&=y+3xh^2z+3xy^2(4+3xy),\\
 F_3&=2x-3x^2y-x^3z.
\end{aligned}
\]
The map $F=(F_1,F_2,F_3):\A^3_\C\to\A^3_\C$ has Jacobian
one, but
\[
 F(0,0,-1/4)=F(1,-3/2,13/2)=(1/8,0,0).
\]
Adjoining identity coordinates gives counterexamples to $\mathrm{JC}_n$
for every $n\geq3$. Lifting the coordinate derivations along these maps
also produces counterexamples to $\mathrm{DC}_n$ for every $n\geq3$
\cite{ChanWeyl}.

Long subsequently constructed a noninvertible symplectic endomorphism
$f=(R,T,D,S):\A^4_\C\to\A^4_\C$ \cite{Long}.
In coordinates $(x,q,p,z)$, its components satisfy
\[
 dR\wedge dD+dT\wedge dS=dx\wedge dp+dq\wedge dz,
 \qquad R=x(2-3xq).
\]
The factorization of $R$ already excludes invertibility: a coordinate of
an automorphism generates a prime ideal. The full formulas are
recalled in Section~\ref{subsec:long-map}. This example, extended by
identity maps on additional canonical pairs, disproves
$\mathrm{Symp}_{2n}$ for every $n\geq2$.
Interestingly, no counterexamples to $\mathrm{DC}_2$ have been announced yet.
Related developments in positive characteristic include the work of
Borisov, Gabber, and Vasiu \cite[Theorem~28.6]{BGV}. In every positive
characteristic, they construct noninvertible \'etale endomorphisms of the
affine plane whose generic degrees are prime to the characteristic.
These are counterexamples to a positive-characteristic analogue of the
Jacobian conjecture.

Lang formulates a low-degree Jacobian conjecture: every
endomorphism $f:\A_k^2\to\A_k^2$ over a field $k$ of characteristic
$p>0$ with Jacobian one and $\deg f<p$ is an automorphism
\cite[\S1]{Lang2025}.

Moh \cite{Moh1983} proved that every polynomial endomorphism
$F=(f,g):\A_{\C}^2\to\A_{\C}^2$ with Jacobian one
is an automorphism if
$d:=\max\{\deg f,\deg g\}\le100$.
Guccione, Guccione, Horruitiner, and Valqui
\cite[Theorem~2.1]{GGHV2022} subsequently extended this result
to $d<108$.

The remaining questions considered here concern $\mathrm{JC}_2$,
$\mathrm{DC}_1$, and $\mathrm{DC}_2$; in the plane,
$\mathrm{Symp}_2$ is simply $\mathrm{JC}_2$.
Zheglov's recent article \cite{Zheglov} on $\mathrm{DC}_1$ contains deep insights. We have not been able
to fully understand the proof. We hope that these methods may lead to a proof of $\mathrm{DC}_1'$.

\subsection{The generalized Dixmier conjecture}\label{subsec:intro-generalized}
One may still expect a suitable generalization of $\mathrm{DC}_1$ to
imply $\mathrm{JC}_2$.

For $n=1,2$, let $\mathcal D_n$ be the class of complex domains Morita
equivalent to $A_n$.
\begin{conjecture}[The generalized Dixmier conjecture]\label{conj:generalized}\label{conj:matrix}
For $n=1,2$, consider the following assertions:
\begin{enumerate}[leftmargin=6.4em,labelwidth=5.9em,labelsep=.5em]
\item[$\mathrm{DC}'_n$:] every homomorphism $A_n\to D$, with
$D\in\mathcal D_n$, is an isomorphism;
\item[$\mathrm{DC}''_n$:] every homomorphism $D\to E$, with
$D,E\in\mathcal D_n$, is an isomorphism;
\item[$\mathrm{MC}_{n,r}$:] every endomorphism of $\Mat_r(A_n)$ is an
automorphism, for a fixed integer $r\geq2$;
\item[$\mathrm{MC}_{n,\mathrm{all}}$:] every endomorphism of
$\Mat_r(A_n)$ is an automorphism for every integer $r\geq1$.
\end{enumerate}
\end{conjecture}
The assertion $\mathrm{DC}''_1$ was formulated by Berest, Eshmatov, and
Eshmatov in \cite[Conjecture~1]{BEE}.

We prove the following (Theorem~\ref{thm:matrix-morita-equivalence} below).
\begin{theorem*}
For $n=1,2$ and every fixed $r\geq2$,
\[
 \mathrm{DC}'_n\Longleftrightarrow\mathrm{DC}''_n
 \Longleftrightarrow\mathrm{MC}_{n,r}
 \Longleftrightarrow\mathrm{MC}_{n,\mathrm{all}}.
\]
These equivalences are proved entirely in characteristic zero.
\end{theorem*}
For an endomorphism $f\colon \Mat_r(A_n)\to\Mat_r(A_n)$,
Proposition~\ref{prop:matrix-corner} shows that the corner algebra
\[
f(e_{11})\Mat_r(A_n)f(e_{11})
\]
is the endomorphism ring of a rank-one projective module. The implication $\mathrm{DC}'_n\Rightarrow\mathrm{MC}_{n,r}$ follows. Conversely, Lemma~\ref{lem:morita-rank}, using Stafford's theorem on
projective modules, identifies $\Mat_r(D)$ with $\Mat_r(A_n)$ for
$D\in\mathcal D_n$ and $r\geq2$. Thus a noninvertible homomorphism between
domains in $\mathcal D_n$ would give a noninvertible endomorphism of
$\Mat_2(A_n)$; see Section~\ref{sec:matrix-morita}.

The implications relevant to this paper are
\begin{equation}\label{eq:known}
\begin{gathered}
\begin{tikzcd}[row sep=1.2em,column sep=2.2em]
 \mathrm{DC}_2\arrow[dr,Rightarrow,double distance=1.2pt] & &
 \mathrm{DC}'_2\arrow[ll,Rightarrow,double distance=1.2pt] \\
 & \mathrm{JC}_2\arrow[dr,Rightarrow,double distance=1.2pt] & \\
 \mathrm{DC}_1 & & \mathrm{DC}'_1\arrow[ll,Rightarrow,double distance=1.2pt]
\end{tikzcd}
\end{gathered}
\end{equation}
The two horizontal arrows follow by taking $D=A_n$ in the definition of
$\mathrm{DC}'_n$.

The implication $\mathrm{DC}_2\Rightarrow\mathrm{JC}_2$ has a direct
characteristic-zero proof. Let $f=(f_1,f_2)$ have Jacobian one.
The coordinate vector fields on the target lift along $f$ to commuting
regular vector fields $\delta_1,\delta_2$ on the source, with
$\delta_i(f_j)=\delta_{ij}$. The assignments
\[
 x_i\longmapsto f_i,\qquad y_i\longmapsto\delta_i\quad(i=1,2)
\]
therefore define an endomorphism of $A_2$. This homomorphism preserves the order of differential operators. Hence, if
$\mathrm{DC}_2$ holds, then the map $f^*:\C[x_1,x_2]\to\C[x_1,x_2]$ is
an automorphism and $f$ is an automorphism.

We prove $\mathrm{JC}_2\Rightarrow\mathrm{DC}'_1$ in
Proposition~\ref{prop:matrix-implications} by adapting the reduction-to-positive-characteristic argument of
\cite{T05,BKK}.

The converse $\mathrm{DC}'_1\Rightarrow\mathrm{JC}_2$ is the main lifting
problem studied below. We give two conditional proofs: one assumes the
existence of modules with prescribed graph arithmetic support, and the
other assumes boundedness of the corresponding modules in positive
characteristic.

Stafford \cite[Corollary~3.2]{Stafford1987} reduced the classification
of domains in $\mathcal D_1$ to the classification of
$\Aut_{\C}(A_1)$-orbits on the set of isomorphism classes of nonzero
right ideals of $A_1$. The work of Berest--Wilson and Kouakou
\cite{BWaut,BW,Kouakou} determines these orbits, yielding a
classification of the domains in $\mathcal D_1$ by nonnegative
integers. Their isomorphism classes are distinguished by the
Letzter--Makar-Limanov invariant \cite{LM,Berest}. We use the
abbreviation \emph{LM-invariant} throughout for this algebra invariant.

\subsection{Lifting Lagrangian subvarieties to $\mathcal D$-modules}\label{subsec:intro-lifting}
Let $M$ be a holonomic $\mathcal D_X$-module on a smooth complex algebraic variety $X$. Choose a finitely generated integral $\Z$-algebra $R_0\subset\C$, a smooth $R_0$-scheme $X_{R_0}$, and a coherent $\mathcal D_{X_{R_0}/R_0}$-module $M_{R_0}$, together with isomorphisms $X_{R_0}\times_{R_0}\C\simeq X$ and $M_{R_0}\otimes_{R_0}\C\simeq M$. These data are called an arithmetic model of $(X,M)$ over $R_0$.

At a positive-characteristic point $s\in\Spec R_0$ with residue field $k(s)$, write $M_s=M_{R_0}\otimes_{R_0}k(s)$ and $\mathcal D_{X_s}=\mathcal D_{X_{R_0}/R_0}\otimes_{R_0}k(s)$. The center of this crystalline differential-operator algebra can be canonically identified with the sheaf of functions on $(T^*X_s)^{(1)}$, and the algebra is Azumaya over its center. The closed subscheme defined by the central annihilator of $M_s$ is called its \emph{$p$-support}, denoted by $\operatorname{Supp}_p(M_s)$. Bitoun proves that these supports, with their reduced scheme structures, are Lagrangian for points in a dense open subset of the arithmetic base \cite{Bitoun}; alternative proofs of Lagrangianity were given by Van den Bergh \cite{VdB} and Etingof \cite{Etingof}. The family of these supports, equipped with normalized multiplicities, is called the arithmetic support.

The module $M$ is said to have \emph{constant arithmetic support} if there are a Lagrangian subvariety $L\subset T^*X$ and a finitely generated integral $\Z$-algebra $R'\subset\C$ containing $R_0$, over which $M$ and $L$ are defined, such that
\[
 \operatorname{Supp}_p(M_s)=L_s^{(1)}
\]
as closed subschemes for all points $s$ of positive residue characteristic in a dense open subset of $\Spec R'$. Any such $M$ is called a \emph{lifting $\mathcal D$-module} for $L$. To each irreducible component of $L$, we assign a positive integer
called the \emph{multiplicity} of $M$ along that component.
If $L$ is irreducible, we say that $M$ has multiplicity $r$
when its multiplicity along $L$ is $r$. See
Section~\ref{subsec:crystalline} for the precise definition.

A Lagrangian subvariety $L\subset T^*X$ is \emph{exact} if the restriction
of the tautological one-form to $L$ is $dh$ for a regular function $h$
on $L$. For a symplectic endomorphism $f:\A^{2n}\to\A^{2n}$, we
regard its graph $L_f=\Gamma_f$ as a Lagrangian subvariety in $T^*\A^{2n}$ by
the identification
\[
 (x,y,x',y')\longmapsto(x,x';y,-y')
\]
from the product with symplectic form $\omega_n\oplus(-\omega_n)$.

With these coordinates, the algebra of global differential operators
has the decomposition
\[
 \Gamma(\A^{2n},\mathcal D_{\A^{2n}})
 \simeq A_{2n}\simeq A_n\otimes_{\C}A_n^{\op}.
\]
The generators of the first factor act as
$x_i,\partial_{x_i}$, while those of the opposite factor act as
$x'_i,-\partial_{x'_i}$. Thus the global sections of a
$\mathcal D_{\A^{2n}}$-module form an $A_n$-bimodule:
the first factor gives the left action, and the opposite factor
gives a commuting right action. Given an $A_{2n}$-module $M$, we say that $M$ is finitely generated
(respectively, projective) over the left Weyl-algebra factor if it is
finitely generated (respectively, projective) as a left $A_n$-module
via the first factor in the above decomposition.

Let $f:\A_k^{2n}\to\A_k^{2n}$ be a symplectic endomorphism over a
perfect field $k$ of characteristic $p$. A
$\mathcal D_{\A_k^{2n}}$-module is called a
\emph{splitting $\mathcal D$-module on the graph} if its $p$-support
is $L_f^{(1)}$ and, viewed as a module over the center, it is a
splitting bundle for the restriction of
$\mathcal D_{\A_k^{2n}}$ to $L_f^{(1)}$.
Such a module exists and is unique up to isomorphism when
$\deg f<p$; see Section~\ref{subsec:graph-p}.
For $2\deg f<p$, its restriction to the left Weyl-algebra factor
$A_n(k)$ is isomorphic to the rank-one projective module $W(f)$
constructed by Tsuchimoto \cite[Proposition~2.5]{T11}.
He asks whether $W(f)$ is always free as a left $A_n(k)$-module
\cite[Problem~2.6]{T11}.

We record the original conjecture to specify precisely the statement
disproved below.
\begin{conjecture}[Kontsevich, {\cite[Conjecture~5, \S4.4]{Kont}}]\label{conj:kont}
Let $X$ be a smooth complex variety and let $L\subset T^*X$ be a smooth closed
connected Lagrangian subvariety with $H_1(L(\C),\Z)=0$. Then $L$ admits a unique
lifting $\mathcal D_X$-module $M_L$ of multiplicity one. Moreover, $\Ext^1(M_L,M_L)=0$.
\end{conjecture}

\begin{theorem}\label{thm:counterexamples}
Conjecture~\ref{conj:kont} is false.
\begin{enumerate}
\item Let $f:\A^4\to\A^4$ be Long's symplectic endomorphism. The graph
$L_f=\Gamma_f$ does not admit a lifting $\mathcal D_{\A^4}$-module
of multiplicity one.

Moreover, for every $r\geq1$, $f$ does not lift to an endomorphism
$u:\Mat_r(A_2)\to\Mat_r(A_2)$. More precisely, there does not exist an endomorphism $u$ with the following property: for some
integral finitely generated $\Z[1/30]$-subalgebra $R\subset\C$ over which $u$ has a model
\[
 u_R:\Mat_r(A_2(R))\longrightarrow\Mat_r(A_2(R)),
\]
the equality
\begin{equation}\label{eq:intro-matrix-lift}
 u_s\circ\iota_s=\iota_s\circ(f_s^{(1)})^*
 \qquad(s\in\Sigma)
\end{equation}
holds on a Zariski-dense set $\Sigma$ of closed points of $\Spec R$.
Here $f_s$ is the reduction of $f$, $p_s=\operatorname{char}k(s)$, and
\[
 \iota_s:k(s)[Z_1,Z_2,Z_3,Z_4]\xrightarrow{\sim}
 Z\bigl(\Mat_r(A_2(k(s)))\bigr),\qquad
 Z_j\longmapsto z_j^{p_s}I_r,
\]
where $(z_1,z_2,z_3,z_4)$ are the Weyl algebra generators corresponding to
$(x,q,p,z)$, $I_r$ is the identity matrix, and the polynomial ring on the left is the coordinate ring
of $(\A^4_{k(s)})^{(1)}$. Thus \eqref{eq:intro-matrix-lift} specifies the
entire action of the reduction of $u$ on the center: it must equal the
pullback by the Frobenius twist of $f_s$.
\item For every integer $m\geq2$, there is a smooth closed exact Lagrangian
$L_m\simeq\A^m\subset T^*\A^m$ that does not admit a lifting
$\mathcal D_{\A^m}$-module of multiplicity one. One may take
\[
 L_m=L_2\times T^*_{\{0\}}\A^{m-2}
 \subset T^*\A^2\times T^*\A^{m-2},
\]
where $L_2$ is the image of
\begin{equation}\label{as:emb}
 j:\A^2_{v,z}\longrightarrow T^*\A^2_{v,w},\qquad
 (v,z)\longmapsto(v,2z+v^2z^2;3v^2z^4,2v^3z^3),
\end{equation}
and $T^*_{\{0\}}\A^{m-2}$ is the fiber of the cotangent bundle over the origin.

\end{enumerate}
\end{theorem}

The proof of Theorem~\ref{thm:counterexamples}, given in
Section~\ref{sec:nonliftability}, reduces both counterexamples to the same
obstruction: two scalar ordinary differential operators of orders two and
three with coefficients in the function field of a curve cannot have a nonzero constant
commutator (Lemma~\ref{lem:orders-two-three}). We prove this by showing that
such a pair would determine a rational differential form on a smooth projective
curve whose residues have nonzero sum, contrary to the residue theorem.
For Long's graph, a lifting $\mathcal D$-module would be a rank-one projective left
$A_2$-module with a commuting right $A_2$-action. An equation of its
arithmetic support gives a locally nilpotent operator; the finite Taylor
expansion of Lemma~\ref{lem:taylor-expansion} then reduces the right action
to the forbidden pair of scalar differential operators. For the
Lagrangian affine plane \eqref{as:emb}, the same obstruction occurs. Finally,
products with fibers of cotangent bundles give the
higher-dimensional examples.

Related work on algebraic Lagrangians includes
Dodd's recent paper \cite{DoddMonodromy} and his earlier
preprint \cite{Dodd}. Theorem~1.7 and Corollary~1.9 of
version~4 of the latter assert the existence of holonomic
$\mathcal D$-modules with constant arithmetic support of
multiplicity one under hypotheses that are satisfied by
our examples $L_m\simeq\A^m$.
Thus our examples show that additional hypotheses are
needed for the lifting statement in \cite{Dodd}.
This raises the question of which additional conditions
on an embedding $L\hookrightarrow T^*X$ would ensure the
existence of such a lift. Our examples show that good
geometric properties of $L$ itself are not sufficient. The following question remains relevant.

\begin{conjecture}\label{conj:plane-graphs}
For every endomorphism $f:\A^2_\C\to\A^2_\C$ with Jacobian one,
the graph $L_f=\Gamma_f\subset T^*\A^2$ admits a lifting
$\mathcal D_{\A^2}$-module of multiplicity one.
\end{conjecture}
Proposition~\ref{prop:graph-uniqueness} proves that such a lift, if it exists,
is unique. Its proof uses bounded-degree isomorphisms of rank-one ideals.

\begin{theorem}\label{thm:plane-equivalence}
\begin{enumerate}
\item $\mathrm{JC}_2$ implies $\mathrm{DC}'_1$.
\item Assuming Conjecture~\ref{conj:plane-graphs}, $\mathrm{DC}'_1$
implies $\mathrm{JC}_2$.
\end{enumerate}
\end{theorem}

We prove this result in Sections~\ref{subsec:conditional-plane} and \ref{subsec:matrix}.

A different perspective on the relation between symplectic
geometry, quantization, and $\mathcal D$-modules is provided by
Gukov and Witten \cite{GukovWitten}. Their approach uses branes
to construct representations of quantized algebras and relates
certain branes to $\mathcal D$-modules.

\subsection{LM-invariants and BW-invariants}\label{subsec:intro-invariants}
For splitting $\mathcal D$-modules on graphs associated with Jacobian-one endomorphisms of the plane, there is a second approach to lifting, based on ultraproducts and uniform bounds on presentations. Let $f=(f_1,f_2)$ be an endomorphism of $\A^2_k$ with Jacobian one, where $k$ is perfect of characteristic $p>\operatorname{deg}f$. Put $Z=k[X,Y]=Z(A_1(k))$, where $X=x^p$ and $Y=y^p$, and let $Z'$ be the center of a second copy $A'_1(k)$. If $f^{[p]}$ denotes coefficientwise Frobenius, define
\begin{equation}\label{eq:pullback-definition}
 B_{f,k}=Z\otimes_{Z',\,(f^{(1)})^*}A'_1(k),\qquad
 (f^{(1)})^*(X',Y')=(f_1^{[p]}(X,Y),f_2^{[p]}(X,Y)).
\end{equation}
The unique splitting $\mathcal D$-module on the graph $M_{f,k}$ is a finitely generated rank-one projective left $A_1(k)$-module and a generator for the category of left $A_1(k)$-modules (Proposition~\ref{prop:graphsplit}), and
\begin{equation}\label{eq:pullback-end}
 B_{f,k}\simeq\End_{A_1(k)}(M_{f,k})^{\op}.
\end{equation}
Thus $B_{f,k}$ is a domain Morita equivalent to $A_1(k)$.

For rank-one modules we use the invariant introduced in \cite{BW} (see also \cite{KouakouBW}), which we call the \emph{Berest--Wilson invariant}, or \emph{BW-invariant} for short. The definitions of the BW- and LM-invariants use the Bernstein filtration and apply as follows over any field. Realize a finitely generated rank-one torsion-free $A_1(k)$-module $P$ as a left ideal $I$, and put $S=k[\xi,\eta]=\gr A_1(k)$. With $g_I$ the homogeneous greatest common divisor of $\gr I$, set
\begin{equation}\label{eq:nu-definition}
 J_P=g_I^{-1}\gr I,\qquad \nu(P)=\dim_k(S/J_P).
\end{equation}
For projective $P$, give $E=\End_{A_1(k)}(P)^{\op}$ the degree filtration induced from the division ring of $A_1(k)$ and set
\begin{equation}\label{eq:lambda-definition}
 \lambda(E)=\dim_k(S/\gr E).
\end{equation}
The second quotient is a quotient of graded vector spaces. These definitions are independent of the ideal realization, and
\begin{equation}\label{eq:lambda-nu}
 \lambda(E)=2\nu(P).
\end{equation}
Over $\C$, these are the BW- and LM-invariants of \cite{BW,LM,Berest}. We use the same definitions in positive characteristic. Section~\ref{subsec:invariants} proves that they are well defined and, following the Hilbert-polynomial argument of \cite{Berest}, establishes \eqref{eq:lambda-nu} in this setting. For $B_{f,k}$ the specified center coordinates determine the filtration intrinsically by
\begin{equation}\label{eq:norm-filtration}
 F_tB_{f,k}=\{0\}\cup
 \{b\ne0:\deg_{X,Y}\Nrd(b)\leq t\}.
\end{equation}
Here $\Nrd$ is the reduced norm of the central simple algebra
$B_{f,k}\otimes_Z\Frac Z$ over $\Frac Z$, restricted to $B_{f,k}$;
its values lie in the specified center $Z=k[X,Y]$.

Let $f$ be an endomorphism of $\A^2_\C$ with Jacobian one, defined over a finitely generated integral $\Z$-algebra $R_0\subset\C$. Assume Conjecture~\ref{conj:plane-graphs}, and let $M$ be the corresponding lifting $\mathcal D$-module for $L_f$. After enlarging and localizing $R_0$ if necessary, its fibers $M_s$ are isomorphic to the splitting $\mathcal D$-modules on graphs $M_{f,s}:=M_{f_s,k(s)}$, where $f_s$ is the reduction of $f$, and satisfy $\nu(M_s)=\nu(M)$ at all positive-characteristic points of a dense open subset of $\Spec R_0$. Conversely, a uniform bound on the BW-invariants $\nu(M_{f,s})$ suffices to lift these splitting $\mathcal D$-modules on graphs to characteristic zero along a faithful arithmetic family, as specified below.

\subsection{Lifting via ultraproducts and boundedness}\label{subsec:intro-realizations}
For every $i\geq1$, fix a map $s_i:R_0\to k_i$ to a finite field such that $p_i=\operatorname{char}k_i\to\infty$. Fix a nonprincipal ultrafilter $\UU$ for which the induced map
\(R_0\to K=\prod_\UU k_i\)
is injective. Such a choice will be called a \emph{faithful arithmetic family}. After discarding finitely many indices, assume that
$p_i>\deg f$. Write $f_{k_i}=f\otimes_{s_i}k_i$ and
$M_i=M_{f_{k_i},k_i}$. For $I\in\UU$ and a family $(\alpha_i)_{i\in I}$ with $\alpha_i\in k_i$,
we write $[\alpha_i]_{\UU}$ for the corresponding element of the ultraproduct $K$.

\begin{definition}\label{def:realization}
An arithmetic realization of this family consists of an integral finitely generated $R_0$-subalgebra $R\subset K$, maps $t_i:R\to k_i$ extending $s_i$ for $\UU$-almost all $i$, and a finitely generated $A_2(R)$-module $M_R$ such that
\[
 [t_i(r)]_{\mathcal U}=r\quad\forall r\in R,\qquad
 M_R\otimes_{R,t_i}k_i\simeq M_i
 \quad\text{for $\UU$-almost all $i$},
\]
and $M_R\otimes_R\Frac R$ is holonomic.
\end{definition}

For a scheme $T$ of finite type over $\Z$, put $X_T=\A^2_T$ and write $\mathcal D_{X_T/T}$ for the sheaf of relative crystalline differential operators (see Section~\ref{subsec:crystalline}). For a coherent $\mathcal D$-module $N$ on an affine space over a field, $\dim N$ denotes the dimension of its characteristic variety. A \emph{family of $\mathcal D$-modules of dimension two over $T$} will mean a coherent $\mathcal D_{X_T/T}$-module $M_T$ that is flat over $T$ and satisfies $\dim M_{T,\bar t}=2$ for every geometric point $\bar t\to T$. Its characteristic-zero fibers are holonomic. Here, for a morphism $u:\Spec\ell\to T$, the notation $M_{T,u}$ means pullback along $X_\ell\to X_T$.

\begin{theorem}\label{thm:equivalence}
For a faithful arithmetic family, the following are equivalent:
\begin{enumerate}
\item $\nu(M_i)$ is bounded on a set belonging to $\UU$;
\item $\lambda(B_{f_{k_i},k_i})$ is bounded on a set belonging to $\UU$;
\item the family has an arithmetic realization in the sense of Definition~\ref{def:realization};
\item there exist a scheme $T$ of finite type over $\Spec R_0$ and a family $M_T$ of $\mathcal D$-modules of dimension two over $T$ such that, for $\UU$-almost every $i$, there is an $R_0$-morphism $u_i:\Spec k_i\to T$ with
\[
 M_{T,u_i}\simeq M_i
\]
as $\mathcal D_{\A^2_{k_i}/k_i}$-modules.
\end{enumerate}
Under these equivalent conditions, the characteristic-zero fiber of an arithmetic realization can be chosen to be a projective rank-one left $A_1(\Frac R)$-module with an $A_1(\Frac R)$-bimodule structure. After an embedding of $\Frac R$ into $\C$ extending $R_0\subset\C$, it gives a holonomic $\mathcal D_{\A^2_\C}$-module.
\end{theorem}

\begin{conjecture}[Boundedness of BW-invariants]\label{conj:bound}
For each $d\geq1$ there is an integer $N(d)\geq0$ such that
\[
 \nu(M_{f,k})\leq N(d)
\]
for every finite field $k$ of characteristic $p>d$ and every
endomorphism $f:\A^2_k\to\A^2_k$ with Jacobian one and degree at most $d$.

Equivalently, for every $d\geq1$ there exist a scheme $T_d$ of finite type over $\Z$ and a family $M_{T_d}$ of $\mathcal D$-modules of dimension two over $T_d$ such that, for every such $k$ and $f$, there is a morphism $u:\Spec k\to T_d$ with
\[
 M_{T_d,u}\simeq M_{f,k}
\]
as $\mathcal D_{\A^2_k/k}$-modules. The scheme $T_d$ and the family $M_{T_d}$ depend only on $d$.
\end{conjecture}
The equivalence of these two formulations is proved in Proposition~\ref{prop:geometric-boundedness}.
\begin{corollary}\label{cor:ultra}
Conjecture~\ref{conj:bound} implies $\mathrm{DC}'_1\Rightarrow\mathrm{JC}_2$.
\end{corollary}
This gives the second conditional proof (see Section~\ref{subsec:ultra-implication}). The BW-invariants of rank-one projective modules over a fixed field of positive characteristic can be arbitrarily large; see Example~\ref{ex:fixed-field-BW}.

Section~\ref{sec:curve-supports} concerns $\mathcal D$-modules on curves.
Kontsevich's conjecture is true when the base is a smooth affine
curve (Proposition~\ref{prop:affine-curve-kontsevich}). We give
examples of curves in $T^*\mathbb A^1$ that admit lifting $\mathcal D$-modules
and establish necessary geometric conditions for a curve to admit such a module.
For a smooth closed curve $C\subset T^*\A^1$, the restriction
$\theta=(y\,dx)|_C$ of the Liouville form extends to a rational
differential form on the smooth projective completion of $C$.
Theorem~\ref{as:thm:res} proves that every residue of this differential form
vanishes if an arithmetic model of a nonzero holonomic module has
reduced $p$-support $C_s^{(1)}$ at every closed point of a dense open
subset of its arithmetic base.
In particular, the hyperbola $xy=c$, $c\ne0$, does not admit a lifting $\mathcal D$-module: its Liouville form is $c\,dx/x$, with residues
$c$ and $-c$ at the two points at infinity
(Corollary~\ref{prop:hyperbola-obstruction}). In contrast,
Proposition~\ref{prop:arbitrary-genus} constructs smooth affine
hyperelliptic curves of every positive genus that admit lifting $\mathcal D$-modules.

We also study an explicit closed embedding of
$\PP^1\setminus\{0,1\}$ into $T^*\PP^1$. For this curve,
Theorem~\ref{as:thm:two} classifies the holonomic modules whose
normalized $p$-support cycles agree with the prescribed curve at a
Zariski-dense set of closed arithmetic points. Their isomorphism
classes are parametrized by $\C/\Z$, and exactly the classes in
$\Q/\Z$ correspond to lifting $\mathcal D$-modules for this curve. This distinguishes
agreement on a dense set of specializations from agreement at every
closed point of a nonempty arithmetic open subset.

We also study holonomic modules with no nontrivial self-extensions.
Theorem~\ref{thm:rigid-semisimple} proves that every holonomic
$A_1$-module $M$ with $\Ext^1_{A_1}(M,M)=0$ is semisimple and describes its
simple summands by Fourier transforms and exponential twists.
Corollary~\ref{cor:rigid-arithmetic} proves that such modules belong
to the extended motivic-exponential class, recalled at the beginning
of Appendix~\ref{app:two-conjectures}.
This establishes Conjecture~\ref{conj:kont-six} in this case.
For each such module, the corollary also proves constancy of
arithmetic support. Its distinct components are pairwise disjoint
copies of $\A^1$, which a single symplectic automorphism of
$\A^2$ sends to parallel lines.
Proposition~\ref{thm:rigid-other-curves} proves
Conjecture~\ref{conj:kont-six} for arbitrary holonomic modules with
vanishing self-$\Ext^1$ on every smooth connected complex curve other
than $\mathbb P^1$. If the base curve is neither $\A^1$ nor
$\mathbb P^1$, such a module is supported at finitely many points.
On $\mathbb P^1$, Theorem~\ref{thm:rigid-projective-line} shows that an
irreducible module of positive generic rank with vanishing self-$\Ext^1$
has at most one singular point and quasi-unipotent formal monodromy.
Combining this with Sabbah's theorem on rigid irregular connections
proves Conjecture~\ref{conj:kont-six} for irreducible modules on
$\mathbb P^1$ (Corollary~\ref{cor:kont-six-projective-line}).

Finally, Theorem~\ref{thm:kont-seven-counterexample} gives a differential
operator $Q$ on $\A^2$ for which the $p$-support cycles of
$A_2/A_2(Q+\lambda)$ have multiplicity $p$ for generic $\lambda \in \C$, while the
centralizer of $Q$ is $\C[Q]$ on every nonempty open subset of $\A^2$.
This disproves Conjecture~\ref{conj:kont-seven} about quantum integrable systems.

Section~\ref{sec:matrix-morita} identifies the matrix and Morita
formulations of the generalized Dixmier conjecture.
Section~\ref{sec:arithmetic} develops the arithmetic-support approach. It proves that a
lifting $\mathcal D$-module of multiplicity one for the graph of a
symplectic endomorphism, if it exists, is projective of rank one over
the left Weyl-algebra factor and is unique up to isomorphism.
Section~\ref{sec:conditional} proves $\mathrm{JC}_2\Rightarrow\mathrm{DC}'_1$ and the conditional converse $\mathrm{DC}'_1\Rightarrow\mathrm{JC}_2$.
Section~\ref{sec:ultra} establishes the boundedness criterion for
arithmetic realizations. Section~\ref{sec:nonliftability} constructs
the counterexamples to Kontsevich's conjecture. Section~\ref{sec:curve-supports} treats
arithmetic supports of $\mathcal D$-modules on curves. Appendix~\ref{app:two-conjectures}
addresses two other conjectures of Kontsevich.

\textbf{Acknowledgements.}
We thank Yuri Berest for helpful e-mail correspondence and for telling us about the generalized Dixmier conjecture. We also
thank Alexander Zheglov and Dmitry Kaledin for useful discussions.

The first author's work was supported by RSF № 25-21-00884 and the Gelbart Institute.

The second author's work was funded within the framework of the
HSE University Basic Research Program (HSE-BR-2025-060).

\section{Equivalent formulations of the generalized Dixmier conjecture}\label{sec:matrix-morita}
In this section, all algebras and Morita equivalences are over $\C$. We prove that the formulations of the generalized Dixmier conjecture in Conjecture~\ref{conj:generalized} are equivalent, without reduction to positive characteristic. Throughout, $A=A_n$, and $Q=\Frac A$.

\begin{proposition}\label{prop:matrix-corner}
Let $u:\Mat_r(A)\to\Mat_r(A)$ be a homomorphism, where $r\geq1$. Set $R=\Mat_r(A)$, $u_{ij}=u(e_{ij})$, and $e=u_{11}$. Then
\[
 B=eRe
\]
is a domain Morita equivalent to $A$. More precisely, $P=eA^r$ is a rank-one projective right $A$-module and a generator for the category of right $A$-modules, $P^{\oplus r}\simeq A^r$, and $B\simeq\End_A(P)$. If
\[
 g:A\longrightarrow B,\qquad g(a)=u(ae_{11}),
\]
then $u$ is an automorphism if and only if $g$ is an isomorphism.
\end{proposition}
\begin{proof}
The $u_{ij}$ form a complete system of matrix units. Inside $Q^r$, the subspaces $u_{ii}Q^r$ form a direct sum and are mutually isomorphic via the $u_{ij}$. Their common dimension is therefore one. Consequently
\[
 B\subset e\Mat_r(Q)e\simeq\End_Q(eQ^r)\simeq Q,
\]
so $B$ is a domain. Also
\[
 1=\sum_i u_{i1}e u_{1i},\qquad
 A^r=\bigoplus_i u_{ii}A^r\simeq P^{\oplus r}.
\]
These identities show that $e$ is full and $P$ is a rank-one projective right $A$-module generating the category of right $A$-modules. Restriction to $P$, with extension by zero on $(1-e)A^r$, identifies $B$ with $\End_A(P)$.

The isomorphism
\[
 \Psi:R\xrightarrow{\sim}\Mat_r(B),\qquad
 \Psi(b)=(u_{1i}b u_{j1})_{i,j},
\]
has inverse $(b_{ij})\mapsto\sum_{i,j}u_{i1}b_{ij}u_{1j}$. The matrix-unit identities give
\[
 \Psi\circ u=\Mat_r(g).
\]
An entrywise matrix map is an isomorphism exactly when the original map is an isomorphism. This proves the last assertion.
\end{proof}

\begin{lemma}\label{lem:morita-rank}
Every $D\in\mathcal D_n$ is isomorphic to $\End_A(P)$ for a rank-one projective right $A$-module $P$ generating the category of right $A$-modules. For every $r\geq2$,
\begin{equation}\label{eq:stable-matrix}
 \Mat_r(D)\simeq\Mat_r(A).
\end{equation}
\end{lemma}
\begin{proof}
Morita theory gives a finitely generated projective generator $P$ for the category of right $A$-modules with $D\simeq\End_A(P)$. Since $P$ is a direct summand of a finite free module, it admits a nonzero right-linear map $v:P\to A$. If $0\ne q\in\ker v$, the endomorphism $m\mapsto qv(m)$ is nonzero, by torsion-freeness, and its square is zero. This contradicts the domain hypothesis. Thus $v$ embeds $P$ into $A$, and $P$ has rank one. The identical argument with sides exchanged applies to left modules.

Stafford's theorem on projective modules \cite[Theorem~3.6]{Stafford} states that a projective $A$-module of rank at least two is free. Hence $P^{\oplus r}\simeq A^r$ for $r\geq2$. Taking endomorphism rings gives
\[
 \Mat_r(D)\simeq\End_A(P^{\oplus r})
 \simeq\End_A(A^r)\simeq\Mat_r(A).
\]
\end{proof}

\begin{theorem}\label{thm:matrix-morita-equivalence}
For $n=1,2$ and every fixed $r\geq2$, the assertions of Conjecture~\ref{conj:generalized} satisfy
\[
 \mathrm{DC}'_n\Longleftrightarrow\mathrm{DC}''_n
 \Longleftrightarrow\mathrm{MC}_{n,r}
 \Longleftrightarrow\mathrm{MC}_{n,\mathrm{all}}.
\]
\end{theorem}
\begin{proof}
Assume $\mathrm{DC}'_n$, and let $u$ be an endomorphism of $\Mat_m(A)$ for any $m\geq1$. Proposition~\ref{prop:matrix-corner} constructs a domain $B\in\mathcal D_n$ and a homomorphism $g:A\to B$. The hypothesis makes $g$ an isomorphism, so the same proposition makes $u$ an automorphism. This proves $\mathrm{DC}'_n\Rightarrow\mathrm{MC}_{n,\mathrm{all}}$, which implies $\mathrm{MC}_{n,r}$.

Now assume $\mathrm{MC}_{n,r}$ for one fixed $r\geq2$. For $h:D\to E$ with $D,E\in\mathcal D_n$, use Lemma~\ref{lem:morita-rank} to identify both $\Mat_r(D)$ and $\Mat_r(E)$ with $\Mat_r(A)$. The entrywise map $\Mat_r(h)$ becomes an endomorphism of $\Mat_r(A)$ and is therefore an isomorphism. Thus $h$ is an isomorphism, proving $\mathrm{DC}''_n$. Taking $D=A$ gives $\mathrm{DC}'_n$.
\end{proof}
\section{Arithmetic support and lifting}\label{sec:arithmetic}\label{sec:graph-modules}

\subsection{Crystalline differential operators}\label{subsec:crystalline}
We recall the framework of \cite{Bitoun,Kont}. For a smooth morphism $X\to S$ of relative dimension $m$, with $S$ noetherian, the sheaf $\mathcal D_{X/S}$ is the enveloping algebra of the relative tangent Lie algebroid. Its formation commutes with base change on $S$. In positive characteristic this is the sheaf of \emph{crystalline} differential operators, rather than the full sheaf of Grothendieck differential operators.

The order filtration has associated graded algebra $\operatorname{Sym}_{\mathcal O_X}\mathcal T_{X/S}$, the algebra of functions on $T^*(X/S)$. Over a field, a good filtration on a coherent $\mathcal D$-module $N$ therefore defines its characteristic variety as the support of the associated graded module in the cotangent bundle. Its dimension is independent of the good filtration. On affine space it can also be computed from a good Bernstein filtration: it is the dimension of the support of the associated graded module, or equivalently the degree of the cumulative filtered Hilbert polynomial. In positive characteristic it also equals the dimension of the $p$-support described below; compare \cite[\S\S4.1--4.2]{Bitoun}. These are the dimension conventions used for the families in Section~\ref{subsec:intro-realizations}.

For the arithmetic setting, let $S=\Spec R$ be integral of finite type over $\Z$, with a nonempty characteristic-zero fiber.

Let $s\in S$ be a point with perfect residue field $k=k(s)$ of characteristic $p$, and write $X_s=X\times_S\Spec k$. The assignments
\[
 a\longmapsto a^p,\qquad
 \partial\longmapsto\partial^p-\partial^{[p]}
\]
identify the center of $\mathcal D_{X_s}$ with $\mathcal O_{T^*X_s^{(1)}}$. Here $\partial^{[p]}$ is the $p$th iterate as a derivation. This represents $\mathcal D_{X_s}$ as an Azumaya algebra on $T^*X_s^{(1)}$ of degree $p^m$. In affine coordinates,
\[
 Z(A_m(k))=k[x_1^p,\ldots,x_m^p,y_1^p,\ldots,y_m^p],
 \qquad\rank_Z A_m(k)=p^{2m}.
\]
These facts follow locally from the action of the central fibers on
$k[t_1,\ldots,t_m]/(t_1^p,\ldots,t_m^p)$; see \cite[Theorem~2.2.3]{BMR}.

Let $M_R$ be a coherent $\mathcal D_{X/S}$-module, and regard its fiber $M_s=M_R\otimes_R k(s)$ as a coherent module over the center on $Y_s=(T^*X_s)^{(1)}$. We always consider $p$-support as a \emph{scheme-theoretic} support:
\[
 \operatorname{Supp}_p(M_s)
 =V\bigl(\operatorname{Ann}_{\mathcal O_{Y_s}}M_s\bigr)\subset Y_s.
\]
The annihilator is an ideal sheaf; it need not be radical. We write $\operatorname{Supp}_p(M_s)_{\mathrm{red}}$ when only the associated reduced subscheme is intended. For an irreducible component with generic point $\eta$, its normalized multiplicity is
\begin{equation}\label{eq:multiplicity}
 p^{-m}\operatorname{length}_{\mathcal O_{Y_s,\eta}}(M_s)_\eta.
\end{equation}
The length is taken in the ambient central local ring, so it detects transverse nilpotent structure at $\eta$.

If the characteristic-zero fiber is holonomic, Bitoun's theorems provide a dense open subset of $S$ on which the reduced $p$-supports are Lagrangian, and the Azumaya algebra splits over their regular loci \cite[Theorems~2.2.1 and 5.1.4]{Bitoun}. Alternative proofs of the Lagrangianity assertion are given in \cite{VdB,Etingof}. Bitoun's Theorem~4.3.3 also bounds their degrees and generic ranks in terms of the Bernstein multiplicity of the original module. These results concern reductions of a fixed module; they do not construct a module with a prescribed support.

In the constant-support condition of Conjecture~\ref{conj:kont}, the equality
\[
 \operatorname{Supp}_p(M_s)=L_s^{(1)}
\]
is an equality of closed subschemes, after choosing a common arithmetic model of $X,L,M$ and restricting the arithmetic base to a dense open subset. Multiplicity one refers to \eqref{eq:multiplicity}.

\begin{remark}\label{rem:purity}
For reductions of a holonomic characteristic-zero module, multiplicity one on every irreducible component implies that the scheme-theoretic $p$-support is reduced, after restricting to a dense open subset of the arithmetic base.

Indeed, the proof of \cite[Theorem~3.1.1 and Remark~3.3.6]{Bitoun} shows that, after this restriction, $M_s$ is pure of dimension $m$ as a coherent $\mathcal O_{Y_s}$-module: every nonzero coherent subsheaf has support of dimension $m$. Let $\mathcal I$ be the ideal of $\operatorname{Supp}_p(M_s)_{\mathrm{red}}$. At a generic point of any component, split the Azumaya algebra by an \'etale extension. Morita equivalence then writes the module as $p^m$ copies of a finite-length module. By \eqref{eq:multiplicity}, that module has length one, so the maximal ideal annihilates it. Faithfully flat descent therefore gives $(\mathcal I M_s)_\eta=0$ at each generic point $\eta$ of the support. Consequently $\mathcal I M_s$, if nonzero, would be a coherent subsheaf of dimension less than $m$. Purity forces $\mathcal I M_s=0$. Since $\mathcal I$ is the radical of the central annihilator, we obtain
\[
 \operatorname{Ann}_{\mathcal O_{Y_s}}M_s=\mathcal I.
\]
Without purity, generic multiplicity alone would not exclude nilpotents.
\end{remark}

\subsection{Splitting bundles on graphs}\label{subsec:graph-p}
Throughout this subsection, $k$ is a perfect field of
characteristic $p>0$, and $X$ is a smooth $k$-scheme.

Let $\mathcal A$ be an Azumaya algebra of degree $q$ on a scheme $S$.
A \emph{splitting bundle} is a locally free $\mathcal O_S$-module $E$ of
rank $q$ with an action inducing an isomorphism
$\mathcal A\simeq\mathcal End_{\mathcal O_S}(E)$.
It exists if and only if $[\mathcal A]=0$ in $\Br(S)$. When one exists,
all splitting bundles are $E\otimes\mathcal L$, for line bundles $\mathcal L$;
indeed $\mathcal Hom_{\mathcal A}(E,E')$ is a line bundle and evaluation
identifies $E'$ with its tensor product with $E$.
Their isomorphism classes therefore form a torsor under $\Pic(S)$.
Let \(S\subset (T^*X)^{(1)}\) be a reduced closed subscheme. A splitting bundle for \(\mathcal D_X|_S\) has rank \(p^{\dim X}\) over \(\mathcal O_S\) and, regarded as a \(\mathcal D_X\)-module, has scheme-theoretic \(p\)-support \(S\) with normalized multiplicity one on every irreducible component.

Let $Y$ be a smooth $k$-scheme and let $F_Y:Y\to Y^{(1)}$ be the
relative Frobenius. For a vector bundle $E$ with an integrable
connection $\nabla$, its \emph{$p$-curvature} is defined by
\[
 \psi_\nabla(D)=\nabla_D^p-\nabla_{D^{[p]}}
       \in\mathcal End_{\mathcal O_Y}(E),
\]
where $D$ is a local $k$-derivation and $D^{[p]}$ is its $p$-fold
iterate as a derivation. This map is additive in $D$ and satisfies
$\psi_\nabla(fD)=f^p\psi_\nabla(D)$. Equivalently, it is a section of
$\mathcal End(E)\otimes_{\mathcal O_Y}F_Y^*\Omega^1_{Y^{(1)}/k}$,
where $F_Y^*$ denotes pullback of the cotangent bundle.

If $E$ is a line bundle, the values of $\psi_\nabla$ are scalar.
They are horizontal for the induced connection on
$\mathcal End(E)=\mathcal O_Y$, so their coefficients have zero
differential and descend through Frobenius; see
\cite[\S5, Proposition~5.2]{KatzNilpotent}. Thus they determine a unique form
\[
 \eta_\nabla\in H^0(Y^{(1)},\Omega^1_{Y^{(1)}/k}),
\]
which we also call the $p$-curvature of $\nabla$. In \'etale coordinates
$x_i$, this convention means
\[
 \psi_\nabla(\partial_{x_i})
   =F_Y^*\bigl(\eta_\nabla(\partial_{x_i^{(1)}})\bigr)\id_E.
\]
Here $F_Y^*$ on the coefficient denotes pullback of regular functions.
The $p$-curvature exact sequence defines a functorial map
\[
 \phi_Y:H^0(Y^{(1)},\Omega^1_{Y^{(1)}/k})\longrightarrow\Br(Y^{(1)})[p],
\]
and $\phi_Y(\eta_\nabla)=0$ for every line bundle with integrable
connection. For a morphism $j:Y\to T^*X$
between smooth $k$-schemes, it gives
\begin{equation}\label{eq:liouville-brauer}
 (j^{(1)})^*[\mathcal D_X]
       =\phi_Y\bigl((j^*\theta_X)^{(1)}\bigr),
\end{equation}
where $\theta_X$ is the tautological one-form. These are
\cite[Definition~5.2.4, Lemma~5.2.6 and Proposition~5.2.9]{Bitoun}.
If $j^*\theta_X=dg$ for $g\in\Gamma(Y,\mathcal O_Y)$, the connection
$d+dg$ has $p$-curvature $(dg)^{(1)}$, so \eqref{eq:liouville-brauer}
vanishes. Thus the crystalline algebra splits on the Frobenius twist of
a smooth exact Lagrangian.

Let $f:\A_k^{2n}\to\A_k^{2n}$ satisfy $f^*\omega_n=\omega_n$,
put $d=\deg f$, and assume $d<p$. Write
$z=(x_1,\ldots,x_n,y_1,\ldots,y_n)$ and
$\lambda=\sum_i y_i\,dx_i$. Under the identification
\[
 (x,y,x',y')\longmapsto(x,x';y,-y')\in T^*\A_k^{2n},
\]
the pullback of the tautological form to $\Gamma_f$ is
$\eta=\lambda-f^*\lambda$. It is closed, since $d\lambda=-\omega_n$,
and its coefficients have degree at most $2d-1$.

Write $\eta=\sum_i a_i\,dz_i$, put $\partial_i=\partial/\partial z_i$,
and set $b_i=\partial_i^{p-1}(a_i)$. Closedness gives
\[
 \partial_j b_i
 =\partial_i^{p-1}\partial_j a_i
 =\partial_i^p a_j=0
 \qquad\text{for all }i,j.
\]
Thus $b_i\in k[z_1^p,\ldots,z_{2n}^p]$. Since
$\deg b_i\leq 2d-p<p$, each $b_i$ is constant.
Choose $c_i\in k$ with $c_i^p=-b_i$. The connection
\[
 \nabla=d+\eta+\sum_i c_i\,dz_i
\]
on the trivial line bundle is integrable. Since $\partial_i^{[p]}=0$,
\[
 \psi_\nabla(\partial_i)=\nabla_{\partial_i}^p
 =(a_i+c_i)^p+\partial_i^{p-1}(a_i+c_i)=a_i^p.
\]
These equalities identify the one-form $\eta_\nabla$ on
$(\A_k^{2n})^{(1)}$ defined above with $\eta^{(1)}$.

Put $X_k=\A_k^{2n}$. Equation~\eqref{eq:liouville-brauer} shows
that the Azumaya algebra $\mathcal D_{X_k}$ on $(T^*X_k)^{(1)}$
splits on $\Gamma_f^{(1)}$. Since
$\Gamma_f^{(1)}\simeq\A_k^{2n}$ has trivial Picard group, its
splitting module is unique up to isomorphism; denote it by
$M_{f,k}$.

Put \(W_k=A_n(k)\), and let \(W'_k\) be a second copy of \(W_k\), with primed generators. Write
$Z_k=Z(W_k)$ and $Z'_k=Z(W'_k)$. The identification above gives
\[
 W_k\otimes_k(W'_k)^{\op}\simeq\Gamma(X_k,\mathcal D_{X_k}),
\]
where the left generators $x_i,y_i$ act as $x_i,\partial_{x_i}$,
and the right generators $x'_i,y'_i$ act as
$x'_i,-\partial_{x'_i}$. Define
\[
 B_{f,k}=Z_k\otimes_{Z'_k,\,(f^{(1)})^*}W'_k,\qquad
 I_{f,k}=\ker\bigl(Z_k\otimes_kZ'_k\longrightarrow Z_k\bigr),
\]
where the map defining $I_{f,k}$ sends $a\otimes b$ to
$a(f^{(1)})^*(b)$. Thus $I_{f,k}$ is the ideal of $\Gamma_f^{(1)}$
in the center of $W_k\otimes_k(W'_k)^{\op}$.

\begin{proposition}\label{prop:graphsplit}
The module $M_{f,k}$ is a rank-one projective left $W_k$-module and a generator for the category of left $W_k$-modules, and
\[
 \End_{W_k}(M_{f,k})^{\op}\simeq B_{f,k}.
\]
Its annihilator in $Z_k\otimes_kZ'_k$ is $I_{f,k}$, and its rank over $Z_k$ is $p^{2n}$. The pullback algebra $B_{f, k}$ is a domain. If $f$ is invertible, $M_{f,k}$ is an invertible $W_k$-bimodule.
\end{proposition}
\begin{proof}
Put $S=\Spec Z_k$. After a faithfully flat base change on $S$ and a further local trivialization, both Azumaya algebras are matrix algebras of size $p^n$, and their splitting bimodule is the space of $p^n$-by-$p^n$ matrices with left and right multiplication. It is the regular module on either side. Projectivity, the generator property, and the endomorphism identity descend. The splitting bundle is faithful over $\mathcal O_S$, so its central annihilator in the original tensor product is exactly $I_{f,k}$. Its central rank is $p^{2n}$; dividing by the central rank of $W_k$ gives rank one. Endomorphisms of a rank-one torsion-free module embed in the opposite of the division ring, so the endomorphism algebra, and hence its opposite $B_{f,k}$, is a domain. If $f$ is invertible, the pullback of $W'_k$ along its center is isomorphic to $W'_k$ as a $k$-algebra, and the displayed Morita equivalence is an autoequivalence.
\end{proof}

For $n=1$ this proves \eqref{eq:pullback-end}.

\subsection{Good filtrations and a finiteness criterion}\label{subsec:restriction}
Fix $n\geq1$. We first work with an arbitrary finitely generated bimodule, without imposing a condition on its arithmetic support. Let $R$ be an integral finitely generated $\Z$-algebra of characteristic zero, put $F=\Frac R$, and write
\[
 W_R=A_n(R),\qquad W_R^{\mathrm e}=W_R\otimes_R W_R^{\op}.
\]
The two actions of $R$ on a $W_R$-bimodule are required to coincide, so that such a bimodule is a left $W_R^{\mathrm e}$-module. Fix a finitely generated $W_R^{\mathrm e}$-module $M_R$ and put
\[
 M_F=M_R\otimes_R F,\qquad M_s=M_R\otimes_R k(s)
 \quad(s\in\Spec R).
\]
When $R\subset\C$, an arithmetic model of a complex bimodule $M$ consists of such an $M_R$ together with an isomorphism $M_R\otimes_R\C\simeq M$.

For this general bimodule setting, let
\[
 (z_1,\ldots,z_{2n})=(x_1,\ldots,x_n,y_1,\ldots,y_n)
\]
be the standard generators of $W_R$, and define elements of $W_R^{\mathrm e}$ by
\[
 z_{j,L}=z_j\otimes1,\qquad z_{j,R}=1\otimes z_j^{\op}.
\]
Their actions are $z_{j,L}m=z_jm$ and $z_{j,R}m=mz_j$. Here the subscript $R$ on $z_{j,R}$ indicates the right action, rather than the coefficient ring. For multi-indices $a,b\in\mathbb Z_{\geq0}^{2n}$, write $z_L^a z_R^b$ for the product in this fixed order. These products form an $R$-basis of $W_R^{\mathrm e}$.

Fix positive integral weights for the left and right generators. The filtered part of $W_R^{\mathrm e}$ of degree at most $\ell$ is the $R$-span of the ordered monomials of weighted degree at most $\ell$. Choose module generators $v_1,\ldots,v_c$ and let $F_\ell M_R$ be the image of this filtered part applied to their $R$-span. This is a good filtration: its associated graded module is finitely generated over the associated graded polynomial algebra. One may also assign fixed integral degrees to the module generators.

After replacing $R$ by $R[r^{-1}]$ for one nonzero $r\in R$, every $F_\ell M_R$ and every graded piece may be assumed finite free over $R$, with
\begin{equation}\label{eq:filtered-base-change}
 (F_\ell M_R)\otimes_R k(s)\simeq F_\ell M_s,\qquad
 (\gr M_R)\otimes_R k(s)\simeq\gr M_s
\end{equation}
for every $s\in\Spec R$. Here $M_s$ has the filtration induced by the chosen generators and weights. To obtain one localization valid in every degree, start with a finite presentation, homogenize for these weights, and compute a finite homogeneous Gr\"obner basis over $F$. Invert the denominators in that basis, in its expressions in the original relations, and in the expressions of the original relations in the basis. Invert also the leading coefficients and the denominators in the reductions of the critical pairs to zero. The resulting monic Gr\"obner basis works over the localized ring and in every residue field. The monomials not divisible by its leading monomials give compatible bases of the filtered parts and graded pieces. In particular, their dimensions over $k(s)$ agree with their dimensions over $F$.

The same construction applies to finitely generated left or right $W_R$-modules. It also shows that a finitely generated module of any of these three types whose fibers vanish on a Zariski-dense set of closed points has zero generic fiber: a nonzero generic filtered part would, after the indicated localization, remain nonzero in every fiber.

\begin{lemma}[Descent of maps and identities]\label{lem:finite-descent}
Let $R$ be an integral domain, $F=\Frac R$, and $A_R$ an
$R$-algebra in which $R$ acts centrally. Let $P_1,\ldots,P_m$ be
finitely presented left $A_R$-modules. Any finite collection of
$A_F$-linear maps between the $(P_i)_F$, together with finitely many
elements and identities involving these maps, is defined over
$R[r^{-1}]$ for some $0\ne r\in R$. This applies in particular to
sections of surjections, inverse isomorphisms, and commuting actions
of a finitely presented algebra. Finite presentations and homomorphisms
of finitely presented algebras over $F$ likewise descend after one
localization, together with finitely many specified identities.
\end{lemma}
\begin{proof}
Specify each map by the images of the generators in a finite
presentation of its source. Clear the finitely many denominators in
these images. The images of the defining relations vanish over $F$,
so each is annihilated by a nonzero element of $R$; invert their product.
The maps are now well defined. Test each remaining identity on the
finite generating sets and invert the finitely many nonzero elements
of $R$ that annihilate its differences. For algebra maps and additional
actions, do the same with the images of algebra generators and the
defining algebra relations. Presentations defined initially over $F$
are obtained over a localization by clearing their coefficient
denominators. All these localizations can be combined into one.
\end{proof}

\begin{lemma}\label{lem:generic-detection}
Let $R$ be an integral finitely generated $\Z$-algebra of characteristic
zero and put $F=\Frac R$. Let $m\geq1$, let $U$ be a principal affine
open in $\A^m_R$, and let $V$ be a
finitely generated $\mathcal D_{U/R}$-module. If $V_s=0$ on a
Zariski-dense set of points of $\Spec R$, then $V_F=0$.
Consequently, if a regular function $g$ acts nilpotently on the
fibers of a finitely generated $\mathcal D_{\A^m_R/R}$-module
over a Zariski-dense set of points of $\Spec R$, then $g$ acts
locally nilpotently on its generic fiber. The nilpotence
exponents on the fibers need not be bounded.
\end{lemma}
\begin{proof}
Choose a good order filtration on $V$, bounded below. Generic freeness
\cite{Stacks} makes $\gr V$ free over a localization of $R$.
Each graded part is then projective over $R$. The filtration parts
and the quotients $V/F_jV$ are flat: they are built by extensions and
filtered unions of these projective graded parts. Thus formation of the
associated graded module commutes with residue-field base change. If
$V_F\ne0$, the resulting nonzero free module $\gr V$ has nonzero fibers
at every point of this open subset, contradicting density.

For the last assertion apply this result to the localization
$V[g^{-1}]$ on the principal open $D(g)$. Its generic fiber is zero.
Localization over the coordinate ring then says that each vector of $V_F$
is annihilated by some power of $g$.
\end{proof}

\begin{lemma}[Finiteness criterion]\label{lem:restriction}
Let $M_R$ be a finitely generated $W_R^{\mathrm e}$-module, with $R$, $F$, and the left and right generators as above.
\begin{enumerate}
\item Give $z_{j,L}$ weight one and $z_{j,R}$ weight $d\geq1$, place the module generators $v_1,\ldots,v_c$ in degree zero, and localize $R$ so that \eqref{eq:filtered-base-change} holds. Suppose that for some point $s\in\Spec R$ of residue characteristic $p>0$ there are polynomials
\[
 h_j\in k(s)[T_1,\ldots,T_{2n}],\qquad \deg h_j\leq d
 \quad(1\leq j\leq2n),
\]
and an integer $e\geq1$ such that
\[
 \bigl(z_{j,R}^p-h_j(z_{1,L}^p,\ldots,z_{2n,L}^p)\bigr)^eM_s=0
 \quad(1\leq j\leq2n).
\]
Then $M_F$ is finitely generated as a left $W_F$-module. The same conclusion holds if these annihilation identities are satisfied after an extension of the residue field, with the $h_j$ defined over that extension.

\item If $M_F$ is finitely generated as a left $W_F$-module, then there exists $0\ne r\in R$ such that $M_{R[r^{-1}]}$ has a finite presentation as a left $W_{R[r^{-1}]}$-module. Base change of this presentation gives a finite left presentation of $M_s$ for every $s\in\Spec R[r^{-1}]$, and also after any extension of its residue field.
\end{enumerate}
Both assertions have right-module versions, obtained by exchanging the left and right actions.
\end{lemma}
\begin{proof}
\emph{(i)} Put $q=pe$. Each displayed relation is monic of degree $q$ in one right generator. Its other terms have strictly smaller degree in that right generator, involve only central left coefficients and powers of the same right generator, and have weighted degree at most $dq$. Successively applying these relations therefore reduces all right exponents to values less than $q$ without increasing weighted degree. It follows that $F_\ell M_s$ is spanned by
\[
 z_L^a z_R^b v_i,\qquad
 1\leq i\leq c,\quad 0\leq b_j<q,\quad |a|+d|b|\leq\ell.
\]
For each fixed $\ell$, take the map from a finite free $R$-module to $F_\ell M_R$ whose images are these same monomials in the chosen generators. It is surjective after base change to $k(s)$. Since the target is finite free, a maximal minor is nonzero modulo $s$ unless the target is zero. That minor is nonzero over $F$, so the map is surjective over $F$. Apply this argument separately in each degree; no further localization depending on $\ell$ is required. The fixed finite set
\[
 \{z_R^b v_i:1\leq i\leq c,\ 0\leq b_j<q\}
\]
therefore generates $M_F$ over $W_F$. If the annihilation identities are known only after a residue-field extension, the same monomials span after that extension. Surjectivity descends along a field extension, and the preceding argument applies unchanged.

\emph{(ii)} Assume that $M_F$ is finitely generated as a left $W_F$-module. Choose left generators $u_1,\ldots,u_b$ of $M_F$ and, after clearing their denominators, regard them as elements of a localization of $M_R$. Keep the bimodule generators $v_1,\ldots,v_c$ fixed above. Over $F$, write
\begin{equation}\label{eq:left-generation-identities}
 v_i=\sum_a c_{ia}u_a,\qquad
 u_a z_j=\sum_b d_{ajb}u_b,\qquad c_{ia},d_{ajb}\in W_F.
\end{equation}
Apply Lemma~\ref{lem:finite-descent} to the chosen elements and
the generator and right-action identities
\eqref{eq:left-generation-identities}; they hold over one localization
of $R$. The left span of the $u_a$ contains the $v_i$ and is stable under every right generator, so it equals $M_R$. Since $W_R$ is noetherian, $M_R$ then has a finite left presentation, which gives a finite left presentation after every residue-field base change.

Exchanging the left and right actions proves the right-module versions.
\end{proof}

\begin{remark}\label{rem:graph-finiteness}
Let $R$ and $M_R$ be as in Lemma~\ref{lem:restriction}. Suppose that $f:\A_R^{2n}\to\A_R^{2n}$ is an endomorphism satisfying $f^*\omega_n=\omega_n$, that $\Sigma\subset\Spec R$ is a Zariski-dense set of closed points of positive residue characteristic, and that
\[
 \operatorname{Supp}_p(M_s)_{\mathrm{red}}\subset L_{f_s}^{(1)}
 \qquad(s\in\Sigma).
\]
This includes the situation in which $M_R$ is an arithmetic model of a characteristic-zero $\mathcal D$-module with constant arithmetic support equal to the graph. Set $d\geq\max(1,\deg f)$ and make the filtration localization required by Lemma~\ref{lem:restriction}(i). The remaining open subset meets $\Sigma$; choose $s$ in this intersection and put $p=\operatorname{char}k(s)$. Since $M_s$ is finite over the center of $W_{k(s)}^{\mathrm e}$, the support inclusion says that the graph ideal $I_{f_s,k(s)}$ lies in the radical of its central annihilator. Some power of this ideal consequently annihilates $M_s$. In particular, the equations in Lemma~\ref{lem:restriction}(i) hold for some $e$, with $h_j=f_{s,j}^{[p]}$. Thus $M_F$ is finitely generated as a left $W_F$-module. If $M_s$ is a splitting $\mathcal D$-module on the graph, one can take $e=1$.

Lemma~\ref{lem:restriction}(ii) now gives a finite left presentation of $M_R$ after a further localization of $R$. Its base changes give finite left presentations on all fibers over the remaining open subset of $\Spec R$. If $f_F$ is an automorphism, spread out its inverse after another localization. The reversed graph is then the graph of $f^{-1}$, so the same argument with the left and right actions exchanged also gives finite right presentations.
\end{remark}

\subsection{Projectivity of $A_n$-bimodules that are finite as left $A_n$-modules}\label{subsec:projectivity}
In this subsection, $n\geq1$, the field $k$ has characteristic zero,
and $W=A_n(k)$.

\begin{lemma}\label{lem:projectivity}
Let $N$ be a $W$-bimodule finitely generated as a left $W$-module. Then $N$ is projective and torsion-free on the left. If its left rank is one, it is a projective generator for the category of left $W$-modules, is simple and holonomic as a module over $W\otimes W^{\op}$, and $\End_W(N)^{\op}$ is a domain Morita equivalent to $W$.
\end{lemma}
\begin{proof}
We may assume $N\ne0$. Let $F_\ell W$ be the Bernstein filtration, and choose left generators $n_1,\ldots,n_c$. They are also bimodule generators. Define
\[
 L_\ell N=\sum_i(F_\ell W)n_i,\qquad
 B_\ell N=\sum_{\substack{a+b\leq\ell\\1\leq i\leq c}}
                (F_aW)n_i(F_bW).
\]
These are good filtrations for the left $W$-module and the left $W\otimes W^{\op}$-module, respectively. For every standard Weyl algebra generator $z_j$, write
\[
 n_i z_j=\sum_h a_{ijh}n_h,\qquad a_{ijh}\in W,
\]
and choose $D\geq1$ with $\deg a_{ijh}\leq D$. Induction on the length of a right word gives $n_i(F_bW)\subset L_{Db}N$. Therefore
\[
 (F_aW)n_i(F_bW)\subset L_{a+Db}N\subset L_{D\ell}N
 \qquad(a+b\leq\ell).
\]
Conversely, $L_\ell N\subset B_\ell N$ because the right factor may be $1\in F_0W$. Thus
\begin{equation}\label{eq:filtration-comparison}
 L_\ell N\subseteq B_\ell N\subseteq L_{D\ell}N.
\end{equation}
Write $\GK$ for Gelfand--Kirillov dimension. A fixed linear rescaling of the filtration index does not change this growth exponent. Hence
\begin{equation}\label{eq:growth}
 \GK_{W\otimes W^{\op}}N=\GK_WN.
\end{equation}

A left $W$-module is torsion if its tensor product with $D_W$ is zero. A finitely generated torsion left module has Gelfand--Kirillov dimension at most $2n-1$: each of its generators is annihilated by a nonzero operator, so it is a quotient of a finite direct sum of modules $W/Wa$, $a\ne0$, each of dimension at most $2n-1$. The torsion submodule $T$ of $N$ is stable under the right action, since a nonzero operator annihilating $v$ also annihilates $vb$. It is finitely generated on the left because $W$ is noetherian. If $T\ne0$, the filtration comparison just proved, applied to $T$ with its own left generators, gives
\[
 \GK_{W\otimes W^{\op}}T=\GK_WT\leq2n-1.
\]
This contradicts Bernstein's inequality for $W\otimes W^{\op}\simeq A_{2n}(k)$. Thus $N$ is torsion-free.

For $j>0$, set $E^j=\Ext_W^j(N,W)$. Choose a finite resolution of $N$ by finitely generated projective left $W$-modules; this exists because $W$ is noetherian of finite global dimension. Right multiplication on $N$, acting contravariantly on Ext, defines a left $W$-action on $E^j$; right multiplication on the target defines a commuting right action. The module $E^j$ is finitely generated on the right. Localization of the resolution gives
\[
 E^j\otimes_WD_W
 \simeq\Ext_{D_W}^j(D_W\otimes_WN,D_W)=0.
\]
Thus $E^j$ is right torsion. The preceding argument, with left and right exchanged, shows that $E^j=0$.

These vanishings imply projectivity of $N$. Indeed, in a finite projective resolution of positive length, the vanishing of the highest Ext group makes the last differential of the dual complex surjective. It splits because its target is projective. Dualizing back, using biduality for finitely generated projective modules, splits off the last term of the original resolution. Induction shortens the resolution to length zero.

Suppose now that the left rank of $N$ is one. Every nonzero bimodule submodule $N'\subset N$ has left rank one, and $N/N'$ has rank zero. The quotient is finitely generated on the left, so the torsion argument above forces it to be zero. Hence $N$ is simple as a bimodule. A finite torsion-free rank-one left module has Gelfand--Kirillov dimension $2n$, so \eqref{eq:growth} proves holonomicity over $A_{2n}(k)$.

Finally, the trace ideal
\[
 \operatorname{tr}_W(N)=\sum_{\lambda\in\Hom_W(N,W)}\lambda(N)
\]
is a nonzero two-sided ideal of $W$: projectivity makes $N$ a direct summand of a finite free module, whose coordinate maps separate its elements. Simplicity of $W$ gives $\operatorname{tr}_W(N)=W$. Equivalently, finitely many $\lambda_a\in\Hom_W(N,W)$ and $u_a\in N$ satisfy $\sum_a\lambda_a(u_a)=1$. The induced map from a finite direct sum of copies of $N$ onto $W$ shows that $N$ is a generator. Extension of endomorphisms to $D_W\otimes_WN$ is injective and identifies $\End_W(N)$ with a subring of $D_W^{\op}$. It is therefore a domain, and Morita theory gives the asserted equivalence.
\end{proof}

\subsection{Projectivity of modules with graph arithmetic support}\label{subsec:recognition}
\begin{proposition}\label{prop:recognition}
Let $n\geq1$, let $R$ be an integral finitely generated $\Z$-algebra of characteristic zero, let $F=\Frac R$, and let $W_R=A_n(R)$. Let
\[
 f:\A_R^{2n}\longrightarrow\A_R^{2n},\qquad f^*\omega_n=\omega_n,
\]
and let $M_R$ be a finitely generated $W_R\otimes_R W_R^{\op}$-module. Suppose there is a Zariski-dense set $\Sigma$ of closed points of $\Spec R$ such that, for every $s\in\Sigma$,
\[
 \operatorname{Supp}_p(M_s)=L_{f_s}^{(1)}
\]
as closed subschemes, and the normalized multiplicity in \eqref{eq:multiplicity}, with $m=2n$, is one.

Then $M_F=M_R\otimes_R F$ is a rank-one projective left $A_n(F)$-module, generates the category of left $A_n(F)$-modules, and is holonomic over $A_{2n}(F)$. There is a nonzero $r\in R$ such that, for every $s\in\Sigma\cap\Spec R[r^{-1}]$, the fiber $M_s$ is isomorphic to the splitting $\mathcal D$-module on the graph $M_{f_s,k(s)}$ of Proposition~\ref{prop:graphsplit}. If $f_F$ is an automorphism, then $M_F$ is also finite projective of rank one on the right.
\end{proposition}
\begin{proof}
We first obtain a finite left presentation and compute rank one from the
arithmetic multiplicity. We then spread projectivity and the generator
property, and identify the resulting fibers as splitting modules.

\emph{Finite generation.}
Choose $d\geq\max(1,\deg f)$. The first application of Lemma~\ref{lem:restriction} described in Section~\ref{subsec:restriction} shows that $M_F$ is finitely generated on the left. The argument using \eqref{eq:left-generation-identities} then provides one localization of $R$ over which $M_R$ has a finite left presentation. Choose a Bernstein filtration for the left module and localize again so that \eqref{eq:filtered-base-change} holds for this filtration. The set of remaining points of $\Sigma$ is still Zariski dense.

\emph{Rank one.}
Fix such a point $s$, and write $p=\operatorname{char}k(s)$ and
\[
 Z_s=k(s)[z_{1,L}^p,\ldots,z_{2n,L}^p],\qquad K_s=\Frac Z_s.
\]
Projection of the graph $L_{f_s}^{(1)}$ onto the left central affine space is an isomorphism. Thus its generic point has residue field $K_s$. Since its ideal annihilates $M_s$, formula \eqref{eq:multiplicity} gives
\[
 \dim_{K_s}(K_s\otimes_{Z_s}M_s)=p^{2n}.
\]
Each generator of $Z_s$ has Bernstein degree $p$. Consequently the leading asymptotic coefficient of the cumulative left filtered Hilbert function of $M_s$ is
\[
 \frac{p^{2n}}{p^{2n}(2n)!}=\frac1{(2n)!}.
\]
For a finite module over a polynomial ring with $2n$ variables of weight $p$, this coefficient is its generic rank divided by $p^{2n}(2n)!$. Terms supported on proper closed subsets have smaller degree. The left Bernstein filtration is also a good filtration over $Z_s$, since $W_{k(s)}$ is finite over $Z_s$ with its filtered basis of ordered monomials whose individual exponents are less than $p$.

By \eqref{eq:filtered-base-change}, the generic fiber has the same filtered Hilbert polynomial. For a finitely generated $W_F$-module $N$, the leading coefficient
in degree $2n$ is its rank divided by $(2n)!$. Indeed, choose
$m_1,\ldots,m_r\in N$ such that $1\otimes m_1,\ldots,1\otimes m_r$
form a basis of $D_{W_F}\otimes_{W_F}N$ over $D_{W_F}$.
The map $W_F^r\to N$ sending the $i$th standard basis vector
to $m_i$ is injective with torsion cokernel, whose growth has
degree at most $2n-1$. Hence $M_F$ has rank one. Lemma~\ref{lem:projectivity} proves projectivity, the generator property, and holonomicity.

\emph{Spreading projectivity.}
Choose a finite surjection
\[
 \epsilon_R:W_R^b\longrightarrow M_R.
\]
Projectivity gives a left-linear section $j_F:M_F\to W_F^b$.
Lemma~\ref{lem:finite-descent}, applied to this map and the identity
$\epsilon_Fj_F=\id_{M_F}$, gives after localization
\begin{equation}\label{eq:projective-splitting}
 j_R:M_R\longrightarrow W_R^b,\qquad \epsilon_Rj_R=\id_{M_R}.
\end{equation}
The generator property is obtained in the same way from finitely many left-linear maps and elements satisfying the trace identity
\begin{equation}\label{eq:trace-identity}
 \lambda_a:M_R\longrightarrow W_R,\qquad u_a\in M_R,
 \qquad\sum_a\lambda_a(u_a)=1.
\end{equation}
The corresponding maps and elements exist over $F$ by
Lemma~\ref{lem:projectivity}. Apply Lemma~\ref{lem:finite-descent}
also to these maps, elements, and the trace identity. After localization, \eqref{eq:projective-splitting} and \eqref{eq:trace-identity} survive every residue-field base change. Thus every remaining $M_s$ is projective and generates the category of left $W_{k(s)}$-modules. Its rank is one, by the same filtered Hilbert polynomial calculation.

\emph{Identifying the fibers.}
Take $s\in\Sigma$ in this open subset and omit the finitely many characteristics with $p\leq\deg f$. As a module over $Z_s$, the projective $W_{k(s)}$-module $M_s$ is finite locally free of rank $p^{2n}$. The ideal of the graph annihilates $M_s$, so it gives an action of the Azumaya algebra obtained by restricting the algebra of crystalline differential operators to the graph:
\[
 \mathcal A_s=
 W_{k(s)}\otimes_{Z_s}\bigl((f_s^{(1)})^*W'_{k(s)}\bigr)^{\op},
\]
whose degree over $Z_s$ is $p^{2n}$. The resulting homomorphism
\[
 \mathcal A_s\longrightarrow\End_{Z_s}(M_s)
\]
is an isomorphism: on each geometric fiber it is a unital homomorphism between matrix algebras of the same size, and hence is an isomorphism. Therefore $M_s$ is a splitting module, and uniqueness in Proposition~\ref{prop:graphsplit} identifies it with $M_{f_s,k(s)}$.

\emph{The right action.}
If $f_F$ is an automorphism, its inverse is defined over $R[r^{-1}]$ for some nonzero $r$. The two composition identities remain valid there. The reversed graphs are then the graphs of the reductions of $f_F^{-1}$. Apply the same argument with the left and right actions exchanged to obtain finite generation and rank one on the right; Lemma~\ref{lem:projectivity}, with sides reversed, gives right projectivity.
\end{proof}

\begin{corollary}[Arithmetic realizations]\label{cor:arithmetic-recognition}
Let $M_R$ be an arithmetic realization of a faithful arithmetic family
as in Definition~\ref{def:realization}, and put $F=\Frac R$.
Then $M_F$ is a rank-one projective left $A_1(F)$-module and a generator for the category of left $A_1(F)$-modules.
After localizing $R$, the module $M_R$ has a finite left presentation
whose base changes are the prescribed modules $M_i$ for
$\UU$-almost all $i$. If $f_F$ is an automorphism, $M_R$ also has a
finite right presentation after a further localization, with the same
specialization property.
\end{corollary}
\begin{proof}
Choose $J\in\UU$ on which the maps $t_i:R\to k_i$ and the prescribed
fiber isomorphisms are defined. The points $\sigma_i=\ker t_i$ are
closed because $k_i$ is finite. The graph-support and multiplicity-one
conditions descend from $k_i$ to $k(\sigma_i)$ by faithful flatness:
central annihilators commute with this field extension, and the generic
rank over the graph is unchanged. Moreover,
\[
 J\cap\{i:t_i(r)\ne0\}\in\UU\qquad(0\ne r\in R),
\]
so the points $\sigma_i$ are Zariski dense and every fixed localization
preserves the prescribed fibers for $\UU$-almost all $i$.
Proposition~\ref{prop:recognition} gives the rank and projectivity
assertions. Lemma~\ref{lem:restriction}(ii) gives the finite left
presentation after localization. If $f_F$ is an automorphism, the same
proposition gives finite right generation, and the right-module version
of that lemma gives the last assertion.
\end{proof}

\subsection{Uniqueness of the lift in characteristic zero}\label{subsec:graph-uniqueness}
The passage from isomorphisms of reductions to an isomorphism in characteristic zero requires control of their degrees.

\begin{lemma}\label{lem:ideal-isomorphism-degree}
Let $k$ be a field of arbitrary characteristic and let $W=A_n(k)$ for $n\geq1$. Suppose that $I,J\subset W$ are nonzero left ideals, each generated by elements of Bernstein degree at most $H$. Every left $W$-module isomorphism $\psi:I\to J$ satisfies
\[
 \deg\psi(a)\leq\deg a+H\qquad(0\ne a\in I).
\]
In particular, the images of the specified generators of $I$ have degree at most $2H$.
\end{lemma}
\begin{proof}
Since $\gr W$ is a polynomial ring over a field, $\deg(ab)=\deg a+\deg b$ for nonzero $a,b\in W$. This degree extends to the division ring $D$ by
\[
 \deg(ab^{-1})=\deg a-\deg b.
\]
Passing to common Ore denominators shows that the definition is independent of the fractional expression and is additive under multiplication. After localization, $I$ and $J$ both become the one-dimensional left $D$-module $D$. Hence $\psi$ is right multiplication by an element $q\in D^\times$, and $Iq=J$. Set
\[
 \mu(I)=\min\{\deg a:0\ne a\in I\},\qquad
 \mu(J)=\min\{\deg b:0\ne b\in J\}.
\]
Degree additivity gives $\mu(J)=\mu(I)+\deg q$. Both minima lie between $0$ and $H$, so $\deg q\leq H$. Therefore $\deg\psi(a)=\deg a+\deg q\leq\deg a+H$.
\end{proof}

\begin{proposition}\label{prop:graph-uniqueness}
Let $R,F,W_R,f$ be as in Proposition~\ref{prop:recognition}, and let $M_R,N_R$ be finitely generated $W_R\otimes_R W_R^{\op}$-modules. Suppose that there is a Zariski-dense set $\Sigma$ of closed points of $\Spec R$ such that
\[
 M_s\simeq M_{f_s,k(s)}\simeq N_s\qquad(s\in\Sigma)
\]
as $A_n(k(s))$-bimodules. Then $M_F\simeq N_F$ as $A_n(F)$-bimodules. After replacing $R$ by a localization, this isomorphism is defined over $R$.

The same conclusion holds if the displayed isomorphisms are given after extensions of the residue fields. In particular, a lifting $\mathcal D_{\A^{2n}_\C}$-module for $L_f$ of multiplicity one, if it exists, is unique up to isomorphism.
\end{proposition}
\begin{proof}
\emph{Ideal realizations and specialization.}
Proposition~\ref{prop:recognition} makes $M_F$ and $N_F$ projective of left rank one. Each can be realized as a nonzero left ideal of $W_F=A_n(F)$: embed the module into its one-dimensional localization over the division ring, and multiply its finitely many generators on the right by a common denominator. After localizing $R$, these realizations give left-module isomorphisms
\[
 M_R\simeq I_R\subset W_R,\qquad
 N_R\simeq J_R\subset W_R.
\]
Indeed, Proposition~\ref{prop:recognition} gives finite left
presentations. Apply Lemma~\ref{lem:finite-descent} to the ideal
realizations and their inverse maps, then transfer the right actions
to $I_R$ and $J_R$.

Choose generators $a_1,\ldots,a_c$ of $I_R$ and generators of $J_R$, and let $H$ bound their Bernstein degrees. The degree-compatible Gr\"obner basis construction of Section~\ref{subsec:restriction}, applied to the ideals and their relation modules, allows a further localization with the following properties. The fibers $I_s,J_s$ are the nonzero left ideals generated by the corresponding reduced operators in $A_n(k(s))$, and
\begin{equation}\label{eq:ideal-truncation-base-change}
 (J_R\cap F_tW_R)\otimes_R k(s)
   \simeq J_s\cap F_tA_n(k(s))
 \qquad(t\geq0).
\end{equation}
Here the intersections are free $R$-modules of finite rank. The corresponding standard monomials give $R$-bases of the quotient modules, whose $R$-flatness ensures that the fiber ideals embed into the fiber Weyl algebras. For each nonstandard ordered monomial $u$, let $\operatorname{NF}(u)$ be its normal form modulo the monic basis of $J_R$. The elements $u-\operatorname{NF}(u)$ with $\deg u\leq t$ form an $R$-basis of $J_R\cap F_tW_R$ and specialize to the analogous basis in every fiber. Inverting a nonzero leading coefficient in each ideal ensures that no fiber ideal is zero. The remaining points of $\Sigma$ are still Zariski dense.

\emph{A finite system of linear equations.}
Put $V_R=J_R\cap F_{2H}W_R$. Choose generators of the left relations among the $a_i$, written
\[
 \sum_i u_{\ell i}a_i=0\qquad(1\leq\ell\leq r),
\]
and express the right action of the $2n$ standard generators $z_j$ of $W_R$ as
\[
 a_i\mathbin{\cdot}z_j=\sum_h v_{ijh}a_h,
 \qquad u_{\ell i},v_{ijh}\in W_R.
\]
The dot denotes the transferred right bimodule action. A tuple $(b_1,\ldots,b_c)\in V_R^c$ determines a bimodule homomorphism $I_R\to J_R$, with $a_i\mapsto b_i$, exactly when
\begin{equation}\label{eq:bounded-bimodule-hom}
 \sum_i u_{\ell i}b_i=0,
 \qquad
 b_i\mathbin{\cdot}z_j-\sum_h v_{ijh}b_h=0
 \quad\text{for every }\ell,i,j.
\end{equation}
The first equations ensure that the left-linear map is well defined; the second ensure that it commutes with the right generators, and hence with the entire right action. These are homogeneous $R$-linear equations in the coefficients of the $b_i$ in a basis of $V_R$. Their left sides belong to a fixed finite filtered part of $W_R$: it suffices to evaluate the finitely many displayed operations on that basis and take the largest resulting degree. Thus \eqref{eq:bounded-bimodule-hom} is the kernel condition for one map
\[
 T_R:V_R^c\longrightarrow(F_DW_R)^{r+2nc}
\]
between finite free $R$-modules, for some integer $D$.

For every remaining $s\in\Sigma$, choose a bimodule isomorphism $I_s\to J_s$. Lemma~\ref{lem:ideal-isomorphism-degree} bounds the degrees of the images of the $a_i$ by $2H$. By \eqref{eq:ideal-truncation-base-change}, these images give a nonzero element of $\ker(T_R\otimes_R k(s))$. Hence every full-column minor of a matrix for $T_R$ vanishes at every such $s$. Zariski density and integrality of $R$ imply that these minors are zero in $R$. Therefore
\[
 \ker(T_R\otimes_R F)\ne0.
\]

\emph{The resulting isomorphism.}
A nonzero vector in this generic kernel defines a nonzero bimodule homomorphism $M_F\to N_F$. Both modules are simple as bimodules by Lemma~\ref{lem:projectivity}, so this homomorphism is an isomorphism. Apply Lemma~\ref{lem:finite-descent} to this map, its inverse,
the bimodule-compatibility identities, and the two composition
identities. This gives the asserted isomorphism after one further
localization of $R$.

If the prescribed fiber isomorphisms hold after residue-field extensions, Proposition~\ref{prop:recognition} applies by descent of its support and rank hypotheses. The linear system above has a nonzero kernel after each such field extension, and therefore already over the residue field. The rest of the proof is unchanged.

Finally, two modules over $\mathbb C$ with the stated constant arithmetic support admit arithmetic models over a common finitely generated integral $\Z$-subalgebra of $\C$. Proposition~\ref{prop:recognition} identifies both families of reductions with the splitting $\mathcal D$-modules on graphs on a common dense open subset. Applying the result just proved and extending scalars to $\C$ gives their isomorphism.
\end{proof}

The proposition also applies to two arithmetic realizations of the same faithful family in Definition~\ref{def:realization}: take a common coefficient ring inside the ultraproduct field and the common representative maps to the residue-field extensions. Faithfulness makes the corresponding points Zariski dense. This uniqueness statement concerns agreement along that family; it does not assert that such a realization has constant arithmetic support at every closed point of a nonempty open subset of the arithmetic base.

\section{The two-dimensional Jacobian conjecture}\label{sec:conditional}

\subsection{Recovering an inverse from arithmetic fibers}\label{subsec:inverse}
For an automorphism of the plane over any field,
\begin{equation}\label{eq:inverse-degree}
 \deg f^{-1}\leq\deg f;
\end{equation}
see \cite{BCW}. This estimate gives the following finite test. For a map of the plane with Jacobian one, translate its value at zero and multiply by the inverse Jacobian matrix at zero to obtain $h(z)=z+$ higher-degree terms. The formal inverse $H$ is determined recursively over the coefficient ring, without division by integers.
Write $H_{\leq d}$ for the truncation of $H$ in total degree
at most $d$.
For a map of degree at most $d$, invertibility is equivalent to the finite polynomial identities
\begin{equation}\label{eq:inverse-identities}
 h\circ H_{\leq d}=\id,\qquad H_{\leq d}\circ h=\id.
\end{equation}
More explicitly, let $f$ be defined over an integral finitely generated $\Z$-algebra $R$ of characteristic zero. If $f_s$ is an automorphism for a Zariski-dense set of closed points $s\in\Spec R$, each coefficient in the differences in \eqref{eq:inverse-identities} vanishes at all these points. It is therefore zero in $R$. The identities give an inverse over $\Frac R$, and over $\C$ when $R\subset\C$. For a faithful arithmetic family the same conclusion follows because a coefficient vanishing for $\UU$-almost all indices has zero image in $K$, and $R\to K$ is injective.

Under $\mathrm{JC}_2$, all maps of the plane of degree at most $d$ and Jacobian one are invertible in sufficiently large characteristics, uniformly for fixed $d$. Indeed, use the finite-type parameter scheme over $\Z$ for these plane maps. The failure of \eqref{eq:inverse-identities} defines a constructible subset. Its characteristic-zero fiber is empty under $\mathrm{JC}_2$, since every finitely generated field of characteristic zero embeds in $\C$. Its image in $\Spec\Z$ is constructible and omits the generic point, hence consists of finitely many primes.

\begin{lemma}[Spreading invertible bimodules]\label{lem:invertible-descent}
Let $R$ be an integral finitely generated $\Z$-algebra of characteristic
zero, put $F=\Frac R$ and $W_R=A_n(R)$, and let $M_R$ be a finitely
generated $W_R\otimes_R W_R^{\op}$-module. If $M_F$ is an invertible
$W_F$-bimodule, then $M_R$ is invertible after localizing $R$.
Every remaining residue-field fiber is then invertible.
\end{lemma}
\begin{proof}
Choose an inverse bimodule $N_F$ and evaluation isomorphisms
\[
 M_F\otimes_{W_F}N_F\simeq W_F,\qquad
 N_F\otimes_{W_F}M_F\simeq W_F.
\]
Both bimodules are finitely generated projective on each side.
Lemma~\ref{lem:restriction}(ii), applied on both sides, gives finite
left and right presentations of $M_R$ after localization. Choose a
finite bimodule presentation of $N_F$ and clear its coefficient
denominators to obtain a model $N_R$; apply the same lemma to make it
finite on both sides. Lemma~\ref{lem:finite-descent} now descends the
evaluation maps, their inverses, and the bimodule-compatibility and
composition identities after one further localization. Thus the
displayed isomorphisms hold over $R$. Residue-field base change
commutes with the tensor products and preserves these isomorphisms.
\end{proof}

\begin{lemma}\label{lem:inverse-graph}
Use the notation $R,F,W_R,f$ of Proposition~\ref{prop:recognition}, with $n=1$. Thus $W_R=A_1(R)$ and $f$ is an endomorphism of the plane with Jacobian one. Let $M_R$ be a finitely generated $W_R\otimes_R W_R^{\op}$-module whose generic fiber $M_F$ is an invertible $W_F$-bimodule. Suppose that, on a Zariski-dense set $\Sigma$ of closed points of $\Spec R$,
\[
 \operatorname{Supp}_p(M_s)_{\mathrm{red}}\subset L_{f_s}^{(1)}.
\]
Then $f_F$ is an automorphism. In particular, the same conclusion holds over $\C$ after extension of scalars from $F\subset\C$.
\end{lemma}
\begin{proof}
By Lemma~\ref{lem:invertible-descent}, after localizing $R$ every
remaining fiber $M_s$ is an invertible bimodule.

An invertible bimodule identifies the centers of its two algebras. Write $\beta_s:Z(W_{k(s)})\to Z(W_{k(s)})$ for the isomorphism characterized by
\[
 mz=\beta_s(z)m\qquad(m\in M_s,\ z\in Z(W_{k(s)})).
\]
Since $M_s$ is faithful on the left, its central annihilator is exactly the ideal generated by $z_R-\beta_s(z)_L$. Its scheme-theoretic support is therefore the reduced graph of this center isomorphism. Containment in $L_{f_s}^{(1)}$ forces $\beta_s=(f_s^{(1)})^*$. Thus $f_s^{(1)}$, and hence $f_s$ over the perfect residue field, is an automorphism for every remaining $s\in\Sigma$. The inverse identities \eqref{eq:inverse-identities} and Zariski density now prove that $f_F$ is an automorphism.
\end{proof}

\subsection{The conditional implication \texorpdfstring{$\operatorname{DC}'_1\Rightarrow\operatorname{JC}_2$}{generalized DC1 implies JC2}}\label{subsec:conditional-plane}
\begin{proof}[Proof of Theorem~\ref{thm:plane-equivalence}(ii)]
Assume $\mathrm{DC}'_1$ and Conjecture~\ref{conj:plane-graphs}, and let
$f:\A^2_\C\to\A^2_\C$ be an endomorphism with Jacobian one. Choose a lifting \(\mathcal D\)-module \(M\) of multiplicity one for \(\Gamma_f\), as provided by Conjecture~\ref{conj:plane-graphs}, and
a common arithmetic model of $f,M$ over an integral finitely generated
$\Z$-subalgebra of $\C$. Proposition~\ref{prop:recognition} shows that $M$
is a rank-one projective left $A_1$-module generating its module category.
The right action gives
\[
 A_1\longrightarrow\End_{A_1}(M)^{\op}\in\mathcal D_1.
\]
By $\mathrm{DC}'_1$ this homomorphism is an isomorphism. Morita evaluation
therefore makes $M$ an invertible bimodule. Lemma~\ref{lem:inverse-graph}
implies that $f$ is an automorphism. This is $\mathrm{JC}_2$.
\end{proof}

\subsection{Proof of the implication \texorpdfstring{$\mathrm{JC}_2\Rightarrow\mathrm{DC}'_1$}{JC2 implies generalized DC1}}\label{subsec:matrix}
The proof follows the reduction argument of \cite{BKK}, applied here
to the two-by-two matrix algebra over the first Weyl algebra.

\begin{proposition}\label{prop:matrix-implications}
The two-dimensional Jacobian conjecture implies $\mathrm{DC}'_1$.
\end{proposition}
\begin{proof}
\emph{Step 1: reduction to a matrix endomorphism.}
By Theorem~\ref{thm:matrix-morita-equivalence}, it suffices to prove
that every endomorphism of $C=\Mat_2(A_1)$ is an automorphism.
Fix such an endomorphism $u$.

\emph{Step 2: centers in positive characteristic.}
Let $k$ be algebraically closed of characteristic $p$, and put
\[
 C_k=\Mat_2(A_1(k)),\qquad Z_k=k[x^p,y^p].
\]
This is an Azumaya algebra of degree $2p$. Every endomorphism of $C_k$
preserves $Z_k$. Indeed, compose the endomorphism with the representation
at a closed central point. The resulting representation has dimension
$2p$, and every finite-dimensional simple $C_k$-module has this
dimension. The representation is therefore simple. Central elements
act by scalars at every closed central point. Their commutators with
all elements vanish, because $C_k$ is free over the reduced Jacobson
ring $Z_k$.

If the center map is invertible, the endomorphism becomes center-linear
after twisting the source by this center automorphism. Each geometric
fiber is then a unital map between matrix algebras of the same size,
hence an isomorphism. Thus the endomorphism is an isomorphism;
compare \cite[Theorem~4.1 and Corollary~5.4]{ACvE}. Both assertions hold over an arbitrary field $k$
by extension to an algebraic closure and faithfully flat descent.

\emph{Step 3: the arithmetic model and its plane center maps.}
Choose an integral finitely generated $\Z$-subalgebra $R_0\subset\C$
over which $u$ is defined. After localization we may suppose that
$R=R_0[a^{-1}]$ is smooth over $\Z$, since its generic point lies in
the smooth locus. Write $C_R=\Mat_2(A_1(R))$.
For a prime $p$, the ring $S=R/pR$ is smooth over $\F_p$ and hence
reduced. Step~2 applied to every residue field shows that the reduction
$u_S$ preserves the center: the coefficients of each commutator
$[u_S(z),b]$, for $z$ central, lie in every prime ideal of $S$ and
therefore vanish.

The center carries the bracket
\[
 \{a,b\}=p^{-1}[\widetilde a,\widetilde b]\pmod p,
\]
where the tildes denote lifts to $C_R$. The map induced by $u$ preserves
this bracket. With the central coordinates $X=x^p$, $Y=y^p$, it satisfies
$\{X,Y\}=1$: the constant term of $[y^p,x^p]$ is $p!$, and all other
terms are divisible by $p^2$; use Wilson's theorem after division by $p$.
Thus every closed-fiber center map has Jacobian one.

Choose $d\geq1$ bounding the degrees of the entries of $u(xI_2)$ and
$u(yI_2)$. Since $u(z^pI_2)=u(zI_2)^p$ for $z=x,y$, the Bernstein
degrees of these central images are at most $pd$. Each central variable
has Bernstein degree $p$, so the plane center maps have degree at most
$d$.

\emph{Step 4: invertibility of the reductions.}
Assume $\mathrm{JC}_2$. The finite inverse test following
\eqref{eq:inverse-identities} shows that all plane maps of degree at
most $d$ and Jacobian one are invertible in sufficiently
large characteristics. After inverting finitely many primes in $R$,
every closed-fiber center map is therefore an automorphism. Its inverse
has degree at most $d$, by \eqref{eq:inverse-degree}. Step~2 makes
every remaining reduction $u_s$ an automorphism of $C_{k(s)}$.

\emph{Step 5: surjectivity in characteristic zero.}
Give $C_R$ the commuting left and right $A_1(R)$-actions
\[
 a\cdot b\cdot c=(aI_2)b\,u(cI_2).
\]
Write $L_a(b)=(aI_2)b$ and $R_a(b)=b\,u(aI_2)$.
The four matrix units form a basis of \(C_R\) as a left \(A_1(R)\)-module. Give the twisted right generators weight one and the
ordinary left generators weight $d$. On each remaining closed fiber,
the inverse center polynomials $g_{s,1},g_{s,2}$ satisfy
\[
 L_{x^p}=g_{s,1}(R_{x^p},R_{y^p}),\qquad
 L_{y^p}=g_{s,2}(R_{x^p},R_{y^p}),\qquad \deg g_{s,j}\leq d.
\]
Lemma~\ref{lem:restriction}, with $n=1$ and sides exchanged, makes
$C_{\Frac R}$ finite under the twisted right action. Its second part
then gives finite right generation over one localization of $R$.
The image $u(C_R)$ is stable under this action, so
\[
 Q_R=C_R/u(C_R)
\]
is a finite right $A_1(R)$-module. Its remaining closed fibers vanish.
The filtration argument in Section~\ref{subsec:restriction} implies
$Q_R\otimes_R\Frac R=0$. Consequently $u$ is surjective over $\C$.
Simplicity of $C$ makes it injective.
\end{proof}

\section{Lifting to characteristic zero via ultraproducts}\label{sec:ultra}

\subsection{BW-invariants and LM-invariants}\label{subsec:invariants}
We justify \eqref{eq:nu-definition}--\eqref{eq:norm-filtration}. Throughout this subsection $k$ is any field, $A=A_1(k)$, $D$ its division ring, and $S=\gr A=k[\xi,\eta]$. Principal symbols extend to $D$ by
\[
 \deg(ab^{-1})=\deg a-\deg b,\qquad
 \sigma(ab^{-1})=\sigma(a)/\sigma(b).
\]
Ore common denominators show that these formulas are well defined, and degree is additive on products.

Every finitely generated torsion-free rank-one module embeds in $D$ and, after clearing a common denominator, is isomorphic to a left ideal. Two ideal realizations satisfy $I'=Iq$ for some $q\in D^\times$. Therefore
\[
 \gr I'=(\gr I)\sigma(q)
\]
as fractional graded ideals. Dividing by their reflexive hulls, which are their greatest-common-divisor ideals, gives $J_{I'}=J_I$. The normalized ideal has no common factor, so its quotient in the two-variable polynomial ring has finite length. This proves that the BW-invariant $\nu(P)$ is a well-defined nonnegative integer. Moreover, $\nu(P)=0$ precisely when $P$ is free: if $\gr I=gS$, a lift of $g$ generates $I$ by induction on degree. In characteristic zero, this is the Berest--Wilson invariant
of \cite{BW,Berest}.

For a filtered module $Q$, our shift convention is
$F_tQ(a)=F_{t+a}Q$; for a graded module it is $Q(a)_t=Q_{t+a}$.
Thus the standard generator of $A(-b)$ has filtered degree $b$.

\begin{proposition}[Presentation bound]\label{prop:presentation}
If $\nu=\nu(P)>0$, there is a strict filtered resolution
\begin{equation}\label{eq:resolution}
 0\longrightarrow\bigoplus_{j=1}^{r}A(-b_j)
 \longrightarrow\bigoplus_{i=1}^{c}A(-a_i)
 \longrightarrow P\longrightarrow0
\end{equation}
with
\[
 c\leq\nu+1,\quad r=c-1\leq\nu,\quad
 1\leq a_i\leq\nu,\quad b_j\leq\nu+1.
\]
In particular, every relation entry has Bernstein degree at most $\nu$.
\end{proposition}
\begin{proof}
Normalize the induced filtration of an ideal realization by shifting by $\deg g_I$, so that $\gr P=J_P$. For $T=S/J_P$, a nonzero degree-$t$ part forces nonzero parts in degrees $0,\ldots,t$. Thus $T_t=0$ for $t\geq\nu$. The minimal graded resolution of the ideal $J_P$ has length one. The Koszul complex on $\xi,\eta$ bounds the generator degrees by $\nu$ and the syzygy degrees by $\nu+1$. The number of syzygies is the dimension of the socle of $T$, hence at most $\nu$; comparison of ranks gives $c-r=1$.

Lift the minimal homogeneous generators to $P$. Leading-symbol subtraction makes the resulting filtered surjection strict. Its kernel has associated graded equal to the free syzygy module just described. Lifts of a homogeneous basis generate that kernel by induction on degree. They are linearly independent, as one sees by taking the leading term of a putative relation. This proves \eqref{eq:resolution}; its entries have degrees at most $b_j-a_i\leq\nu$.
\end{proof}

The Hilbert series also gives
\begin{equation}\label{eq:moments}
 \sum_i a_i=\sum_j b_j,\qquad
 \nu=\frac12\left(\sum_jb_j^2-\sum_i a_i^2\right),\qquad
 \dim_kF_tP=\binom{t+2}{2}-\nu\quad(t\gg0).
\end{equation}

\begin{proposition}\label{prop:lambda}
For projective $P$, the LM-invariant in \eqref{eq:lambda-definition} is finite and satisfies \eqref{eq:lambda-nu}. It is independent of the ideal realization with its induced degree filtration. For the pullback algebra it is computed by \eqref{eq:norm-filtration}.
\end{proposition}
\begin{proof}
Identify $E=\End_A(P)^{\op}$ with the ring of right multipliers preserving $I$,
\[
 E=\{q\in D:Iq\subset I\},
\]
also called the \emph{right idealizer} of $I$ in $D$. If $q\in D^\times$ satisfies $Iq\subset I$, then
$J_P\sigma(q)\subset J_P$. Localizing at the height-one primes of the normal ring $S$ gives
\[
 \{q\in\Frac S:qJ_P\subset J_P\}=S.
\]
It follows that $\gr E\subset S$.

Apply $\Hom_A(-,P)$ to \eqref{eq:resolution}. Projectivity gives an exact sequence
\[
 0\longrightarrow E\longrightarrow\bigoplus_iP(a_i)
 \xrightarrow{u}\bigoplus_jP(b_j)\longrightarrow0.
\]
The kernel filtration is precisely the division-ring degree filtration. The cokernel of $\gr u$ is $\Ext_S^1(J_P,J_P)$, which is finite-dimensional because $J_P$ is locally free away from the origin. Hence $\gr u$ is surjective in sufficiently high degrees. Degree reduction lifts elements of the target up to a fixed finite-dimensional remainder. Lift a basis of that remainder using surjectivity of $u$. It follows that $F_tu$ is surjective for all large $t$. Using \eqref{eq:moments}, we obtain
\[
 \dim F_tE=
 \sum_i\dim F_{t+a_i}P-\sum_j\dim F_{t+b_j}P
 =\binom{t+2}{2}-2\nu.
\]
Since $\gr E\subset S$, the difference of the two cumulative dimensions is the cumulative dimension of $S/\gr E$, proving the claim. A change $I\mapsto Iq$ conjugates the idealizer by $q$; conjugation preserves degree and principal symbols. This proves independence of the realization. For $\nu=0$ the ideal is principal and the result follows directly.

In characteristic $p$, the reduced norm on $D$ satisfies
\begin{equation}\label{eq:norm-degree}
 \deg_{X,Y}\Nrd(q)=\deg q\qquad(q\ne0),
\end{equation}
where rational-function degree is numerator degree minus denominator degree. For $q\in A$, use the center basis $x^iy^j$, $0\leq i,j<p$, with weights $i+j$. The leading matrix of multiplication by $q$ is multiplication by its nonzero symbol on $\Frac S$ over $\Frac k[\xi^p,\eta^p]$, hence is invertible. Its determinant has Bernstein degree $p^2\deg q$. This determinant is $\Nrd(q)^p$, and a central polynomial has Bernstein degree $p$ times its central degree. This proves \eqref{eq:norm-degree}; multiplicativity extends it to $D$. Under \eqref{eq:pullback-end}, reduced norms agree over the specified center, proving \eqref{eq:norm-filtration}.
\end{proof}

This proof follows Berest's Hilbert-polynomial calculation of $\lambda=2\nu$ \cite[\S3]{Berest}. His argument uses the canonical extension to the noncommutative projective plane; here the strict resolution and the eventual surjectivity of the filtered Hom map give the corresponding calculation directly.

\subsection{Uniform bounds for presentations and module computations}\label{subsec:uniform}
Let $k$ be a field and let $A=A_1(k)$. We use the Bernstein degree on $A$; the degree of a vector in $A^c$ is the maximum degree of its coordinates in the specified basis. Fix a degree-compatible admissible order on the ordered monomials of each free module. The next lemma concerns a single presentation $P=A^c/N$ over this fixed field.

\begin{lemma}\label{lem:uniformGB}
Let $c\geq1$ and $r,d\geq0$ be integers. Let $g_1,\ldots,g_r\in A^c$ have degree at most $d$, and put $N=\sum_jAg_j$. There is an integer $G=G(c,r,d)$, depending only on the displayed integers, such that:
\begin{enumerate}
\item $N$ has a Gr\"obner basis of at most $G$ vectors of degree at most $G$, each expressible in the $g_j$ with coefficient degrees at most $G$.
\item The kernel of the left-linear map
\[
 F_g:A^r\longrightarrow A^c,\qquad
 (a_j)\longmapsto\sum_j a_jg_j,
\]
has at most $G$ generators of degree at most $G$.
\item If $v\in N$ has degree at most $t$, there is an expression $v=\sum_j a_jg_j$ with $\deg a_j\leq t+G$.
\end{enumerate}
The same assertions hold for right submodules, with coefficients on the right.
\end{lemma}
\begin{proof}
The leading monomial of a product of ordered Weyl monomials is their commutative product, with coefficient one; all correction terms have smaller degree. Thus the division and critical-pair algorithms for submodules apply over $k$ in every characteristic.

\emph{Step 1: the Gr\"obner basis and its expressions.}
Start by reducing the $g_j$ successively, retaining each nonzero remainder, and then apply Buchberger's algorithm. Whenever the algorithm adds a vector, its leading monomial is not divisible by any previously retained leading monomial. If the current vectors have degree at most $D$, each critical pair and its remainder have degree at most $2D$. Consequently, writing the $j$th retained leading monomial as
$x^{\alpha_{j,1}}y^{\alpha_{j,2}}e_{a_j}$, with
$\alpha_j=(\alpha_{j,1},\alpha_{j,2})\in\mathbb N^2$, we have
\[
 |\alpha_j|\leq 2^j\max(1,d),\qquad 1\leq a_j\leq c,
\]
and no earlier monomial in the sequence divides it.

There is a bound, depending only on $c,d$, on the length of every such sequence. Indeed, the tree of these finite sequences is finitely branching, since there are only finitely many choices at each position. An infinite branch would contradict Dickson's lemma: in an infinite sequence in $\mathbb N^2\times\{1,\ldots,c\}$, some earlier pair has the same last coordinate and componentwise smaller exponents than a later one. A finitely branching tree of unbounded height has an infinite branch, so this tree has finite height. This bounds both the number and the degrees of the retained basis vectors.

The number of critical pairs is consequently bounded. Each division has only a bounded number of steps, since its leading monomial strictly decreases among the finitely many module monomials of bounded degree. Track every operation as an expression in the original $g_j$. Each step uses a scalar or a monomial of bounded degree, and there are only a bounded number of steps. Hence the degrees of all coefficients in these expressions are bounded in terms of $c,r,d$. Enlarging $G$ gives (i).

\emph{Step 2: the kernel of the presentation map.}
Let $b_1,\ldots,b_s$ be the basis obtained in Step~1, and let $F_b:A^s\to A^c$ send the standard basis to these vectors. Reducing each critical pair to zero gives its syzygy among the $b_\ell$. These syzygies generate $\ker F_b$: in a relation among the $b_\ell$, the largest image monomial must cancel between two terms, and subtracting the corresponding critical-pair syzygy lowers the leading term in the induced Schreyer order. Iteration terminates. There are at most $s^2$ such syzygies, and their entries have bounded degree by Step~1.

The expressions of the $b_\ell$ in the $g_j$, and division of each $g_j$ by the $b_\ell$, give maps
\[
 u:A^s\longrightarrow A^r,\qquad v:A^r\longrightarrow A^s,
 \qquad F_gu=F_b,\quad F_bv=F_g,
\]
whose entries have bounded degree. One has
\begin{equation}\label{eq:syzygy-transfer}
 \ker F_g=u(\ker F_b)+\operatorname{im}(1-uv).
\end{equation}
Both summands on the right lie in the kernel. Conversely, if $w\in\ker F_g$, then $vw\in\ker F_b$ and $w=u(vw)+(1-uv)w$. The images of the above syzygies under $u$, together with the images of the $r$ standard basis vectors under $1-uv$, therefore generate $\ker F_g$. Their number and degrees are bounded, proving (ii) after increasing $G$.

\emph{Step 3: membership expressions.}
Divide $v\in N$ by the basis $b_\ell$. Its remainder is zero, and degree compatibility bounds all quotient degrees by $t$. Substitution of the expressions from (i) gives coefficients of degree at most $t+G$ in the original $g_j$, proving (iii). Finally, the degree-preserving isomorphism $A_1(k)^{\op}\simeq A_1(k)$ obtained by fixing $x$ and negating $y$ gives the right-module assertions. Explicit numerical degree bounds can also be obtained from \cite{AL}.
\end{proof}

For the applications to reductions, we use the following terminology. A family $(P_i)_{i\in I}$ of finitely generated left $A_1(k_i)$-modules, whose base fields may vary, \emph{admits bounded presentations} if there are integers $c,r,d$ and presentations
\[
 P_i\simeq A_1(k_i)^{c_i}/\sum_{j=1}^{r_i}A_1(k_i)g_{ij},
 \qquad c_i\leq c,\quad r_i\leq r,\quad \deg g_{ij}\leq d
 \quad(i\in I).
\]
The right-module convention is analogous. Lemma~\ref{lem:uniformGB} applies separately to each presentation. Taking the maximum of $G(c',r',d)$ over $c'\leq c$ and $r'\leq r$ gives a common bound for these computations.

\begin{lemma}\label{lem:bounded-ideals}
For integers $c\geq1$ and $r,d\geq0$, there is an integer $H_0=H_0(c,r,d)$ such that, for every field $k$, a torsion-free rank-one left $A_1(k)$-module with a presentation having at most $c$ generators, at most $r$ relations, and relation entries of degree at most $d$ is isomorphic to a nonzero left ideal generated by at most $c$ operators of degree at most $H_0$.

For a family of torsion-free rank-one left $A_1(k_i)$-modules $P_i$, the following are equivalent: the family admits bounded presentations; the integers $\nu(P_i)$ are bounded.
\end{lemma}
\begin{proof}
The assertion is immediate for a presentation without relations, since a free module of rank one is isomorphic to $A_1(k)$. Otherwise write
\[
 P=A_1(k)^{c'}/\sum_{i=1}^{r'}A_1(k)(t_{i1},\ldots,t_{ic'}),
 \qquad c'\leq c,\quad r'\leq r,\quad \deg t_{ij}\leq d,
\]
and put $T=(t_{ij})$. A left-linear map $P\to A_1(k)$ is specified by a column $a=(a_j)$ with $Ta=0$. Thus the maps form the right syzygy module of the columns of $T$. This module is nonzero because $P$ has an ideal realization. Lemma~\ref{lem:uniformGB}(ii), applied on the right with ambient rank $r'$ and $c'$ input columns, supplies a nonzero solution of bounded degree. The resulting map $P\to A_1(k)$ is injective: after localization it is a nonzero map between one-dimensional vector spaces over the division ring, and its kernel is torsion in the torsion-free module $P$. Its coordinates generate the required ideal. Taking the maximum over $c'\leq c$ and $r'\leq r$ gives $H_0$.

Apply Lemma~\ref{lem:uniformGB}(i) to these ideal generators. It bounds the degrees of generators of $\gr I$ and hence of $J_P$. After extending to an infinite field, choose two homogeneous elements of one bounded degree in $J_P$ with no common factor. Their complete intersection bounds $\dim_k(S/J_P)$ by the square of that degree. This proves that bounded presentations give bounded BW-invariants. Conversely, Proposition~\ref{prop:presentation} gives bounded presentations when $\nu(P)>0$; use a free generator when $\nu(P)=0$.
\end{proof}

\begin{proposition}\label{prop:actions}
Let $k$ be a perfect field of characteristic $p>0$, and let $f:\A_k^2\to\A_k^2$ be an endomorphism with Jacobian one and degree at most $d$, where $d\geq1$. Suppose $p>d$, so that Proposition~\ref{prop:graphsplit} defines the splitting $\mathcal D$-module on the graph $M_{f,k}$. Assume
\[
 \nu(M_{f,k})\leq N
\]
for an integer $N\geq0$. Then there is an integer $B=B(d,N)$, depending only on $d,N$, such that $M_{f,k}$ has a left $A_1(k)$-presentation with at most $B$ generators and $B$ relations, relation entries of degree at most $B$, and matrices for the right actions of $x',y'$ whose entries have degree at most $B$.

More precisely, if $\nu=\nu(M_{f,k})>0$, the corresponding $A_2(k)$-module has a presentation with at most $\nu+1$ generators, $3\nu+2$ relations, and relation entries of degree at most $\nu+d$. If $\nu=0$, one generator and two relations of degree at most $d$ suffice.
\end{proposition}
\begin{proof}
Realize $M_{f,k}$ as a nonzero left ideal $I\subset A_1(k)$, and put $D=\Frac A_1(k)$. Write $x',y'$ for the generators of the right copy of $A_1(k)$. Their actions on $I$ are right multiplication by elements $q_x,q_y\in D$ satisfying $Iq_x,Iq_y\subset I$. The equations of the graph $\Gamma_f^{(1)}$ identify the right action of $(x')^p$ with the left action of $f_1^{[p]}(x^p,y^p)$, and similarly for $y'$ and $f_2$. These left-acting elements are central. Cancelling any nonzero element of $I$ in $D$ therefore gives
\[
 q_x^p=f_1^{[p]}(x^p,y^p),\qquad
 q_y^p=f_2^{[p]}(x^p,y^p).
\]
Additivity of degree yields $\deg q_x=\deg f_1$ and $\deg q_y=\deg f_2$. Suppose first that $\nu>0$, and use the normalized strict presentation from Proposition~\ref{prop:presentation}. A generator $e_i$ has degree $a_i$, so $e_iq_x$ and $e_iq_y$ have degree at most $a_i+d$. Strictness expresses these elements in the generators $e_j$ with coefficient degrees at most
\[
 a_i+d-a_j\leq\nu+d-1.
\]
Adjoin the resulting $2c$ right-action relations to the $r$ left relations. These present the bimodule: every word in the right generators reduces to a left linear combination, and the given right action identifies the resulting module with $M_{f,k}$. The bounds $c\leq\nu+1$ and $r\leq\nu$ give the assertion. If $\nu=0$, choose a free left generator; its two right-action coefficients have degree at most $d$.

\end{proof}

\subsection{Proof of the arithmetic realization criterion}\label{subsec:realization}
The following consequence of the filtration argument in Section~\ref{subsec:restriction} will also be used for parameter schemes whose fibers have different characteristics.

\begin{lemma}\label{lem:family-stratification}
Let $S$ be a scheme of finite type over $\Z$, and let $N$ be a coherent $\mathcal D_{\A^2_S/S}$-module. There is a finite partition of $S$ into reduced, integral, affine, locally closed subschemes $S_a$ such that $N_{S_a}$ is flat over $S_a$ and the dimensions of its geometric fibers are constant on $S_a$. More precisely, on each $S_a$ one can choose a good Bernstein filtration whose filtered pieces and graded pieces commute with base change and are free of finite rank over $S_a$.
\end{lemma}
\begin{proof}
On an integral affine subscheme $\Spec B$, apply the monic
Gr\"obner-basis construction of Section~\ref{subsec:restriction} to the
finitely generated $A_2(B)$-module of sections. The construction of
\eqref{eq:filtered-base-change} works in every characteristic and gives
one principal open on which all filtered and graded pieces are finite
free and commute with base change. The module is flat there as a
filtered union of free modules, and its geometric fibers have the same
filtered Hilbert function, hence the same dimension; compare
\cite[Lemma~4.1.5]{Bitoun}.

For general $S$, pass first to its reduction; this does not change any field-valued fiber. On each irreducible component choose a nonempty affine open disjoint from the other components, and shrink it as in the preceding paragraph. These opens form disjoint strata containing all generic points. Apply the same construction to the reduced closed complement. Noetherian induction yields a finite partition into strata with the asserted properties.
\end{proof}

\begin{proof}[Proof of Theorem~\ref{thm:equivalence}]
The equivalence of (i) and (ii) is Proposition~\ref{prop:lambda}.
Assume (i). By Proposition~\ref{prop:actions}, the $A_2(k_i)$-modules
$M_i$ have presentations with bounded numbers of generators and
relations and bounded relation degrees on a set belonging to $\UU$.
Restricting to such a set, fix the numbers $c,r$ and write
\[
 M_i\simeq A_2(k_i)^c\Big/\sum_{j=1}^r A_2(k_i)g_{ji}.
\]
Taking coefficientwise ultraproducts of these bounded-degree vectors
gives $g_j\in A_2(K)^c$. Let $R\subset K$ be the $R_0$-algebra
generated by their finitely many coefficients, and put
\[
 M_R=A_2(R)^c\Big/\sum_{j=1}^r A_2(R)g_j.
\]
The defining ideal of $R$ in a polynomial ring over $R_0$ is finitely
generated. Its generators vanish under the chosen coefficient
representatives for $\UU$-almost every $i$. Thus these representatives
induce homomorphisms $t_i:R\to k_i$ extending $s_i$, with
$[t_i(a)]_\UU=a$ for every $a\in R$, and
\[
 M_R\otimes_{R,t_i}k_i\simeq M_i
 \qquad\text{for $\UU$-almost every $i$}.
\]
The points $\sigma_i=\ker t_i$ are closed because $k_i$ is finite.
They are Zariski dense: for every nonzero $a\in R\subset K$, the set
of indices with $t_i(a)\ne0$ belongs to $\UU$. The graph-support and
multiplicity-one conditions descend from $k_i$ to $k(\sigma_i)$ by
faithful flatness, as in the proof of
Corollary~\ref{cor:arithmetic-recognition}.
Proposition~\ref{prop:recognition} now shows that $M_{\Frac R}$ is holonomic and
projective of left rank one. Hence $M_R$ is an arithmetic realization,
proving (iii).

Conversely, assume (iii). Corollary~\ref{cor:arithmetic-recognition}
gives a finite left presentation of $M_R$ after localization, specializing
to the prescribed $M_i$ for $\UU$-almost all $i$. Its numbers of
generators and relations and its relation degrees give bounded left
presentations of those modules. Lemma~\ref{lem:bounded-ideals} gives (i).
The same corollary proves that the generic module is rank-one projective
on the left.

\emph{(iii)$\Rightarrow$(iv).}
Regard an arithmetic realization $M_R$ as a coherent $\mathcal D_{\A^2_R/R}$-module. Apply Lemma~\ref{lem:family-stratification} on a principal open $T=\Spec R[r^{-1}]$ containing the generic point. The family is flat over $T$, and all its geometric fibers have the same dimension as the holonomic generic fiber, namely two. For $\UU$-almost every $i$, the maps $t_i$ of Definition~\ref{def:realization} satisfy $t_i(r)\ne0$ and hence define morphisms $u_i:\Spec k_i\to T$. The prescribed fiber isomorphisms give (iv).

\emph{(iv)$\Rightarrow$(iii).}
Choose a finite affine open cover of $T$. There is an affine open $\Spec B$ such that $u_i$ has image in $\Spec B$ for a set of indices belonging to $\UU$. Restrict to that set and let $b_i:B\to k_i$ be the corresponding $R_0$-algebra maps. They induce
\[
 b:B\longrightarrow K,\qquad b(a)=[b_i(a)]_{\UU}.
\]
Put $\mathfrak q=\ker b$ and identify $R=B/\mathfrak q$ with its image in $K$. This is an integral finitely generated $R_0$-algebra. Indeed, the restriction of $b$ to $R_0$ is the fixed injective map supplied by faithfulness.

The ideal $\mathfrak q$ has finitely many generators. For each generator, the set of indices at which its image under $b_i$ vanishes belongs to $\UU$. Intersecting these finitely many sets shows that $b_i$ factors through $R$ for $\UU$-almost every $i$. Write $t_i:R\to k_i$ for these factorizations. If $N_B$ is the module of sections of $M_T$ over $\A^2_B$, set
\[
 M_R=N_B\otimes_B R.
\]
It is finitely generated over $A_2(R)$ and satisfies $M_R\otimes_{R,t_i}k_i\simeq M_i$, with $[t_i(r)]_{\UU}=r$ for every $r\in R$. Its generic fiber is the fiber of $M_T$ at the point $\mathfrak q\in\Spec B$, whose residue field is $\Frac R$. This field has characteristic zero, and the fiber has dimension two by (iv); hence it is holonomic. Thus $M_R$ is an arithmetic realization.

Finally, $R$ is countable and its fraction field embeds in $\C$ extending the given embedding of $\Frac R_0$: extend a transcendence basis and then the algebraic extension. The embedding is possible because $\C$ has infinite transcendence degree over the countable field $\Frac R_0$. This proves the last assertion.
\end{proof}

\begin{proposition}\label{prop:geometric-boundedness}
The numerical and geometric formulations of Conjecture~\ref{conj:bound} are equivalent.
\end{proposition}
\begin{proof}
\emph{Step 1: a numerical bound gives a family of finite type.}
Fix $d$ and assume the numerical formulation. Proposition~\ref{prop:actions}
gives integers $c,r,D$, depending only on $d$ and $N(d)$, such that every
module under consideration has an $A_2(k)$-presentation with at most $c$
generators and $r$ relations, whose entries have Bernstein degree at most $D$.

For each $1\leq c'\leq c$ and $0\leq r'\leq r$, introduce independent
coefficients for the ordered monomials of degree at most $D$ in the entries
of $r'$ relations on $c'$ generators. These coefficients define an affine
space over $\Z$ and a universal presentation over the sheaf of relative
differential operators on the relative affine plane. Let $S$ be the finite
disjoint union of these parameter spaces, and let $N_S$ be the resulting
coherent $\mathcal D_{\A^2_S/S}$-module. Every specified splitting
$\mathcal D$-module is a fiber of $N_S$ at a point defined over its original
field $k$.

Apply Lemma~\ref{lem:family-stratification} and retain the strata whose
fibers have dimension two. Their disjoint union $T_d$, with the restricted
module $M_{T_d}$, satisfies the geometric formulation. In particular, this
construction includes every characteristic $p>d$.

\emph{Step 2: a family of finite type gives a bound in each fixed
characteristic.}
Conversely, assume the geometric formulation and fix $d$. A finite affine
open cover of $T_d$ and finite presentations of $M_{T_d}$ on its members
give integers $c,r,D$ bounding the numbers of generators and relations and
the Bernstein degrees of their entries in all fibers. Fix a prime $p>d$.
We show that these presentations give bounded left $A_1(k)$-presentations
of the specified fibers over all finite fields $k$ of characteristic $p$.

Write $A=A_1(k)$ for the algebra acting on the left and $u,v$ for the right-acting
generators; thus $[u,v]=1$. For a specified endomorphism $f$, put
\[
 h_j=f_j^{[p]}(x^p,y^p)\in Z(A)\qquad(j=1,2).
\]
The relations $u^p=h_1$ and $v^p=h_2$ annihilate $M_{f,k}$. Consequently,
its given bimodule presentation is also a presentation over
\[
 C_f=(A\otimes_k A_1(k)^{\op})/(u^p-h_1,v^p-h_2).
\]
As a left $A$-module, $C_f$ is free with basis
\[
 u^a v^b,\qquad 0\leq a,b<p.
\]
Indeed, the copy of $A_1(k)^{\op}$ acting on the right is free over its center with precisely this
basis, and $C_f$ is obtained by the indicated central base change.

A presentation with $c'\leq c$ generators and $r'\leq r$ relations over
$C_f$ therefore gives a left $A$-presentation with $c'p^2$ generators and
$r'p^2$ relations: multiply each original relation on the left by every
$u^a v^b$ in this basis and reduce its right exponents modulo $p$.
Give $x,y$ weight one and $u,v$ weight $d$. The substitutions
$u^p=h_1$, $v^p=h_2$ do not increase weighted degree, since
$\deg h_j\leq pd$. Ordering the monomials also does not increase it.
Thus the left coefficients in this presentation have degree at most
\[
 dD+2d(p-1).
\]
Lemma~\ref{lem:bounded-ideals} now gives a bound for $\nu(M_{f,k})$
depending on $d,p$ and the chosen family, but not on $k$ or $f$.

\emph{Step 3: the bounds are uniform over all $p>d$.}
Suppose that no single bound works for these primes. Step~2 bounds the
invariants in any finite set of characteristics. We can therefore choose
finite fields $k_i$ and endomorphisms $f_i:\A^2_{k_i}\to\A^2_{k_i}$ with
Jacobian one such that
\[
 \deg f_i\leq d,\qquad
 p_i:=\operatorname{char}k_i>\max\{i,d\},\qquad
 \nu(M_{f_i,k_i})>i.
\]
Choose points $u_i:\Spec k_i\to T_d$ realizing these modules and a
nonprincipal ultrafilter $\UU$. The field $K=\prod_{\UU}k_i$ has
characteristic zero. Taking the ultraproducts of the finitely many
coefficients gives an endomorphism $f$ of $\A^2_K$ of degree at most $d$
with Jacobian one.

Let $R_0\subset K$ be the $\Z$-algebra generated by these
coefficients. As in the proof of (i)$\Rightarrow$(iii) in
Theorem~\ref{thm:equivalence}, the coefficient representatives
give maps $s_i:R_0\to k_i$ for $\UU$-almost every $i$, with
\[
 [s_i(a)]_{\UU}=a\quad(a\in R_0),\qquad
 f\otimes_{s_i}k_i=f_i.
\]
These maps form a faithful arithmetic family. Choose an embedding $R_0\hookrightarrow
\C$ to apply Theorem~\ref{thm:equivalence}; this does not change the
specializations or their modules.

Base change $T_d$ and $M_{T_d}$ to $\Spec R_0$. The pairs $(u_i,s_i)$
give the points required by condition~(iv) of that theorem. Its implication
(iv)$\Rightarrow$(i) bounds $\nu(M_{f_i,k_i})$ on a set belonging to
$\UU$, contradicting $\nu(M_{f_i,k_i})>i$. This proves the numerical
formulation for every $p>d$.
\end{proof}

\subsection{The second conditional implication}\label{subsec:ultra-implication}
\begin{proof}[Proof of Corollary~\ref{cor:ultra}]
Let $f$ be an endomorphism of $\mathbb A^2_{\mathbb C}$ with Jacobian one, and let $R_0$ be the integral $\Z$-algebra generated by its coefficients. Enumerate the nonzero elements of $R_0$. For each $i$, choose a finite-field point of characteristic greater than $i$ and $\deg f$ at which the first $i$ elements remain nonzero. These points exist because a nonempty open subset of $\Spec R_0$ has closed points in arbitrarily large characteristics. Any nonprincipal ultrafilter then gives a faithful arithmetic family.

Under Conjecture~\ref{conj:bound}, Theorem~\ref{thm:equivalence} produces an arithmetic realization. Embed its characteristic-zero coefficient field into $\C$ as in that theorem. Its generic bimodule $M$ is a rank-one projective left $A_1$-module and a generator for the category of left $A_1$-modules, so its right action maps $A_1$ into the domain $\End_{A_1}(M)^{\op}\in\mathcal D_1$. Under $\mathrm{DC}'_1$ this map is an isomorphism. Thus $M$ is invertible.

Invertibility descends along the faithfully flat coefficient-field
extension. Lemma~\ref{lem:invertible-descent} makes the arithmetic model
invertible after one localization of $R$, which faithfulness permits
for $\UU$-almost all $i$. Hence the prescribed modules $M_i$ are
invertible bimodules for those indices. Their center isomorphisms are $(f_{k_i}^{(1)})^*$, so $f_{k_i}$ is invertible. The finite identities \eqref{eq:inverse-identities} pass to $R_0\subset K$ and then to $\C$, proving that $f$ is invertible. The proof uses only one faithful arithmetic family with bounded BW-invariants for each $f$.
\end{proof}

\subsection{Remarks on this approach}\label{subsec:limitations}
The BW-invariants of rank-one projective modules are unbounded
even over one fixed finite field.

\begin{example}\label{ex:fixed-field-BW}
Fix a field $k$ of characteristic $p>0$ and put $A=A_1(k)$, with
$[y,x]=1$. For every $N\geq1$, set
\[
 g_N(y)=y^{p^N}-y,\qquad
 I_N=A\bigl(xg_N(y)+1\bigr)+A g_N(y)^2.
\]
Then
\[
 I_N\oplus A\simeq A^2,\qquad \nu(I_N)=p^N.
\]
In particular, one may take $k=\F_p$ throughout.
\end{example}
\begin{proof}
Write $g=g_N$, $m=p^N$, $a=xg+1$, and $b=g^2$. Since $g'=-1$,
\[
 gx=xg-1,\qquad ga=xb.
\]
These identities give the complex
\begin{equation}\label{eq:fixed-field-resolution}
 0\longrightarrow A
 \xrightarrow{\alpha} A^2
 \xrightarrow{\beta} I_N\longrightarrow0,
 \qquad
 \alpha(w)=(-wg,wx),\quad \beta(u,v)=ua+vb.
\end{equation}
We verify exactness at $A^2$ and compute the associated graded ideal at
the same time. The principal symbols of $a,b$ are
$\xi\eta^m,\eta^{2m}$, respectively. Their graded syzygy module is
generated by $(-\eta^m,\xi)$, which is the principal symbol of
$(-g,x)$ when the two summands of $A^2$ are placed in degrees $m+1,2m$.
Whenever the leading terms in an expression $ua+vb$ cancel, the leading
coefficient pair is a multiple of this graded syzygy. Subtracting a lift
of that multiple of $(-g,x)$ lowers its filtered degree without changing
$ua+vb$. Induction on degree proves both
\[
 \ker\beta=\operatorname{im}\alpha,
 \qquad
 \gr I_N=(\xi\eta^m,\eta^{2m}).
\]
The first arrow is injective because $A$ is a domain. Thus
\eqref{eq:fixed-field-resolution} is exact.

The left $A$-linear map
\[
 s:A^2\longrightarrow A,\qquad s(u,v)=ux+vg,
\]
satisfies $s\alpha(w)=w(-gx+xg)=w$. It splits the first arrow, proving
$I_N\oplus A\simeq A^2$. The normalized graded ideal is
\[
 J_{I_N}=(\xi,\eta^m),
\]
so its colength, and hence the BW-invariant, is $m=p^N$.
\end{proof}

Lemma~\ref{lem:bounded-ideals} shows that the modules in
Example~\ref{ex:fixed-field-BW} do not admit bounded presentations.
This does not contradict Conjecture~\ref{conj:bound}: no realization of
these ideals as splitting $\mathcal D$-modules on graphs associated with
endomorphisms of uniformly bounded degree is asserted.

The BW-invariant in \eqref{eq:nu-definition} uses the fact that the associated graded ring has dimension two. For $A_n$, $n>1$, removal of the common divisor leaves a quotient supported in codimension at least two, but it need not have finite length.

\subsection{Unbounded presentations for the counterexamples}\label{as:sec:ultra}
The two counterexamples constructed in
Section~\ref{sec:nonliftability} also give an obstruction to bounded
presentations. We record that consequence here, using the nonlifting
results proved there.

\begin{corollary}\label{cor:no-bounded-graph-presentations}\label{as:prop:ultra}
Let $L=\Gamma_f\subset T^*\A^4$, where $f$ is the endomorphism of
Section~\ref{subsec:long-map}, or let $L=L_2\subset T^*\A^2$ be the plane
\eqref{as:emb}. Write $m=4$ or $m=2$, respectively, and let $M_{L,k}$
be the splitting $A_m(k)$-module on $L_k^{(1)}$.
Let $k_i$ be finite fields of sufficiently large characteristics $p_i$, and let
$\UU$ be a nonprincipal ultrafilter such that $p_i\to_\UU\infty$,
that is, $\{i:p_i>B\}\in\UU$ for every integer $B$.
There are no uniform bounds on the numbers of generators and relations
and on the Bernstein degrees of the relation entries in presentations
of $M_{L,k_i}$ on any set belonging to $\UU$.
In particular, if $p_i\to\infty$, no infinite subsequence admits such bounds.
\end{corollary}
\begin{proof}
Suppose bounded presentations exist, and put $K=\prod_\UU k_i$.
On a set belonging to $\UU$ the presentation matrices have a common
size. Take the ultraproducts of their coefficients in the standard ordered monomials.
This gives a finite presentation over $A_m(K)$. Let $R\subset K$ be
generated over $\Z[1/30]$ by its finitely many coefficients, and let
$M_R$ be the presented module.

As in the proof of (i)$\Rightarrow$(iii) in
Theorem~\ref{thm:equivalence}, the coefficient representatives
give maps $t_i:R\to k_i$ for $\UU$-almost every $i$.
The corresponding closed points $\ker t_i$ are Zariski dense, and
\[
 M_R\otimes_{R,t_i}k_i\simeq M_{L,k_i}.
\]
Support and normalized multiplicity descend along the finite
residue-field extensions $k(\ker t_i)\subset k_i$.

For $L=\Gamma_f$, Proposition~\ref{prop:recognition}, with $n=2$,
gives a rank-one projective left $A_2(\Frac R)$-module and finite left
generation after localizing $R$. This contradicts
Proposition~\ref{prop:module-nonlift}. For $L=L_2$, apply the coordinate
change and tensor product with a connection leading to
\eqref{as:normal}. Lemma~\ref{as:reduce} gives operators of orders two
and three over $(\Frac R)(v)$ with commutator one, contradicting
Lemma~\ref{lem:orders-two-three}. The assertion about subsequences
follows by choosing a nonprincipal ultrafilter on such a subsequence.
\end{proof}

\section{Counterexamples to Kontsevich's conjecture}\label{sec:nonliftability}
Both counterexamples rely on the same obstruction: scalar differential
operators of orders two and three over the function field of a curve
cannot have a nonzero constant commutator. We prove this obstruction
first, then apply it to the Lagrangian affine plane of
Theorem~\ref{thm:counterexamples}(ii). Long's graph gives the additional
matrix-endomorphism nonlifting assertion in
Theorem~\ref{thm:counterexamples}(i). Finally, we prove that products with fibers of a cotangent bundle give counterexamples in every higher dimension.

\subsection{The obstruction}
\begin{lemma}\label{lem:orders-two-three}
Let $K$ be a field of characteristic zero and let $F$ be a function field of some curve over $K$. Let $D$ be a nonzero $K$-derivation of $F$. No two operators in $F[D]$ of orders two and three have a commutator belonging to $K^\times$.
\end{lemma}
\begin{proof}
We may assume that $K$ is algebraically closed.

Suppose that such operators exist, with leading coefficients $a,b$.
The coefficient of $D^4$ in their commutator gives
$2aD(b)-3bD(a)=0$. For the derivation $\delta=(b/a)D$, both operators
have leading coefficient $h=a^3/b^2\in K^\times$. Dividing them by $h$
makes them monic. If the resulting order-two operator is
$\delta^2+u_1\delta+u_0$, put $\Delta=\delta+u_1/2$.
The relation $[\Delta,f]=\delta(f)$ holds for every $f\in F$.
In terms of $\Delta$ the operators have the form
\[
 P=\Delta^2+U,\qquad
 Q=\Delta^3+b_2\Delta^2+b_1\Delta+b_0,
 \qquad [P,Q]=\kappa\in K^\times.
\]
Write $f'=\delta(f)$. Comparing the positive-order coefficients of the
commutator shows that $b_2\in K$ and, after replacing $Q$ by $Q-b_2P$,
\[
 b_1=\tfrac32U+c,\qquad b_0=\tfrac34U'+d
 \quad(c,d\in K).
\]
The constant coefficient gives
\[
 -\tfrac14U'''-(\tfrac32U+c)U'=\kappa.
\]
Set $W=U+2c/3$, $\lambda=-4\kappa$, and
$t=(W''+3W^2)/\lambda$. Then $t'=1$, so $F$ is finite over $K(t)$,
$\delta$ extends $d/dt$, and
\begin{equation}\label{eq:operator-residue-equation}
 W''+3W^2=\lambda t,\qquad \lambda\ne0.
\end{equation}

Let $C$ be the smooth projective curve with function field $F$, and let
$e_q$ denote the ramification index of $t:C\to\PP^1$ at $q$.
Consider the rational differential form $\omega=tW\,dt$ on $C$.
At a pole $q$ of $W$ over a finite value $t=a$, put $x=t-a$
and write $\ord_x=\ord_q/e_q$. Comparison of the leading
terms in a Puiseux expansion gives $W=-2x^{-2}+w$ with
$\ord_xw>-2$. Substitution in \eqref{eq:operator-residue-equation} gives
\[
 w''-12x^{-2}w+3w^2=\lambda(a+x).
\]
If the leading exponent $s$ of $w$ were less than $2$, its coefficient
would satisfy $s(s-1)-12=0$. The roots are $-3$ and $4$, so neither
lies in $(-2,2)$. Hence $W=-2x^{-2}+O(x^2)$, and
\[
 \Res_q(\omega)=-2e_q.
\]
At a point over infinity, the leading terms of
\eqref{eq:operator-residue-equation} instead give
\[
 W=c_qt^{1/2}+\tfrac1{24}t^{-2}+O(t^{-9/2}),
 \qquad 3c_q^2=\lambda.
\]
Indeed the leading exponent must be $1/2$; the next coefficient $B$
satisfies $-c_q/4+6c_qB=0$. Thus
$\Res_q(\omega)=-e_q/24$. At every other point $\omega$ is regular.
Since $\sum_{t(q)=\infty}e_q=[F:K(t)]$, we obtain
\[
 \sum_{q\in C}\Res_q(\omega)
 =-2\!\sum_{\substack{q:\,W\text{ has a pole at }q\\t(q)\ne\infty}}
       e_q-\frac{[F:K(t)]}{24}\ne0.
\]
This contradicts the global
residue theorem \cite[Corollary to Theorem~3, p.~155]{TateResidues}.
\end{proof}

\subsection{A Lagrangian affine plane}\label{as:sec:plane}
Let $L_2$ be the image of \eqref{as:emb}, with the arithmetic model over
$\Z[1/30]$ defined by the same parametrization.

Write $(v,w;\xi,\sigma)$ for the cotangent coordinates on
$T^*\A^2_{v,w}$, so that the tautological one-form is
$\xi\,dv+\sigma\,dw$.

\begin{lemma}\label{as:lem:emb}
The map \eqref{as:emb} is a closed embedding.
For every finite field $k$ of characteristic $p>5$, the restriction of
the crystalline differential-operator algebra to $L_{2,k}^{(1)}$ splits,
and its splitting bundle is unique up to isomorphism.
\end{lemma}
\begin{proof}
The identity
\begin{equation}\label{as:inv}
 z=\frac w2-\frac{v^2w^2}{8}+\frac{v\sigma}{4}+\frac{v^4\xi}{24}
\end{equation}
on the image makes the homomorphism of coordinate rings surjective.
Moreover,
\begin{equation}\label{as:primitive}
 j^*(\xi\,dv+\sigma\,dw)=dF,\qquad
 F=v^3z^4+\tfrac45v^5z^5.
\end{equation}
Thus the image is an exact Lagrangian surface. These identities remain
valid in characteristic $p>5$. The exactness criterion
\eqref{eq:liouville-brauer} gives a splitting on $L_{2,k}^{(1)}$.
Splitting bundles differ by a line bundle, and
$\operatorname{Pic}(L_{2,k}^{(1)})=\operatorname{Pic}(\A_k^2)=0$.
\end{proof}

We denote this splitting bundle, regarded as an $A_2(k)$-module, by $M_k$.
On $v\ne0$, set
\[
 a=vz+v^{-1},\qquad t=w+3v^{-2},\qquad
 G=3v^{-1}w^2+8v^{-3}w+\tfrac{24}{5}v^{-5}.
\]
Translate the fibers of the cotangent bundle by $dG$ and use base
coordinates $(v,t)$. If $(e,b)$ are the new covector coordinates, then
$e\,dv+b\,dt=\xi\,dv+\sigma\,dw+dG$, and substitution gives
\begin{equation}\label{as:normal}
 e=24v^{-5}a,\qquad t=a^2+2v^{-2},\qquad b=2a^3+6v^{-2}a.
\end{equation}
Hence, over $\C(v)$, this is a graph $t=T(e)$, $b=H(e)$ with
$\deg T=2$ and $\deg H=3$.

The corresponding operations on $\mathcal D$-modules are change of base
coordinates and tensor product with the rank-one connection $d+dG$.
They preserve arithmetic multiplicities. In characteristic $p>5$,
\[
 (\partial_i+\partial_iG)^p
 =\partial_i^p+(\partial_iG)^p+\partial_i^{\,p}(G)
 =\partial_i^p+(\partial_iG)^p,
\]
because the $p$th iterate of a coordinate derivation vanishes on the
rational function field. Thus the induced transformation of the center
is precisely the Frobenius twist of the stated cotangent transformation.

For the following lemma, let $R$ be an integral finitely generated
$\Z$-algebra of characteristic zero, put $K=\Frac R$, and choose
$h\in R[v]$ whose image in $k(s)[v]$ is nonzero for every
$s\in\Spec R$. Put
\[
 B_R=R[v,h^{-1}],\qquad U_R=\Spec B_R,\qquad
 X_R=U_R\times_R\A^1_R,
\]
and use coordinates $(v,t)$ on $X_R$. We denote the fiber coordinates
on $T^*(X_R/R)$ by $(e,b)$, so that its tautological differential
form is $e\,dv+b\,dt$. Set
\[
 D_R=B_R\langle\partial_v\rangle,\qquad
 \mathcal A_R=D_R\otimes_R A_1(R),
\]
where $[\partial_v,f]=df/dv$ for $f\in B_R$, and the Weyl algebra
$A_1(R)$ has generators $t,\partial_t$ satisfying
$[\partial_t,t]=1$. Thus $\mathcal A_R$ is the ring of relative
crystalline differential operators on $X_R$. We also write
\[
 D_K=K(v)[\partial_v],\qquad
 \mathcal A_K=D_K\otimes_K A_1(K).
\]

\begin{lemma}\label{as:reduce}
Let $T,H\in B_R[e]$ have positive degrees in $e$, and let $N$ be a
finitely presented left $\mathcal A_K$-module. Suppose there are a
finitely presented left $\mathcal A_R$-module $N_R$ and an
$\mathcal A_K$-linear isomorphism
\[
 K(v)\otimes_{B_R}N_R\simeq N,
\]
where $\partial_v$ acts on the tensor product by the Leibniz rule.

For a closed point $s\in\Spec R$ with residue field $k(s)$ of
characteristic $p$, put
\[
 B_s=k(s)[v,h_s^{-1}],\qquad N_s=N_R\otimes_R k(s),\qquad
 N_s^{\mathrm{gen}}=k(s)(v)\otimes_{B_s}N_s.
\]
The last module is a module over
\[
 D_s\otimes_{k(s)}A_1(k(s)),\qquad D_s=k(s)(v)[\partial_v],
\]
whose center is
\[
 Z_s=k(s)(v^p)[\partial_v^p,t^p,\partial_t^p].
\]
For $f(e)=\sum_j c_j(v)e^j\in k(s)(v)[e]$, write
$f^{[p]}(u)=\sum_j c_j(v)^p u^j$, and define the prime ideal
\[
 I_s=\bigl(t^p-T_s^{[p]}(\partial_v^p),\,
          \partial_t^p-H_s^{[p]}(\partial_v^p)\bigr)\subset Z_s.
\]
Assume that, for a Zariski-dense set $S$ of closed points of
$\Spec R$,
\begin{equation}\label{as:eq:generic-support}
 \sqrt{\operatorname{Ann}_{Z_s}N_s^{\mathrm{gen}}}=I_s,
 \qquad
 \operatorname{length}_{(Z_s)_{I_s}}
       (N_s^{\mathrm{gen}})_{I_s}=p^2.
\end{equation}
Equivalently, after restriction to the generic point of
$U_s^{(1)}$, the reduced $p$-support of $N_s$ is the Frobenius twist
of the graph $t=T_s(e)$, $b=H_s(e)$, and its normalized
multiplicity is one.

Then there exist $P,Q\in K(v)[\partial_v]$ such that
\[
 [P,Q]=1,\qquad
 \operatorname{ord}_{\partial_v}P=\deg_e T,\qquad
 \operatorname{ord}_{\partial_v}Q=\deg_e H.
\]
\end{lemma}

\begin{proof}
We first show that $N$ is finitely generated over $D_K$.
Give $B_R$ weight zero, $\partial_v$ weight one, and
$t,\partial_t$ a common integral weight
$d>\max(1,\deg_e T,\deg_e H)$. Choose a good filtration on $N_R$
for this filtration of $\mathcal A_R$. Its associated graded module
$G$ is finite over
\[
 \operatorname{gr}\mathcal A_R=B_R[e,u,w],
\]
where $e,u,w$ are the principal symbols of
$\partial_v,t,\partial_t$, respectively. By generic freeness,
after inverting a nonzero element of $R$, formation of the associated
graded module commutes with reduction to every residue field.

For each $s\in S$, the first condition in
\eqref{as:eq:generic-support} implies that a power of each generator
of $I_s$ annihilates $N_s^{\mathrm{gen}}$. Their principal symbols
are $u^p$ and $w^p$, because of the choice of $d$.
Consequently $u$ and $w$ act nilpotently on
$k(s)(v)\otimes_{B_R}G$.

They also act nilpotently on $K(v)\otimes_{B_R}G$. Indeed, otherwise
one of the sets
\[
 \operatorname{Supp}(G)\cap\{u\ne0\},\qquad
 \operatorname{Supp}(G)\cap\{w\ne0\}
\]
would have image in $U_R$ containing its generic point. By Chevalley's
theorem, that image would contain a nonempty open subset of $U_R$.
After inverting a further nonzero element of $R$, this open subset
contains the generic point of every fiber of $U_R\to\Spec R$.
Compatibility of the associated graded module with reduction would
then contradict the preceding nilpotence for every remaining
$s\in S$. Thus $\operatorname{gr}N$ is finite over $K(v)[e]$.
Lifting finitely many homogeneous generators proves that $N$ is
finite over $D_K$.

The ring $D_K$ admits left and right division with remainder by
differential order. Hence a finitely generated torsion $D_K$-module
is finite-dimensional over $K(v)$, and a finitely generated torsion-free
$D_K$-module is free. The torsion submodule of $N$ is preserved by
the $D_K$-linear actions of $t$ and $\partial_t$.
Taking the $K(v)$-linear trace of $[\partial_t,t]=1$ on this torsion
submodule shows that it is zero. Therefore $N\simeq D_K^m$.

After replacing $R$ by $R[a^{-1}]$ for some $a\ne0$ and $h$ by
$hg$ for a suitable nonzero $g\in R[v]$, we may assume that
$N_R\simeq D_R^m$ as left $D_R$-modules. To justify this, first
clear the denominators in the nilpotence identities for $u$ and $w$
on the finitely generated module $G$; this makes $N_R$ finite over
$D_R$. Then clear the finitely many coefficient denominators in a
$D_K$-basis, its inverse isomorphism, and their defining identities,
as in Lemma~\ref{lem:finite-descent}. We may also arrange that
$g_s\ne0$ and that $T_s,H_s$ retain their degrees for every
remaining $s$. In particular,
$N_s^{\mathrm{gen}}\simeq D_s^m$ for every such $s$.

Fix $s\in S$ in this open subset and put
\[
 Z_s^0=k(s)(v^p)[\partial_v^p],\qquad
 E_s=\operatorname{Frac}(Z_s^0),\qquad
 V_s=E_s\otimes_{Z_s^0}N_s^{\mathrm{gen}}.
\]
The algebra $D_s$ is free of rank $p^2$ over $Z_s^0$, so
$\dim_{E_s}V_s=mp^2$. On the other hand,
$Z_s/I_s\simeq Z_s^0$. Thus, as an
$E_s[t^p,\partial_t^p]$-module, $V_s$ is supported at the single
$E_s$-rational point defined by $I_s$. Its dimension over $E_s$
equals the length in \eqref{as:eq:generic-support}, namely $p^2$.
It follows that $m=1$.

Choose a $D_K$-basis of $N$. Under the resulting identification
$N\simeq D_K$, the actions of $t$ and $\partial_t$ are right
multiplication by operators $P,Q\in D_K$, respectively. The identity
$[\partial_t,t]=1$ gives $PQ-QP=1$. After further localization,
these operators and the chosen identification specialize to the
corresponding actions on $N_s^{\mathrm{gen}}\simeq D_s$.
By the first condition in \eqref{as:eq:generic-support}, powers of
\[
 P_s^p-T_s^{[p]}(\partial_v^p),\qquad
 Q_s^p-H_s^{[p]}(\partial_v^p)
\]
vanish in $D_s$. Since $D_s$ is a domain, both elements are zero.

Finally, localize $R$ once more so that the leading coefficients
of $P,Q$ remain nonzero after reduction. For any remaining $s\in S$,
comparison of differential orders gives
\[
 p\operatorname{ord}_{\partial_v}P
 =\operatorname{ord}_{\partial_v}(P_s^p)
 =\operatorname{ord}_{\partial_v}
       T_s^{[p]}(\partial_v^p)
 =p\deg_e T.
\]
The same argument gives
$\operatorname{ord}_{\partial_v}Q=\deg_e H$.
\end{proof}

\begin{proposition}\label{as:prop:counterexample}
The plane $L_2$ does not admit a lifting
$\mathcal D_{\A^2_\C}$-module of multiplicity one.
\end{proposition}
\begin{proof}
Choose an arithmetic model over $R$ of a putative lifting
$\mathcal D_{\A^2_\C}$-module. After restriction to $v\ne0$ and
the preceding transformation, its generic fiber satisfies
Lemma~\ref{as:reduce} with the graph \eqref{as:normal}.
Extending the resulting operators from $\Frac R$ to $\C$ gives
operators of orders two and three in $\C(v)[\partial_v]$ with
commutator one, contrary to Lemma~\ref{lem:orders-two-three}.
\end{proof}

\subsection{A noninvertible symplectic endomorphism of $\mathbb A^4$}\label{subsec:long-map}
We use the following polynomials from \cite{Long}. In this subsection the
generators of the Weyl algebra $A_2$ are $x,q,p,z$, with
$[p,x]=[z,q]=1$ and all other commutators zero. The same letters
denote commuting coordinates on $T^*\mathbb A^2$, with
$\{p,x\}=\{z,q\}=1$ and all other brackets zero.
The following expressions are polynomials in these commuting coordinates. Set
\begin{equation}\label{eq:long-source}
 a_0=1-3xq,\qquad
 \beta=3x^2p+2a_0z-9q^2,\qquad
 y=q-\frac{x\beta}{3},\qquad h=xy.
\end{equation}
Define
\begin{align}
 R&=2x-3x^2y-x^3\beta=x(2-3xq),\label{eq:long-R}\\
 S&=y+3x(1+xy)^2\beta+3xy^2(4+3xy),\label{eq:long-S}\\
 T&=-\frac12\bigl((1+xy)^3\beta+y^2(1+xy)(4+3xy)\bigr),\label{eq:long-T}\\
 D_0&=\frac{1+3xq}{2}p-3q^2z,\qquad D=D_0+H,\label{eq:long-D}\\
 H&=\frac{y^4(18h^2+78h+125)}{20}
 +\frac{3\beta y^2(h^3+5h^2+10h-5)}{10}\notag\\
 &\hspace{18mm}-\frac{\beta^2(9h+2)}6-\frac{x^2\beta^3}6.\label{eq:long-H}
\end{align}
All these polynomials are defined over $\Z[1/30]$.

The symplectic identity
\begin{equation}\label{eq:long-symplectic}
 dx\wedge dp+dq\wedge dz=dR\wedge dD+dT\wedge dS
\end{equation}
is proved in \cite{Long}. Thus $f=(R,T,D,S)$ is symplectic.
It is not invertible: $R=x(2-3xq)$ is a product of two nonunits,
whereas a coordinate of an automorphism generates a prime ideal.

For a field $k$ of characteristic $\ell>0$, write
\[
 \iota_\ell:k[x,q,p,z]\xrightarrow{\sim}Z(A_2(k)),\qquad
 (x,q,p,z)\longmapsto(x^\ell,q^\ell,p^\ell,z^\ell).
\]
As before, $J^{[\ell]}$ means coefficientwise Frobenius applied to a polynomial $J$.

Write $A_\Lambda=A_2(\Lambda)$, where $\Lambda$ is an integral
finitely generated $\Z[1/30]$-algebra of characteristic zero, and put
$F=\Frac\Lambda$. We use the following module statement to prove
both assertions in Theorem~\ref{thm:counterexamples}(i).

\begin{proposition}\label{prop:module-nonlift}
There is no $A_\Lambda$-bimodule $M_\Lambda$, finitely generated as a left $A_\Lambda$-module, whose
generic fiber is projective of left rank one and whose reductions on a
Zariski-dense set of closed points satisfy
\begin{align}
 v\cdot x^\ell&=\iota_\ell(R_s^{[\ell]})v,&
 v\cdot z^\ell&=\iota_\ell(S_s^{[\ell]})v,&
 v\cdot q^\ell&=\iota_\ell(T_s^{[\ell]})v
 \label{eq:nonlift-module-centers}
\end{align}
for every $v\in M_s$.
\end{proposition}

Suppose that the module in Proposition~\ref{prop:module-nonlift} exists.
After localizing $\Lambda$, its finite projective presentation is defined
over $\Lambda$, as in Proposition~\ref{prop:recognition}. Denote left and
right actions by $L_a$ and $R_a$, and put
\[
 G=R_x-L_R.
\]
In the identification $A_2\otimes A_2^{\op}\simeq\mathcal D_{\A^4}$,
this is multiplication by the base function $x_{\rm right}-R(x,q)$.
Since $R$ depends only on position coordinates, the first identity in
\eqref{eq:nonlift-module-centers} gives $G_s^\ell=0$.
Lemma~\ref{lem:generic-detection} shows that $G$ is locally nilpotent on
$M_F$.

Introduce rational source coordinates
\begin{equation}\label{eq:rational-source}
 r=R=x(2-3xq),\quad v=x^{-1},\qquad
 x=v^{-1},\quad q=(2v-rv^2)/3.
\end{equation}
The coordinate derivations in these coordinates are
\begin{equation}\label{eq:coordinate-derivations}
 \partial_r=-\frac{v^2}{3}z,\qquad
 \partial_v=-x^2p-\frac{2(1-3xq)}3z.
\end{equation}
Localize the left coefficient ring and put
\[
 \mathcal D=F(r,v)\langle\partial_r,\partial_v\rangle,
 \qquad \mathcal B=F(r,v)\langle\partial_v\rangle,
 \qquad \widetilde M=\mathcal D\otimes_{A_2(F)}M_F.
\]
These Ore localizations are flat. The module $\widetilde M$ remains
torsion-free, finite projective, and of left rank one. The right action
extends to it. Since $G$ commutes with multiplication by rational functions in the source coordinates,
it remains locally nilpotent, and
\begin{equation}\label{eq:transverse-commutator}
 [G,\mathcal B]=0,\qquad [G,\partial_r]=1.
\end{equation}

\begin{lemma}[Finite Taylor expansion]\label{lem:taylor-expansion}
Let $V$ be a vector space over a field of characteristic zero, and let
$t,\partial\in\End(V)$ satisfy $[t,\partial]=-1$, with $t$ locally
nilpotent. Then
\[
 \Pi(v)=\sum_{j\geq0}\frac{\partial^jt^jv}{j!}
\]
is a projection onto $\ker t$, and every $v\in V$ has the unique finite
expansion
\begin{equation}\label{eq:taylor-general}
 v=\sum_{j\geq0}\frac{(-1)^j}{j!}\partial^j\Pi(t^jv).
\end{equation}
\end{lemma}
\begin{proof}
All sums are finite. The identity
$t\partial^j=\partial^jt-j\partial^{j-1}$ gives $t\Pi=0$ by
cancellation, and $\Pi$ is the identity on $\ker t$. Substituting its
formula into \eqref{eq:taylor-general}, the coefficient of
$\partial^nt^nv$ is
$\sum_{j=0}^n(-1)^j/(j!(n-j)!)$, which is zero for $n>0$ and one
for $n=0$. If $\sum_{j=0}^d\partial^jv_j=0$ with $v_j\in\ker t$,
applying $t^d$ gives $(-1)^dd!v_d=0$; descending induction proves
uniqueness.
\end{proof}

\begin{lemma}\label{lem:finite-taylor}
Let $\widetilde M$, $\mathcal D$, $\mathcal B$, and $G=R_x-L_R$ be
as constructed above from a hypothetical module in
Proposition~\ref{prop:module-nonlift}. There is an element
$e\in\ker G$ such that $\widetilde M$ is a free left
$\mathcal D$-module with basis $e$, and $\ker G$ is a free left
$\mathcal B$-module with the same basis.
\end{lemma}
\begin{proof}
Put $N=\ker G$. Apply Lemma~\ref{lem:taylor-expansion}
with $t=-G$ and $\partial=\partial_r$, since
$[-G,\partial_r]=-1$. The projection and expansion become
\[
 \Pi(w)=\sum_{j\geq0}\frac{(-1)^j}{j!}\partial_r^jG^j(w),
\]
\begin{equation}\label{eq:finite-taylor}
 w=\sum_{j\geq0}\frac{\partial_r^j}{j!}\Pi(G^j(w)).
\end{equation}
Uniqueness of this expansion identifies the natural map
$\mathcal D\otimes_{\mathcal B}N\to\widetilde M$ as an isomorphism.

Apply \eqref{eq:finite-taylor} to a finite set of $\mathcal D$-generators.
Their finitely many Taylor coefficients generate $N$ over $\mathcal B$:
write every operator in the form $\sum_j\partial_r^j b_j$ and use
uniqueness. Torsion-freeness is inherited from $\widetilde M$. Passing to the
division rings in the displayed isomorphism shows that $N$ has
$\mathcal B$-rank one. The ring $\mathcal B$ of differential operators in one
derivation over $F(r,v)$ admits division with remainder, with respect to order,
on both sides. Realize $N$ inside the division ring of $\mathcal B$ and clear a common
right denominator of its finite generators. This identifies $N$ with a
nonzero left ideal, which is principal by the division algorithm.
Thus $N=\mathcal B e$ and $\widetilde M=\mathcal D e$.
\end{proof}

The right actions of $z,q$ preserve $N$ and commute with all left actions.
There are therefore scalar operators $\widehat S,\widehat T\in\mathcal B$
with
\begin{equation}\label{eq:scalar-actions}
 e\cdot z=\widehat S e,\qquad e\cdot q=\widehat T e,
 \qquad[\widehat S,\widehat T]=1.
\end{equation}
The sign follows directly from
$(e\cdot z)\cdot q=\widehat S\widehat T e$.

\begin{lemma}\label{lem:spread-scalar-basis}
There are a localization $\Lambda'=\Lambda[a^{-1}]$, with $a\ne0$,
and a polynomial $h\in\Lambda'[r,v]$ divisible by $v$, with
$h_s\ne0$ for every $s\in\Spec\Lambda'$, such that the following holds.
Put
\[
 O=\Lambda'[r,v,h^{-1}],\qquad
 \mathcal D_O=O\langle\partial_r,\partial_v\rangle,\qquad
 \mathcal B_O=O\langle\partial_v\rangle,
\]
and, using \eqref{eq:rational-source}, set
\[
 M_O=\mathcal D_O\otimes_{A_2(\Lambda')}
       (M_\Lambda\otimes_\Lambda\Lambda').
\]
The basis isomorphism $\mathcal D\to\widetilde M$, $b\mapsto be$,
of Lemma~\ref{lem:finite-taylor} and its inverse descend to $O$.
The descended basis $e_O$ satisfies $G e_O=0$, and
$\widehat S,\widehat T$ descend to operators in $\mathcal B_O$
satisfying \eqref{eq:scalar-actions} with $e_O$.

For every $s\in\Spec\Lambda'$, base change and localization to
$k(s)(r,v)$ preserve these inverse and action identities. Write
\[
 \mathcal D_s=k(s)(r,v)\langle\partial_r,\partial_v\rangle,\qquad
 \widetilde M_s=\mathcal D_s\otimes_{A_2(k(s))}M_s.
\]
The specialization $e_s$ is a free $\mathcal D_s$-basis of
$\widetilde M_s$, and the specialized operators
$\widehat S_s,\widehat T_s$ satisfy
\eqref{eq:scalar-actions} with $e_s$.
\end{lemma}
\begin{proof}
Represent $M_\Lambda$ by a finite idempotent matrix. The basis
isomorphism, its inverse, and $\widehat S,\widehat T$ involve only
finitely many differential operators with coefficients in $F(r,v)$.
Clear their arithmetic and polynomial denominators, including $v$
as a factor of $h$ so that \eqref{eq:rational-source} is defined over
$O$. This ring is stable under $\partial_r$ and $\partial_v$.
By Lemma~\ref{lem:finite-descent}, a further arithmetic localization
makes the two inverse identities, $G e_O=0$, and
\eqref{eq:scalar-actions} hold over $O$. Invert a nonzero coefficient
of $h$ to ensure that $h_s\ne0$ in every remaining fiber.
Base change and localization then preserve all these identities,
including both inverse identities, so $e_s$ remains a free basis.
\end{proof}

Put
\[
 t=v^2/2,\qquad \delta=v^{-1}\partial_v,\qquad [\delta,t]=1.
\]
We retain the coefficient field $F(r,v)$; it is a quadratic extension of
$F(r,t)$. If $\xi_v$ is the cotangent coordinate for $\partial_v$ and
$\eta=\xi_v/v$, equations~\eqref{eq:long-source} and
\eqref{eq:coordinate-derivations} give
\[
 \beta=-3\xi_v-9q^2,\qquad y=q+(\xi_v+3q^2)/v.
\]
Define
\begin{equation}\label{eq:local-shift}
 c(r,v)=v+q+3q^2/v
       =3v-\frac53rv^2+\frac13r^2v^3,\qquad a=\eta+c(r,v)=v+y.
\end{equation}
Substitution in \eqref{eq:long-S}--\eqref{eq:long-T} gives
\begin{equation}\label{eq:local-ST}
 S=2v+4a-3ra^2,\qquad
 T=\frac12(ra^3-a^2-va).
\end{equation}
For example, setting $w=1+xy$ and $\alpha=2-3xy-x^2\beta$ first gives
$S=x^{-1}(2+4w-3\alpha w^2)$ and
$T=x^{-2}(\alpha w^3-w^2-w)/2$; now $w=a/v$ and $\alpha=rv$.

These changes of cotangent coordinates induce the following identities
in the center in characteristic $\ell>3$:

\begin{align}
 \partial_r^\ell&=(-v^2/3)^\ell z^\ell,\notag\\
 \partial_v^\ell&=(-x^2)^\ell p^\ell
             +\bigl(-2(1-3xq)/3\bigr)^\ell z^\ell,\label{eq:exact-center-change}\\
 \delta^\ell&=v^{-\ell}\partial_v^\ell.\notag
\end{align}
Indeed, $\partial_r,\partial_v$ have zero $\ell$th iterate as derivations
of $k(s)(r,v)$. Their crystalline $\ell$th powers are central. The proposed
right sides are central with the same order-$\ell$ symbols. The differences
are central operators of order less than $\ell$, hence multiplication
operators, and act as zero on the rational function field. They vanish.
For $\delta$, its $\ell$th iterate vanishes on $k(s)(r,t)$, and hence on
the separable quadratic extension $k(s)(r,v)$. The same argument applies.

On the fibers where Lemma~\ref{lem:spread-scalar-basis} applies, set
$b_\ell=\delta^\ell+c(r,v)^\ell$. Evaluation of
\eqref{eq:nonlift-module-centers} on the free basis $e_s$, followed by
\eqref{eq:local-ST} and \eqref{eq:exact-center-change}, gives
\begin{align}
 \widehat S_s^{\,\ell}&=2v^\ell+4b_\ell-3r^\ell b_\ell^2,\label{eq:S-center-order}\\
 \widehat T_s^{\,\ell}&=\frac12(r^\ell b_\ell^3-b_\ell^2-v^\ell b_\ell).
 \label{eq:T-center-order}
\end{align}
For a nonzero scalar differential operator $P$, let $\ord(P)$
and $\operatorname{lc}(P)$ denote its order and leading coefficient
with respect to $\delta$. Orders are additive and leading
coefficients are multiplicative under products. Consequently,
\[
 \ord(\widehat S_s)=2,\quad\operatorname{lc}(\widehat S_s)=-3r,
 \qquad
 \ord(\widehat T_s)=3,\quad\operatorname{lc}(\widehat T_s)=r/2.
\]
These assertions descend to characteristic zero. Each operator has only
finitely many rational coefficients; after localizing the arithmetic base,
its highest nonzero coefficient remains nonzero in every fiber. The orders
are then constant, and the leading-coefficient identities on the dense
set imply the corresponding rational-function identities over $F$.
Thus
\begin{equation}\label{eq:forced-orders}
 \widehat S=-3r\delta^2+\text{lower-order terms},\qquad
 \widehat T=(r/2)\delta^3+\text{lower-order terms}.
\end{equation}

\begin{proof}[Proof of Proposition~\ref{prop:module-nonlift}]
Take $K=F(r)$. By \eqref{eq:scalar-actions} and
\eqref{eq:forced-orders}, the operators
$\widehat S,\widehat T\in F(r,v)[\delta]$ have orders two and
three and satisfy $[\widehat S,\widehat T]=1$, contrary to
Lemma~\ref{lem:orders-two-three}.
\end{proof}

\begin{proof}[Proof of Theorem~\ref{thm:counterexamples}(i)]
Suppose first that an endomorphism $u$ as in
\eqref{eq:intro-matrix-lift} exists. Choose its model over $\Lambda$,
put $C_\Lambda=\Mat_r(A_\Lambda)$, and let $e=e_{11}$. The module
\[
 M_\Lambda=eC_\Lambda u(e),\qquad b\cdot a=b\,u(ae)
\]
is an $A_\Lambda$-bimodule. It is an idempotent summand of the free
left module $eC_\Lambda$, so it is finite projective and its formation
commutes with base change. The $r$ summands $eC_Fu(e_{ii})$ are isomorphic
via the images of the matrix units and sum to a free module of rank
$r$; hence each has left rank one. The prescribed action on centers
gives \eqref{eq:nonlift-module-centers}, contradicting
Proposition~\ref{prop:module-nonlift}.

Now suppose $\Gamma_f$ admits a lifting $\mathcal D_{\A^4}$-module
of multiplicity one. Proposition~\ref{prop:recognition}
makes its arithmetic model finitely generated as a left $A_2$-module after localization,
with generic fiber projective of left rank one. Its reductions
satisfy the central graph equations, including
\eqref{eq:nonlift-module-centers}. Proposition~\ref{prop:module-nonlift}
again gives a contradiction.
\end{proof}

\subsection{Products with fibers of cotangent bundles}
\begin{proof}[Completion of the proof of Theorem~\ref{thm:counterexamples}(ii)]
The case $m=2$ is Proposition~\ref{as:prop:counterexample}. For $m>2$,
put $r=m-2$ and
\begin{equation}\label{eq:higher-dimensional-L}
 L_m=L_2\times T^*_{\{0\}}\A^r
 \subset T^*\A^2\times T^*\A^r=T^*\A^m.
\end{equation}
In coordinates $(t_1,\ldots,t_r;\tau_1,\ldots,\tau_r)$ on $T^*\A^r$,
the second factor is defined by $t_1=\cdots=t_r=0$.
Thus $L_m$ is smooth, closed, and isomorphic to $\A^m$.
The Liouville form vanishes on the second factor, so $L_m$
is exact and Lagrangian.

Suppose $L_m$ admits a lifting $\mathcal D_{\A^m}$-module $M$
of multiplicity one. Choose an arithmetic model over an integral finitely generated
$\Z[1/30]$-subalgebra $R\subset\C$ and put $F=\Frac R$.
After localizing $R$, the central equations give $t_i^pM_s=0$ in
every closed fiber of characteristic $p$.
Lemma~\ref{lem:generic-detection} implies that each $t_i$ acts locally
nilpotently on $M_F$. Hence $M_F$ is supported on
$\A^2\times\{0\}\subset\A^m$.

By Kashiwara's equivalence \cite[Theorem~1.6.1]{HTT},
\begin{equation}\label{eq:conormal-module-decomposition}
 M_F\simeq N_F\boxtimes\Delta_r(F),\qquad
 \Delta_r(F)=A_r(F)\Big/\sum_{i=1}^rA_r(F)t_i,
\end{equation}
for a nonzero holonomic $A_2(F)$-module $N_F$.
In coordinates, $N_F=\bigcap_i\ker t_i$, and the isomorphism sends
$n\otimes\partial_t^\alpha$ to $\partial_t^\alpha n$.
By Lemma~\ref{lem:finite-descent}, a finite presentation of $N_F$
and the two inverse maps in \eqref{eq:conormal-module-decomposition}
spread out after localizing $R$. Thus
$M_R\simeq N_R\boxtimes\Delta_r(R)$, compatibly with reduction.

Over $k(s)$, the module $\Delta_r(k(s))$ has central annihilator
$(t_1^p,\ldots,t_r^p)$ and is free of rank $p^r$ over
$k(s)[\partial_{t_1}^p,\ldots,\partial_{t_r}^p]$. Consequently, writing
$Z_{j,s}=Z(A_j(k(s)))$, we have
\[
 \operatorname{Ann}_{Z_{m,s}}M_s
 =\bigl(\operatorname{Ann}_{Z_{2,s}}N_s\bigr)Z_{m,s}
  +(t_1^p,\ldots,t_r^p).
\]
Contraction to $Z_{2,s}$ gives
$\operatorname{Supp}_p(N_s)=L_{2,s}^{(1)}$ as closed subschemes.
Tensoring multiplies the generic central rank by $p^r$, while the
normalizing factor changes from $p^2$ to $p^{2+r}$. Hence the normalized
multiplicity of $N_s$ is one. The module $N_F\otimes_F\C$ contradicts
Proposition~\ref{as:prop:counterexample}, completing the proof.
\end{proof}

\section{Arithmetic supports of \texorpdfstring{$\mathcal D$}{D}-modules on curves}\label{sec:curve-supports}

\subsection{Kontsevich's conjecture in dimension one and some examples}\label{subsec:dodd}
\begin{proposition}\label{prop:affine-curve-kontsevich}
Conjecture~\ref{conj:kont} holds when $X$ is a smooth affine complex curve.
\end{proposition}
\begin{proof}
We may restrict to the connected component of $X$ met by $L$.
The closed curve $L\subset T^*X$ is affine. If its smooth projective
completion has genus $g$ and its complement has $b$ points, then
$\operatorname{rank}H_1(L(\C),\Z)=2g+b-1$. Since $b\geq1$, the
vanishing hypothesis gives $g=0$, $b=1$, and $L\simeq\A^1$.

Let $\pi:T^*X\to X$ be the projection.
Suppose first that $\pi|_L$ is constant, with $\pi(L) = a\in X$.
Then $L=T_a^*X$. The delta module $\delta_a$ is a lifting $\mathcal D_X$-module
for $T_a^*X$ of multiplicity one, as is seen in an
\'etale coordinate at $a$ from the presentation $A_1/A_1x$.
For any other candidate, its restriction to $X\setminus\{a\}$ has
zero reductions on a dense open arithmetic base. The good-order-filtration
argument of Lemma~\ref{lem:generic-detection}, applied on affine
coordinate neighborhoods, makes that restriction zero in characteristic
zero. Kashiwara's equivalence for $\{a\}\hookrightarrow X$
\cite[Theorem~1.6.1]{HTT} therefore identifies the candidate with
$\delta_a\otimes_\C V$ for a finite-dimensional vector space $V$.
Its arithmetic multiplicity is $\dim V$, so multiplicity one gives
$\delta_a$. The same equivalence gives
$\Ext^1(\delta_a,\delta_a)=0$: extensions of modules supported at $a$
are still supported there, and finite-dimensional vector spaces have
no nontrivial extensions.

Now suppose that $\pi|_L$ is nonconstant. The induced map from
$\PP^1$ to the smooth projective completion of $X$ forces that
completion to have genus zero, by the Riemann--Hurwitz formula.
Every point of its nonempty boundary has its entire inverse image
at infinity in $L\simeq\A^1$, so the boundary has exactly one point.
Thus $X\simeq\A^1$.
The Abhyankar--Moh theorem \cite{AM} sends $L\subset\A^2$ to the
coordinate line $x=0$ by an automorphism. Rescaling one target coordinate
makes its Jacobian one while preserving that line.

Write $g$ for this symplectic automorphism of $\A^2$.
It lifts to an automorphism $\varphi$ of $A_1$ whose reductions
act on the center by $(g_s^{(1)})^*$, after enlarging the
coefficient ring and excluding finitely many characteristics
\cite[\S1.1, Theorem~1, and Propositions~2--3]{BKKAut}.
Then $M=A_1/A_1\varphi(x)$ is the required $\mathcal D$-module.
It is unique and satisfies $\operatorname{Ext}_{A_1}^1(M,M)=0$, since
$A_1/A_1x$ is the unique lifting $\mathcal D_{\A^1}$-module
for the line $x=0$ of multiplicity one (by the discussion in the preceding paragraphs)
and $\operatorname{Ext}_{A_1}^1(A_1/A_1x,A_1/A_1x)=0$.
\end{proof}

It is not hard to reduce the remaining cases of Kontsevich's
conjecture in dimension one to the case where
$X=\mathbb P^1$, $L\simeq\A^1$, and the projection
$\pi|_L:L\to X$ is surjective.

The genus of a smooth affine curve means the genus of its smooth
projective completion.

\begin{proposition}\label{prop:arbitrary-genus}
For every integer $g\geq1$, put $r=2g+1$ and
\begin{equation}\label{eq:genus-operator}
 a_r(x)=x^{r+2}(1+x^r),\qquad
 P_r=a_r(x)\partial^2+\tfrac12a'_r(x)\partial-1
 \in A_1(\Q),\qquad [\partial,x]=1.
\end{equation}
The left $A_1$-module $M_r=A_1/A_1P_r$ is a lifting
$\mathcal D_{\A^1}$-module for the curve
\begin{equation}\label{eq:genus-curve}
 C_r=\{a_r(x)y^2=1\}\subset T^*\A^1
\end{equation}
with multiplicity one. The curve $C_r$ is smooth, geometrically
irreducible, and has genus $g$. The restriction of the Liouville form $y\,dx$ to $C_r$ is exact.
\end{proposition}
\begin{proof}
Set $R=\Z[1/(2r)]$. The curve $C_r$ admits the model
\[
 C_{r,R}=\Spec R[x,y]/(a_r(x)y^2-1)\subset T^*\A^1_R.
\]
Both $x$ and $y$ are invertible on $C_{r,R}$. The substitution
\[
 u=x^{-1},\qquad v=(x^{r+1}y)^{-1}
\]
identifies $C_{r,R}$ with the open subscheme $uv\ne0$ of $v^2=1+u^r$ over $R$. Its inverse is $x=u^{-1}$, $y=u^{r+1}/v$. Since $2r$ is invertible in $R$, the polynomial $1+u^r$ is square-free in every fiber. The hyperelliptic model, and hence $C_{r,R}$, is smooth over $R$, with geometrically irreducible fibers whose smooth projective completions have genus $(r-1)/2=g$.
Moreover,
\begin{equation}\label{eq:genus-exactness}
 y\,dx=-\frac{u^{r-1}}v\,du=d\left(-\frac{2v}{r}\right).
\end{equation}
Equation~\eqref{eq:genus-exactness} is an identity of relative one-forms on this model. For a point $s\in\Spec R$, write $C_{r,s}=C_{r,R}\times_R k(s)$.

We next compute a central equation annihilating the reductions of $M_r$.
Let $s\in\Spec R$ have residue characteristic $p>0$, and let $k$ be a field extension of $k(s)$. Write $a$ for the image of $a_r$ in $k[x]$,
and extend $d/dx$ to the separable quadratic extension
$K=k(x,w)$, where $w^2=a(x)$. In the ring $K\langle\partial\rangle$
of differential operators, set $D=w\partial$. Then
\[
 D^2=a\partial^2+\tfrac12a'\partial,
 \qquad P_r=D^2-1.
\]
The derivation $\delta=w\,d/dx$ of $K$ satisfies
\[
 \delta\left(-\frac{2w}{r x^{r+1}}\right)=1.
\]
Thus, for $S=-2w/(r x^{r+1})$, the identity $\delta(S)=1$
implies that the extension $K/k(S)$ is finite and separable. The derivation $\delta$ extends $d/dS$;
its $p$th iterate vanishes because it is a derivation that is zero
on $k(S)$ and $K/k(S)$ is separable. Likewise $(d/dx)^p=0$ on $K$.
The $p$th-power identity for a vector field and its corresponding
differential operator therefore gives
\[
 D^p=w^p\partial^p.
\]
It follows that
\begin{equation}\label{eq:genus-p-power}
 P_r^p=(D^2-1)^p=a(x)^p\partial^{2p}-1
       =a(x^p)\partial^{2p}-1.
\end{equation}
The injection $A_1(k)\hookrightarrow K\langle\partial\rangle$
shows that this is an identity in $A_1(k)$ itself.

To identify the scheme-theoretic support and its multiplicity, use the reduced norm.
Put $X=x^p$, $Y=\partial^p$, and $F=a(X)Y^2-1$.
Since $A_1(k)$ is Azumaya of degree $p$ over its center,
\[
 \Nrd(P_r)^p=\Nrd(P_r^p)=F^p.
\]
Since $k[X,Y]$ is reduced, $\Nrd(P_r)=F$.
After an \'etale base change splitting the Azumaya algebra,
$P_r$ is represented by a $p\times p$ matrix with determinant $F$.
At a point of the smooth hypersurface $F=0$, this matrix has rank
exactly $p-1$: rank at most $p-2$ would make all its cofactors, and
hence all first derivatives of its determinant, vanish there.
Consequently its cokernel is a line bundle on $F=0$ under the local
Morita equivalence. The corresponding quotient of the Weyl algebra
is a splitting module, locally free of rank $p$ over $F=0$.
Its central annihilator is exactly $(F)$, and its normalized multiplicity is one.

Define $M_{r,R}=A_1(R)/A_1(R)P_r$. The preceding calculation identifies the scheme-theoretic $p$-support of $M_{r,s}=M_{r,R}\otimes_R k(s)$ with $C_{r,s}^{(1)}$ at every positive-characteristic point $s\in\Spec R$.
\end{proof}

\subsection{Residues of the Liouville form}\label{as:sec:res}
Let $C\subset T^*\A^1$ be a smooth connected closed curve, with
cotangent coordinates $(x,y)$. The restriction $\theta=(y\,dx)|_C$ of
the Liouville form is regular on $C$ and extends to a rational
differential form on its smooth projective completion $\overline C$.
We relate its residues to the splitting of the algebra of crystalline
differential operators in positive characteristic, and then use
Bitoun's splitting theorem to obtain a necessary condition for $C$
to admit a lifting $\mathcal D$-module.

We use the $p$-curvature convention of
Section~\ref{subsec:graph-p}. Let $k$ be a perfect field of
characteristic $p$. Choose a rational basis of a line bundle
with connection on a smooth curve $C$ over $k$, and write
$\nabla=d+\alpha$, where $\alpha\in\Omega^1_{k(C)/k}$.
For a $k$-derivation $D$ of $k(C)$, put $a=\alpha(D)$.
The identity
\[
 (D+a)^p=D^p+a^p+D^{p-1}(a),
\]
as operators on $k(C)$, gives
\[
 \psi_\nabla(D)
   =\bigl(a^p+D^{p-1}(a)-\alpha(D^{[p]})\bigr)\id;
\]
see \cite[\S5.2]{Bitoun}. Here $D^{p-1}(a)$ is repeated
differentiation of the scalar $a$, whereas $D^{[p]}$ is the
$p$-fold iterate of $D$ as a derivation.
With this convention, the descended $p$-curvature form equals
$(y\,dx)^{(1)}$ if and only if
\[
 \psi_\nabla(D)=\bigl(yD(x)\bigr)^p\id
 \qquad\text{for every }D\in\operatorname{Der}_k(k(C)).
\]

\begin{proposition}\label{as:prop:finite}
Let $C\subset T^*\A^1_k$ be a smooth geometrically connected closed
curve over a finite field $k$ of characteristic $p$, and let
$\overline C$ be its smooth projective completion. Put $K=k(C)$.
Assume that the coordinate function $x|_C$ is nonconstant and that
the extension $K/k(x)$ is finite and separable. The restriction of
the crystalline differential-operator algebra to $C^{(1)}$ splits
if and only if, for every closed point $q\in\overline C$,
\begin{equation}\label{as:eq:trace}
 \Tr_{k(q)/\F_p}\bigl(\Res_q(y\,dx)\bigr)=0 .
\end{equation}
\end{proposition}
\begin{proof}
Identify the function field of $C^{(1)}$ with $K^p$, sending the
central coordinates $X,Y$ to $x^p,y^p$. Since $k$ is perfect and
$K$ is a function field of one variable, $[K:K^p]=p$. The separability
of $K/k(x)$ gives $x\notin K^p$, hence $K=K^p(x)$.
The derivation $d/dx$ of $k(x)$ extends uniquely to a derivation
$D_x$ of $K$. Its $p$th iterate is zero, since it is a derivation
vanishing on $k(x)$ and $K/k(x)$ is separable.

Let $\mathcal A$ denote the restriction of the crystalline
differential-operator algebra to $C^{(1)}$, and let $A$ be its generic
fiber, a central simple algebra of degree $p$ over $K^p$. Write
$\widehat x,\widehat\partial$ for the images in $A$ of the Weyl algebra
generators. The algebra $A$ splits over $K^p$ if and only if it admits
a left module $V$ of dimension $p$ over $K^p$. Since
\[
 K^p[\widehat x]\simeq K^p[u]/(u^p-x^p)\simeq K,
\]
the action of $\widehat x$ equips $V$ with the structure of a
one-dimensional $K$-vector space. The action of $\widehat\partial$
defines a connection on $V$, with
\[
 \nabla_{D_x}(fv)=D_x(f)v+f\nabla_{D_x}(v),\qquad
 \psi(D_x)=y^p\id_V.
\]
Conversely, a one-dimensional $K$-vector space with such a connection
becomes an $A$-module by letting $\widehat x$ act by multiplication by
$x$ and $\widehat\partial$ act by $\nabla_{D_x}$.

Fix a closed point $q\in\overline C$, choose a uniformizer $t$ at $q$,
and identify the completion of $K$ at $q$ with $L=k(q)((t))$.
The corresponding completion of $K^p$ is $L^p$. The nonzero rational
differential form $dx$ has a nonzero Laurent expansion in $L\,dt$;
hence $x\notin L^p$ and $L=L^p(x)$. The preceding argument therefore
applies to
\[
 A_q=A\otimes_{K^p}L^p.
\]

Put $\delta=t\,d/dt$ and $b=y\,\delta(x)$. Then
\[
 \delta^{[p]}=\delta,\qquad
 y\,dx=b\,\frac{dt}{t},\qquad
 \delta=\delta(x)D_x.
\]
By Frobenius linearity of the $p$-curvature, the condition
$\psi(D_x)=y^p\id$ is equivalent to $\psi(\delta)=b^p\id$.
After choosing a basis $e$ of a one-dimensional $L$-vector space,
a connection is determined by an element $a\in L$ through
\[
 \nabla_\delta(fe)=(\delta(f)+af)e.
\]
Its $p$-curvature satisfies
\[
 \psi(\delta)=\bigl(a^p+\delta^{p-1}(a)-a\bigr)\id.
\]
Consequently, $A_q$ splits over $L^p$ if and only if there exists
$a\in L$ satisfying
\begin{equation}\label{as:eq:local}
 a^p+\delta^{p-1}(a)-a=b^p.
\end{equation}

Write
\[
 a=\sum_{n\in\Z}a_nt^n,\qquad
 b=\sum_{n\in\Z}b_nt^n,\qquad a_n,b_n\in k(q).
\]
Since $\delta(t^m)=mt^m$, comparison of the coefficients of $t^{pn}$
in \eqref{as:eq:local} gives
\begin{equation}\label{as:eq:coefs}
 a_n^p-a_{pn}=b_n^p\qquad(n\in\Z).
\end{equation}
At exponents not divisible by $p$, both sides of \eqref{as:eq:local}
have coefficient zero. For $n=0$, equation~\eqref{as:eq:coefs} is an
Artin--Schreier equation over the finite field $k(q)$. It has a solution
if and only if
\[
 \Tr_{k(q)/\F_p}(b_0)=0,\qquad b_0=\Res_q(y\,dx).
\]

There are no further conditions on the coefficients. Indeed, choose
$N<0$ such that $b_n=0$ for every $n<N$, and set $a_n=0$ for those
indices. For $n=N,\ldots,-1$, define recursively
\[
 a_n=b_n+a_{pn}^{1/p}.
\]
The required $p$th roots exist uniquely in $k(q)$. For each positive
integer $m$ not divisible by $p$, choose $a_m$ arbitrarily and then
define
\[
 a_{p^{j+1}m}=a_{p^jm}^{\,p}-b_{p^jm}^{\,p}\qquad(j\geq0).
\]
Together with a solution for $a_0$, these coefficients define an element
$a\in L$ satisfying \eqref{as:eq:local}. Thus $A_q$ splits over $L^p$
if and only if the trace condition \eqref{as:eq:trace} holds at $q$.

By the local-global theorem for the Brauer group of the global function
field $K^p$, the algebra $A$ splits if and only if $A_q$ splits for
every closed point $q\in\overline C$ \cite[Corollary~6.5.4]{GS}.
Moreover, the homomorphism
\[
 \Br(C^{(1)})\longrightarrow\Br(K^p)
\]
is injective because $C^{(1)}$ is a smooth integral curve; see also
\cite[proof of Proposition~5.1.2]{Bitoun}. Hence $\mathcal A$ splits
on $C^{(1)}$ if and only if \eqref{as:eq:trace} holds at every closed
point of $\overline C$. At points of $C$, this condition is automatic
because the differential form $y\,dx$ is regular.

\end{proof}

\begin{theorem}\label{as:thm:res}
Let $C\subset T^*\A^1_\C$ be a smooth connected closed curve, and
put $\theta=(y\,dx)|_C$. Suppose a nonzero holonomic $A_1(\C)$-module
admits an arithmetic model, together with an arithmetic model of $C$,
over a common finitely generated integral $\Z$-subalgebra
$R\subset\C$. Assume that its reduced $p$-support is $C_s^{(1)}$
at every closed point $s$ of a dense open subset of $\Spec R$. Then $\Res_q\theta=0$ at every point of the
smooth projective completion $\overline C$.
If $\overline C\simeq\PP^1$, then $\theta=dS$ for some $S\in\mathcal O(C)$.
\end{theorem}
\begin{proof}
The form $\theta$ is regular on $C$, so only points at infinity need be
considered. If $x|_C$ is constant, then $\theta=0$. Otherwise,
$\C(C)/\C(x)$ is a finite extension. After restricting
the arithmetic base to a dense open subset, the reductions $C_s$
are smooth and geometrically connected, and
$k(s)(C_s)/k(s)(x_s)$ is finite and separable at every closed point $s$.
Here $x_s$ is the reduction of $x|_C$; the asserted separability
follows by spreading out a nonempty open subset on which
$x:C\to\A^1_\C$ is \'etale. By Bitoun's theorem
\cite[Theorem~5.1.4]{Bitoun}, after a further restriction of the base,
the crystalline differential-operator algebra splits on the regular
locus of the reduced $p$-support.
The fibers $C_s^{(1)}$ are smooth, so the algebra splits on all of them.
This conclusion holds for every positive multiplicity.

Suppose $q\in\overline C\setminus C$ has residue $r\ne0$. Enlarge and
localize the ring of definition so that the smooth completion $\bar C$, the marked
point as a section, and the principal part of $\theta$ along that section
are defined, and so that $r$ is a unit. This localized ring has
homomorphisms to $\F_p$ for arbitrarily large primes $p$: choose a closed
point of its characteristic-zero fiber, then use primes splitting
completely in a normal closure of its number field, omitting the finitely
many denominators. This is a consequence of the Chebotarev density theorem
\cite[Chapter~VII, Theorem~13.4]{Neukirch}.
At such a specialization the marked point is $\F_p$-rational and its
residue $r_s$ is nonzero. Splitting and Proposition~\ref{as:prop:finite} give
$\Tr_{\F_p/\F_p}(r_s)=0$, a contradiction.

On $\PP^1$ a rational differential form with zero residues has a rational
primitive. Since $\theta$ is regular
on $C$, the primitive is regular there.
\end{proof}

\begin{corollary}\label{prop:hyperbola-obstruction}
Let $c\in\C^\times$ and $H_c=\{xy=c\}\subset T^*\A^1$.
No nonzero holonomic $A_1(\C)$-module admits an arithmetic model
whose reduced $p$-support equals $H_{c,s}^{(1)}$ at every closed
point of a dense open subset of its arithmetic base.
In particular, $H_c$ does not admit a lifting $\mathcal D_{\A^1}$-module
of any positive multiplicity.
\end{corollary}
\begin{proof}
On $H_c\simeq\mathbb G_m$, the restriction of the Liouville form is $c\,dx/x$.
Its residues at $0$ and $\infty$ are $c$ and $-c$, contradicting
Theorem~\ref{as:thm:res}.
\end{proof}

\begin{example}\label{as:ex:one}
On $C_0:y^2=x^3-1$ one has
\[
 y\,dx=\tfrac25 d(xy)-\tfrac35\,dx/y,
\]
and this form is not exact. For every prime $p>3$, the reduction is smooth and has only one
point at infinity. The residue there is zero by the residue theorem,
so Proposition~\ref{as:prop:finite} gives a splitting in every such
characteristic. Thus these splittings can exist even when the
Liouville form has no rational primitive.
\end{example}

\subsection{A twice-punctured line in $T^*\PP^1$}\label{as:sec:two}
We now give an example that distinguishes agreement of arithmetic
supports on a Zariski-dense set from constancy on an arithmetic open
subset. Use cotangent coordinates $(z,b)$ and
$(r,e)=(z^{-1},-z^2b)$ on $T^*\PP^1$, and put
\[
 f(t)=t^{-2}+(t-1)^{-2},\qquad
 g(t)=t^{-1}-(t-1)^{-1}.
\]
The parametrization
\[
 z=t+f(t),\qquad b=f(t)
\]
extends to a closed embedding of $\PP^1\setminus\{0,1\}$, sending
$t=\infty$ to $(r,e)=(0,-2)$. Denote its image by $L$.
To check these assertions, set
\[
 h(t)=t^2(t-1)^2,\qquad A(t)=h(t)f(t),\qquad J(t)=h(t)g(t).
\]
In the first chart the equation of $L$ is
$F(z,b)=h(z-b)b-A(z-b)=0$, with inverse parameter $t=z-b$.
In the second chart it is
\[
 G(r,e)=r^2F(r^{-1},-r^2e)=0,\qquad G(0,e)=-e-2.
\]
The expression $G$ is polynomial, and $r\sim t^{-1}$ at the added
point. These equations also show that the reductions are smooth for
all sufficiently large characteristics. The restriction of the
Liouville form is exact:
\[
 b\,dz=d\bigl(-t^{-1}-(t-1)^{-1}+f(t)^2/2\bigr).
\]

For an open immersion $j:U\hookrightarrow X$ of smooth complex
curves, write $j_{!*}E$ for the \emph{middle extension} of a
connection $E$ on $U$: this is the unique holonomic extension,
up to isomorphism, having no nonzero submodule or quotient
supported on $X\setminus U$.
In a local coordinate $u$, we use the formal adjoint determined by
\[
 a(u)^*=a(u),\qquad \partial_u^*=-\partial_u,\qquad
 (PQ)^*=Q^*P^*.
\]

\begin{theorem}\label{as:thm:two}
Let $M$ be a holonomic $\mathcal D_{\PP^1_\C}$-module with an arithmetic
model whose normalized $p$-support cycle equals $[L_s^{(1)}]$ at a
Zariski-dense set of closed points of its arithmetic base. Then
$M\simeq M_\lambda$ for some $\lambda\in\C$, uniquely determined modulo $\Z$.
The module $M_\lambda$ is defined on the two affine charts by
\[
 \mathcal D_{\A^1_z}/\mathcal D_{\A^1_z}P_{\lambda,z},\qquad
 \mathcal D_{\A^1_r}/\mathcal D_{\A^1_r}P_{\lambda,r},
\]
where
\[
 S=z-\partial_z,\qquad
 P_{\lambda,z}=h(S)\partial_z-A(S)-\lambda J(S),\qquad
 P_{\lambda,r}=r^2P_{\lambda,z}.
\]
In the last expression one substitutes
$z=r^{-1}$ and $\partial_z=-r^2\partial_r$.
Every $\lambda\in\C$ gives a module with the stated dense-set property.
It is a lifting $\mathcal D_{\PP^1_\C}$-module for $L$ of multiplicity
one if and only if $\lambda\in\Q$.
\end{theorem}
\begin{proof}
\emph{Construction and reductions.}
The operator $P_{\lambda,r}$ has polynomial coefficients in $r$;
substitution of $S=r^{-1}+r^2\partial_r$ verifies that multiplication
by $r^2$ clears its poles. Since $r^2$ is invertible on the overlap,
the two quotient modules glue. The Weyl-algebra isomorphism
$t\mapsto z-\partial_z$, $\partial_t\mapsto\partial_z$ identifies
the first quotient, on $V=\A^1_t\setminus\{0,1\}$, with the rank-one
connection
\[
 \partial_t m=(f+\lambda g)m.
\]
For an arithmetic specialization of sufficiently large characteristic
$p$, write $Z=z^p$, $B=\partial_z^p$, and $T=Z-B$. The formula for
$p$-curvature gives
\begin{equation}\label{eq:two-punctures-norm}
 \Nrd(P_{\lambda,z})
 =h(T)B-A(T)-(\bar\lambda^p-\bar\lambda)J(T).
\end{equation}
Indeed, $\partial_t^{p-1}f=0$ and
$\partial_t^{p-1}g=-g^p$. If $\bar\lambda\notin\F_p$, the support
differs from $L_s^{(1)}$ already on $h\ne0$. If
$\bar\lambda\in\F_p$, the reduced norms in the two charts are
exactly the defining equations of $L_s^{(1)}$. After an \'etale
splitting of the Azumaya algebra, a matrix whose determinant defines
a smooth reduced divisor has corank one along that divisor. The
quotient is therefore a splitting bundle on $L_s^{(1)}$.

For algebraic $\lambda$, primes splitting completely in a normal
closure of $\Q(\lambda)$ give a dense set of arithmetic specializations
with $\bar\lambda\in\F_p$, after excluding denominators
\cite[Chapter~VII, Theorem~13.4]{Neukirch}. For transcendental $\lambda$, use
$\Z[1/2,\lambda]\simeq\Z[1/2,u]$. The points $(p,u-a)$ with
$a\in\F_p$ form a dense set: a nonzero polynomial cannot vanish at
every element of $\F_p$ once $p$ exceeds its degree and its reduction
is nonzero. This proves the claimed existence for every parameter.

\emph{Classification.}
Any module satisfying the hypothesis is simple. Indeed, a proper nonzero submodule and its quotient can be spread out over a common arithmetic base so that their reductions remain nonzero and exact after restricting to a dense open subset. After shrinking the arithmetic base further, both reductions
are pure of dimension one as coherent sheaves on the cotangent
bundle of the Frobenius twist
\cite[proof of Theorem~3.1.1]{Bitoun}.
Their normalized $p$-support cycles are nonzero effective
integral cycles. Additivity would express \([L_s^{(1)}]\) as the sum of these two cycles, which is impossible.

Restrict to the $z$-chart, apply the inverse Weyl-algebra isomorphism,
and then restrict to $V$. The resulting module is a line bundle with
connection on all of $V$. To see coherence over $\mathcal O_V$, choose
a good order filtration compatible with arithmetic base change.
On each specified fiber a power of $\partial_t^p-f^p$ annihilates
the module, so the graded support lies in the zero section. The
constructibility argument used in Lemma~\ref{as:reduce} gives the
same conclusion in characteristic zero. Coherence over $\mathcal O_V$
follows; a coherent module with a connection on a smooth curve in
characteristic zero is locally free. Its rank is one, since its
Frobenius pushforward has generic rank $p$ on the prescribed graph.

Since $\operatorname{Pic}(V)=0$, write the connection as
$\partial_t m=a(t)m$. Only the simple-pole terms of $a$ contribute
to $\partial_t^{p-1}a$ for all sufficiently large $p$.
Comparing $a_s^p+\partial_t^{p-1}a_s=f_s^p$ on the specified dense
set and comparing the remaining coefficients in characteristic zero gives
\[
 a=f+\frac{\alpha}{t}+\frac{\beta}{t-1},\qquad
 \bar\alpha^p=\bar\alpha,\quad \bar\beta^p=\bar\beta
 \quad\text{at the specified specializations}.
\]

We claim that $\sigma=\alpha+\beta$ is an integer. Put
$K(t)=h(t)(\alpha/t+\beta/(t-1))$.
The middle extension of the connection to $\A^1_t$ has
presentation $h\partial_t-A-K$. At each $q\in\{0,1\}$, the
polynomial $h$ has a double zero and $(A+K)(q)=1$.
Thus, in the coordinate $u=t-q$, this operator and its formal
adjoint have the form $-1$ plus terms $u^i\partial_u^j$ with
$i>j$. On the delta module these terms lower the degree in
$\partial_u$, so both operators act by triangular matrices
with diagonal entries $-1$. Hence the quotient has no nonzero
submodule or quotient supported at either point.
After the Weyl-algebra isomorphism, its presentation on the $z$-chart is
$P_z=h(S)\partial_z-A(S)-K(S)$.

Suppose $\sigma\notin\Z$. Then its middle extension across $r=0$
has presentation $Q=r^3P_z$. In fact,
$Q=-r\partial_r-\sigma$ plus monomials $r^i\partial_r^j$ with $i>j$.
On the delta module with basis $\partial_r^n v$, where $rv=0$,
the diagonal entries for $Q$ and its formal adjoint are
$n+1-\sigma$ and $-n-\sigma$, respectively; the other terms lower
degree. None of these entries vanishes. There is therefore no quotient
supported at $r=0$, and duality for holonomic modules likewise excludes a
submodule supported there. On every specified arithmetic fiber the
residue corrections vanish, giving
\[
 \Nrd(Q)=R\,G(R,E),\qquad R=r^p,\quad E=\partial_r^p.
\]
Its $p$-support cycle contains the fiber of the cotangent bundle over
$r=0$, contrary to the hypothesis. This proves $\sigma\in\Z$.

Changing the trivialization by $t^m(t-1)^n$ shifts
$(\alpha,\beta)$ by an integer pair. We can therefore normalize the
residues to $(\lambda,-\lambda)$. Two such connections are isomorphic
exactly when their parameters differ by an integer: their difference
must be the logarithmic derivative of a unit of
$\C[t,t^{-1},(t-1)^{-1}]$. Simplicity and uniqueness of the middle
extension give the same classification on $\PP^1_z$.

\emph{Constancy.}
For rational $\lambda$, its reduction lies in $\F_p$ at every good
arithmetic point, so \eqref{eq:two-punctures-norm} shows that $M_\lambda$
is a lifting $\mathcal D_{\PP^1_\C}$-module for $L$ of multiplicity one. Conversely, constancy for any arithmetic model
would imply $\bar\lambda^p=\bar\lambda$ on a nonempty open subset of
$\Spec\Z[\lambda]$, by constructibility. If $\lambda$ is transcendental,
every such open contains closed points for which
$\bar\lambda\notin\F_p$. If it is algebraic but irrational, there
are infinitely many primes with a residue field in which its reduction
is not in $\F_p$, by the Chebotarev density theorem
\cite[Chapter~VII, Theorem~13.4]{Neukirch}. Both cases contradict constancy of the arithmetic support.
\end{proof}

\appendix
\section{Two more conjectures of Kontsevich}\label{app:two-conjectures}

\subsection{Rigidity and semisimplicity of modules over the first Weyl algebra}
In this appendix, rigidity means the vanishing of the full group
$\Ext^1(M,M)$ in the category of algebraic $\mathcal D$-modules.
This is stronger than rigidity with prescribed formal types at the
singular points. We first recall the statement to which it applies.

In \cite[\S4.3]{Kont} a holonomic $\mathcal D$-module is of \emph{extended motivic-exponential} type if it belongs to the smallest class of holonomic $\mathcal D$-modules on smooth
varieties that contains the exponential connections
$(\mathcal O_X,d+dF)$, for regular functions $F$, and is closed under
extensions, subquotients, and taking the cohomology modules of direct and
inverse images. Conjecture~4 from \cite{Kont} predicts that this class consists precisely of the modules with constant arithmetic support. We consider the following related conjecture.
\begin{conjecture}[Kontsevich, {\cite[Conjecture~6]{Kont}}]\label{conj:kont-six}
Every holonomic $\mathcal D_X$-module $M$ satisfying
$\Ext^1(M,M)=0$ belongs to the extended motivic-exponential class.
\end{conjecture}

We first prove this assertion for $X=\A^1$ without assuming that $M$
is irreducible. The proof uses the index formula and the formal
classification of connections, followed by a rank-reduction argument
of the kind used in \cite{ArinkinRigid}. The rigidity theorems in
\cite{BlochEsnault,ArinkinRigid} concern irreducible connections with
prescribed formal types; \cite{SabbahRigid} proves exponential-geometric
origin under a quasi-unipotence hypothesis. We include the
semisimplicity argument because neither irreducibility nor a
prescribed formal type is assumed in the affine-line result.

Proposition~\ref{thm:rigid-other-curves} proves
Conjecture~\ref{conj:kont-six} on every smooth connected complex
curve other than $\mathbb P^1$.
For irreducible modules on $\mathbb P^1$,
Theorem~\ref{thm:rigid-projective-line} derives quasi-unipotence and
rigidity with prescribed formal types from full self-Ext vanishing;
Sabbah's theorem then applies.

Write $A=A_1$.
For an automorphism $\sigma$ of $A$, twisting a module means
transporting its action along $\sigma$. We use the Fourier
automorphism and the exponential twists
\begin{equation}\label{eq:elementary-Weyl-automorphisms}
 \mathcal F(x)=-\partial,\quad \mathcal F(\partial)=x,
 \qquad
 T_g(x)=x,\quad T_g(\partial)=\partial+g'(x)
 \quad(g\in\C[x]),
\end{equation}
together with linear changes of $(x,\partial)$ of determinant one.

\begin{theorem}\label{thm:rigid-semisimple}
Let $M$ be a holonomic left $A$-module with $\Ext_A^1(M,M)=0$.
Then $M$ is semisimple. Each of its simple summands is obtained from
the standard module $\C[x]$ by a finite composition of the
automorphisms in \eqref{eq:elementary-Weyl-automorphisms} and linear
changes of determinant one.
\end{theorem}

We separate the formal argument from the Fourier reduction. For a
connection $E$ on $\A^1$, its formal connection at infinity is
$V=E\otimes_{\C[x]}\C((t))$, where $t=x^{-1}$.
Write $\mathcal D_{\C((t))}=\C((t))\langle\partial_t\rangle$,
where $[\partial_t,f]=df/dt$ for $f\in\C((t))$.
The irregularity $\operatorname{irr}(V)$ is the sum of its formal
slopes, counted with their ranks. It is an integer. For a connection $E$, write $\mathcal{E}nd(E)$ for
$E^\vee\otimes E$ with its induced connection, and $\End(E)$
for the algebra of horizontal endomorphisms of $E$.

\begin{lemma}\label{lem:rigid-formal-endomorphisms}
Let $E$ be a connection of rank $r$ on $\A^1$ with
$\Ext_A^1(E,E)=0$. Restriction to the formal punctured neighborhood
of infinity induces an isomorphism
\[
 \End_A(E)\xrightarrow{\sim}\End_{\mathcal D_{\C((t))}}(V).
\]
Moreover, $E$ is a direct sum of connections whose formal connections
at infinity are irreducible. In particular, $E$ is semisimple.
\end{lemma}
\begin{proof}
\emph{Step 1: all formal endomorphisms are algebraic.}
Put
\[
 e=\dim_\C\End_A(E),\qquad
 h=\dim_\C\End_{\mathcal D_{\C((t))}}(V),\qquad
 I=\operatorname{irr}(\mathcal{E}nd(V)).
\]
Since $E$ is a vector bundle with connection,
$\Ext_A^1(E,E)=H^1_{\mathrm{dR}}(\A^1,\mathcal{E}nd(E))$.
The ordinary de Rham index formula gives
\begin{equation}\label{eq:rigid-ordinary-index}
 e=r^2-I.
\end{equation}
For $j:\A^1\hookrightarrow\mathbb P^1$, the index formula for
the middle extension gives
\begin{equation}\label{eq:rigid-middle-index}
 \chi_{\mathrm{dR}}\bigl(\mathbb P^1,j_{!*}\mathcal{E}nd(E)\bigr)
       =r^2-I+h=e+h.
\end{equation}
Here de Rham degrees are normalized so that a connection on a smooth
curve has degrees $0,1,2$. These index formulas are the usual
Euler--Poincar\'e formulas; see
\cite[Proposition~3.1]{ArinkinRigid} and
\cite[\S4.4]{ArinkinFourier}.
The trace pairing identifies $\mathcal{E}nd(E)$ with its dual. Duality on
$\mathbb P^1$ therefore makes both $H^0$ and $H^2$ of its middle
extension have dimension $e$. Equation~\eqref{eq:rigid-middle-index}
implies
\[
 \dim H^1_{\mathrm{dR}}\bigl(\mathbb P^1,j_{!*}\mathcal{E}nd(E)\bigr)=e-h.
\]
Restriction of horizontal endomorphisms is injective, so $e\leq h$.
Nonnegativity of the displayed dimension gives the reverse inequality.
Thus $e=h$, and the injective restriction map is an isomorphism of
algebras.

\emph{Step 2: indecomposable formal summands are irreducible.}
Decompose $V$ into indecomposable formal connections. Their projection
idempotents lift uniquely by Step~1, giving a corresponding direct-sum
decomposition of $E$. Each summand still has vanishing self-$\Ext^1$.
It suffices to consider one such summand, denoted again by $E$, for
which $V$ is indecomposable.

The formal classification in characteristic zero describes $V$ as
$W\otimes U_m$, where $W$ is an irreducible formal connection of
rank $r_0$ and $U_m$ is the regular singular unipotent connection
with one Jordan block of size $m$. To recall why, after a finite
ramified cover the exponential factors of an indecomposable connection
form a single orbit; the regular singular factor on one branch has
one Jordan block. Its rank-one factor descends by induction to $W$,
and the unipotent block descends to $U_m$.
After this cover, summands with distinct exponential factors have no
horizontal homomorphisms between them.
The centralizer of one Jordan block has dimension $m$, while a regular
singular tensor factor does not change positive slopes. Hence
\begin{equation}\label{eq:formal-Jordan-numbers}
 \rank V=mr_0,\qquad \dim\End(V)=m,
 \qquad \operatorname{irr}(\mathcal{E}nd(V))
            =m^2\operatorname{irr}(\mathcal{E}nd(W)).
\end{equation}

Step~1 and equation~\eqref{eq:rigid-ordinary-index} now give
\[
 m=m^2\bigl(r_0^2-\operatorname{irr}(\mathcal{E}nd(W))\bigr).
\]
The expression in parentheses is an integer. Thus $m=1$, and $V=W$
is irreducible. A nonzero proper $A$-submodule of $E$ would give a
nonzero proper formal subconnection: both the submodule and its
quotient are coherent connections on $\A^1$, and a connection has
no nonzero subquotient supported at finitely many points. This is
impossible. Every summand obtained in Step~2 is therefore simple.
\end{proof}

\begin{proof}[Proof of Theorem~\ref{thm:rigid-semisimple}]
Suppose $M\ne0$. After a generic determinant-one linear change of the Weyl algebra generators, $M$ is a connection: one can make a nonzero annihilating operator for each of finitely many generators monic in $\partial$, so that the module is finite over $\C[x]$. Its $\C[x]$-torsion is a finite-dimensional $A$-module and is therefore zero. Twisting preserves self-$\Ext^1$, so Lemma~\ref{lem:rigid-formal-endomorphisms} proves semisimplicity.
It remains to describe a simple summand $E$, whose formal connection
$V$ at infinity is irreducible and whose self-$\Ext^1$ vanishes.

\emph{Step 1: remove the integral part of the exponential factor.}
Suppose $\rank E=r>1$. By the formal classification, after the cover
$t=z^r$ the exponential factors of $V$ are the conjugates of a
minimally ramified Laurent polynomial
$\phi\in z^{-1}\C[z^{-1}]$, with rank-one regular singular factors.
The terms of $\phi$ with exponents divisible by $r$ are a polynomial
in $x=t^{-1}$ with zero constant term. An exponential twist removes
these terms globally. Denote the resulting connection by $E'$ and
its formal connection by $V'$. It remains simple and has vanishing
self-$\Ext^1$. Its single slope $s$ is positive and nonintegral.
There is a nonintegral term because otherwise the irreducible formal
connection would have rank one.

Since $r>1$, simplicity gives $H^0_{\mathrm{dR}}(\A^1,E')=0$.
The ordinary index formula implies
\[
 r-rs=-\dim H^1_{\mathrm{dR}}(\A^1,E')\leq0.
\]
Consequently $s>1$. Let $q\geq2$ be the denominator of $s$ in
lowest terms. Among the $r^2$ ordered pairs of conjugate
exponential factors, exactly $r^2(1-1/q)$ have distinct leading
coefficients. The corresponding factors of $\mathcal{E}nd(V')$ have slope
$s$. Therefore
\[
 \operatorname{irr}(\mathcal{E}nd(V'))\geq r^2(1-1/q)s.
\]
On the other hand, \eqref{eq:rigid-ordinary-index} and simplicity
give $\operatorname{irr}(\mathcal{E}nd(V'))=r^2-1$. If $s\geq2$, the preceding
inequality would give $\operatorname{irr}(\mathcal{E}nd(V'))\geq r^2$.
We have proved
\begin{equation}\label{eq:rigid-slope-window}
 1<s<2.
\end{equation}

\emph{Step 2: Fourier transformation lowers the rank.}
The formal stationary-phase theorem identifies the finite
singularities of $\mathcal F(E')$ with the slope-at-most-one
parts of $V'$, after the appropriate twists $T_{ax}$ with
$a\in\C$. There are no such parts,
because every slope of $V'$ equals $s>1$. Thus $\mathcal F(E')$ is
again a connection on the whole affine line; see
\cite[Corollary~2.3]{ArinkinFourier}. Its rank is
\begin{equation}\label{eq:rigid-Fourier-rank}
 r'=\operatorname{irr}(V')-\rank(V')=r(s-1),
\end{equation}
by the Fourier rank formula
\cite[Proposition~3.5]{ArinkinRigid}. In view of
\eqref{eq:rigid-slope-window}, $0<r'<r$.
Fourier transformation preserves simplicity and self-$\Ext^1$.
We may therefore repeat Steps~1 and~2. The positive integer rank
strictly decreases until it is one.

\emph{Step 3: reverse transformations.}
A rank-one connection on $\A^1$ has trivial underlying line bundle
and is of the form $d+b(x)\,dx$, with $b\in\C[x]$.
Choose $g\in\C[x]$ with $g'=b$. An exponential twist identifies it
with the standard connection. Reversing all the transformations
proves the assertion for every simple summand and then for $M$.
\end{proof}

\begin{corollary}\label{cor:rigid-arithmetic}
Conjecture~\ref{conj:kont-six} holds for all holonomic $A_1$-modules.
More precisely, if $\Ext^1(M,M)=0$, its arithmetic support, in the sense of \cite[\S2.4]{Kont}, is an effective cycle
\[
 \sum_{i=1}^a m_i[L_i],\qquad m_i\in\Z_{>0},
\]
where the distinct curves $L_i\subset T^*\A^1$ are pairwise disjoint smooth closed copies of $\A^1$. There is a symplectic automorphism of $\A^2$ that sends all the $L_i$ to distinct parallel lines. Equivalently, they are distinct fibers of one coordinate function on $\A^2$. After choosing and localizing an arithmetic model, the scheme-theoretic
$p$-support is the reduced union of the $L_{i,s}^{(1)}$, and the
normalized multiplicity of $L_{i,s}^{(1)}$ is $m_i$.
\end{corollary}
\begin{proof}
The class of extended motivic-exponential $\mathcal D$-modules is closed under external
products. This holds for its generating exponential connections,
and extends to the whole class by induction on the defining
operations, using exactness over $\C$ and compatibility of external
products with direct and inverse images. Pullback along the diagonal
then gives closure under tensor products. The Fourier transform is
an integral transform with exponential kernel, and $T_g$ is tensoring
with an exponential connection. Both operations therefore preserve
the class. Linear determinant-one changes are generated by Fourier
transformation, linear coordinate scalings, and the twists
$T_{a x^2/2}$ with $a\in\C$. Theorem~\ref{thm:rigid-semisimple} therefore places
every simple summand, and hence $M$, in that class.

For the arithmetic assertion, choose a finitely generated integral
$\Z$-algebra containing all coefficients of the finitely many
transformations, and localize it so that their inverses and the
module isomorphisms are defined. The action of an exponential twist on
the center is given by
\[
 (\partial+g'(x))^p=\partial^p+(g'(x))^p.
\]
Fourier transformation and linear changes likewise
induce the Frobenius twists of their classical linear symplectic
maps on the center. The standard module has scheme-theoretic $p$-support
the zero section, with normalized multiplicity one. Each simple
summand therefore has constant scheme-theoretic support equal to the image
of that section under a fixed composition of symplectic
automorphisms, again with multiplicity one.

These images are smooth closed copies of $\A^1$. The central
annihilator of a direct sum is the intersection of the central
annihilators of its summands, so its support is their reduced union.
The generic multiplicities add, giving the stated cycle.

It remains to prove disjointness. Let $S_i,S_j$ be simple summands
with distinct support curves $L_i,L_j$.
For $q\in L_i\cap L_j$, let $I_q(L_i,L_j)$ denote the local
intersection multiplicity at $q$.
The Ext--intersection formula
of \cite[\S3.4, final paragraph]{Kont} gives
\begin{equation}\label{eq:affine-lines-ext-intersection}
 \dim_\C\Ext_A^1(S_i,S_j)
       =\sum_{q\in L_i\cap L_j}I_q(L_i,L_j).
\end{equation}
We verify its hypothesis that the supports have disjoint truncated
Puiseux data at infinity. We use the truncated germs at infinity defined in
\cite[\S3.4]{Kont}: after choosing coordinates in which the branch
approaches the point at infinity on the $x$-axis, the Puiseux
expansion of $y$ is considered modulo $O(1/x)$. By the Abhyankar--Moh theorem \cite{AM},
straighten $L_i$ to $y=0$ and rescale one coordinate to make the
automorphism symplectic. In dimension two it lifts to a Weyl-algebra
automorphism, so the transformation preserves the Ext groups and
transports the arithmetic supports; see \cite[\S\S1.1 and 4]{BKKAut}.
Parametrize the transformed $L_j$ by $x=u(t),y=v(t)$, with
$u,v\in\C[t]$. Its unique branch at infinity cannot have the
truncated germ of $y=0$: that would require $x\to\infty$ and
$y=O(1/x)$, forcing $v=0$ and hence $L_j=L_i$.
This proves the required disjointness of the truncated data and
justifies \eqref{eq:affine-lines-ext-intersection}.
The group on the left is a direct summand of $\Ext_A^1(M,M)$, so
it vanishes. Thus $L_i\cap L_j=\varnothing$.

Finally, straighten one component to $y=0$. On every other
$L_i\simeq\A^1$, the coordinate $y$ is a nowhere-vanishing polynomial
in an affine parameter, hence a nonzero constant $a_i$. Therefore
$L_i=\{y=a_i\}$. This proves the simultaneous assertion.
\end{proof}

\subsection{Rigidity on other smooth curves}\label{subsec:rigid-curves}
Throughout this subsection, $X$ is a smooth connected complex algebraic
curve. All Ext groups are taken in the category of algebraic
$\mathcal D_X$-modules, without imposing conditions on the formal types.

\begin{lemma}\label{lem:curve-trace-injection}
Let $M$ be a holonomic $\mathcal D_X$-module of positive generic rank.
Tensoring extensions of the trivial connection with $M$ defines an
injective linear map
\begin{equation}\label{eq:curve-trace-injection}
 H^1_{\mathrm{dR}}(X)\longrightarrow
 \Ext^1_{\mathcal D_X}(M,M).
\end{equation}
\end{lemma}
\begin{proof}
The Spencer resolution identifies
$H^1_{\mathrm{dR}}(X)$ with
$\Ext^1_{\mathcal D_X}(\mathcal O_X,\mathcal O_X)$.
Represent a class $\alpha$ by an extension
\[
 0\longrightarrow\mathcal O_X\longrightarrow E_\alpha
 \longrightarrow\mathcal O_X\longrightarrow0.
\]
Its underlying sequence of $\mathcal O_X$-modules is locally split,
so $E_\alpha$ is a vector bundle with connection. Tensoring over
$\mathcal O_X$ with $M$, with the tensor-product connection, therefore
gives an exact sequence
\[
 0\longrightarrow M\longrightarrow M\otimes_{\mathcal O_X}E_\alpha
 \longrightarrow M\longrightarrow0.
\]
This construction is compatible with Baer sums and scalar multiplication,
and hence defines \eqref{eq:curve-trace-injection}.

Choose a dense open subset $U\subset X$ on which $M$ is a vector bundle
with connection $E$ of rank $r>0$. The restriction of the resulting
extension class is the image of $\alpha|_U$ under
\[
 H^1_{\mathrm{dR}}(U)\longrightarrow
 H^1_{\mathrm{dR}}(U,\mathcal{E}nd(E))
 =\Ext^1_{\mathcal D_U}(E,E),
\]
induced by the scalar inclusion
$\mathcal O_U\to\mathcal End_{\mathcal O_U}(E)$,
which sends a local regular function $a$ to the endomorphism
$e\mapsto ae$. Composing with the trace map gives
multiplication by $r$ on $H^1_{\mathrm{dR}}(U)$. Thus, if the
self-extension of $M$ splits, then $r\alpha|_U=0$, and therefore
$\alpha|_U=0$.

The restriction map $H^1_{\mathrm{dR}}(X)\to H^1_{\mathrm{dR}}(U)$
is injective: in the localization sequence, the contribution of the
finite set $X\setminus U$ begins in degree two. Consequently $\alpha=0$.
\end{proof}

\begin{proposition}\label{thm:rigid-other-curves}
Let $X$ be a smooth connected complex algebraic curve not isomorphic
to $\mathbb P^1$. Then Conjecture~\ref{conj:kont-six} holds for every
holonomic $\mathcal D_X$-module $M$ with
$\Ext^1_{\mathcal D_X}(M,M)=0$.
Such a module is semisimple. If $X$ is not isomorphic to $\A^1$ either,
then $M$ is supported at finitely many points.
\end{proposition}
\begin{proof}
For $X\simeq\A^1$, the assertions follow from
Theorem~\ref{thm:rigid-semisimple} and
Corollary~\ref{cor:rigid-arithmetic}. Otherwise,
$H^1_{\mathrm{dR}}(X)\ne0$, since the only smooth connected complex
curves with vanishing first de Rham cohomology are $\A^1$ and
$\mathbb P^1$. Lemma~\ref{lem:curve-trace-injection} and the assumed
self-Ext vanishing therefore force $M$ to have generic rank zero,
so its support is finite. By Kashiwara's equivalence,
\[
 M\simeq\bigoplus_{x\in S}(i_x)_+\C^{m_x},\qquad
 i_x:\{x\}\hookrightarrow X,
\]
for a finite set $S$ and nonnegative integers $m_x$. Thus $M$ is
semisimple. The trivial connection on a point belongs to the extended
motivic-exponential class, which is closed under direct images and
finite direct sums. Hence $M$ satisfies Conjecture~\ref{conj:kont-six}.
\end{proof}

\subsection{Irreducible modules on the projective line}\label{subsec:rigid-projective-line}
We now take $X=\mathbb P^1_\C$. For an irreducible module of positive
generic rank, full self-Ext vanishing also constrains the formal types.
We first record the formal calculation needed below. Write $\mathbf1$
for the trivial rank-one connection over $K=\C((t))$.

\begin{lemma}\label{lem:formal-small-endomorphisms}
Let $V$ be a nonzero formal meromorphic connection over $K$, and put
\[
 s=\dim_\C H^0_{\mathrm{dR}}(K,V),\qquad
 h=\dim_\C H^0_{\mathrm{dR}}(K,\mathcal{E}nd(V)).
\]
Then $h\geq s^2$, with equality exactly when $V$ is trivial.
If $h-s^2=1$, then either $V$ is regular singular with unipotent
monodromy, or
\begin{equation}\label{eq:formal-one-nontrivial-factor}
 V\simeq\mathbf1^{\oplus s}\oplus W
\end{equation}
for an irreducible nontrivial formal connection $W$.
If, in addition, $\det V$ has trivial formal monodromy, then the
formal monodromy of $V$ is quasi-unipotent. In the second case,
the formal monodromy $T_W$ of $W$ satisfies
$T_W^{2m}=I$, where $m=\rank W$.
\end{lemma}
\begin{proof}
\emph{Step 1: the lower bound.}
The formal classification shows that
$\dim H^0_{\mathrm{dR}}(K,V^\vee)=s$: both dimensions count the
Jordan blocks in the regular singular unipotent summand.
Horizontal vectors are linearly independent over $K$ whenever they
are linearly independent over $\C$. Thus the map
\[
 H^0_{\mathrm{dR}}(K,V)\otimes_\C
 H^0_{\mathrm{dR}}(K,V^\vee)
 \longrightarrow H^0_{\mathrm{dR}}(K,\mathcal{E}nd(V)),
 \qquad v\otimes\ell\longmapsto(w\mapsto\ell(w)v),
\]
is injective. If $V$ is nontrivial, its maximal trivial submodule has
rank $s<\rank V$. Every endomorphism in the image has its image in
that submodule, so their span does not contain $\id_V$. Hence
$h\geq s^2+1$ in this case. For $V\simeq\mathbf1^{\oplus s}$ one has
$h=s^2$.

\emph{Step 2: the case $h-s^2=1$.}
Decompose $V=V_{\mathrm{uni}}\oplus V_{\mathrm{rest}}$, where
$V_{\mathrm{uni}}$ is the regular singular unipotent summand and
$V_{\mathrm{rest}}$ consists of all other formal summands.
There are no horizontal homomorphisms between these two summands.
If the unipotent Jordan blocks have sizes $m_1,\ldots,m_s$, then
\[
 \dim\End(V_{\mathrm{uni}})
 =\sum_{i,j=1}^s\min(m_i,m_j)\geq s^2,
\]
with equality exactly when all $m_i=1$. If $V_{\mathrm{rest}}=0$,
the asserted unipotent case holds. Otherwise its endomorphism algebra
has dimension at least one, and equality $h-s^2=1$ forces
\[
 V_{\mathrm{uni}}\simeq\mathbf1^{\oplus s},\qquad
 \dim\End(V_{\mathrm{rest}})=1.
\]
As in the proof of Lemma~\ref{lem:rigid-formal-endomorphisms},
every indecomposable formal connection is $W_0\otimes U_a$ with
$W_0$ irreducible, and its horizontal endomorphism algebra has
dimension $a$. More than one direct summand gives independent
projection endomorphisms. It follows that $V_{\mathrm{rest}}$ is
irreducible, proving \eqref{eq:formal-one-nontrivial-factor}.

\emph{Step 3: formal monodromy.}
There is nothing to prove in the unipotent case. In the other case,
$\det W$ has trivial formal monodromy. The Levelt--Turrittin
classification realizes an irreducible formal connection of rank $m$
as the direct image of a rank-one connection under $t=u^m$.
Its formal monodromy cyclically permutes the $m$ rank-one factors.
Thus, for some $c\in\C^\times$,
\[
 T_W^m=cI,\qquad \det T_W=(-1)^{m-1}c.
\]
Since $\det T_W=1$, one has $c=(-1)^{m-1}$ and therefore
$T_W^{2m}=I$.
\end{proof}

A connection $E$ on an open subset $U\subset\PP^1$, with
inclusion $j$, is said to be cohomologically rigid with prescribed formal types if
\[
 H^1_{\mathrm{dR}}(\PP^1,j_{!*}\mathcal{E}nd(E))=0.
\]

\begin{theorem}\label{thm:rigid-projective-line}
Let $M$ be an irreducible holonomic $\mathcal D_{\mathbb P^1}$-module
of positive generic rank such that
$\Ext^1_{\mathcal D_{\mathbb P^1}}(M,M)=0$.
Then $M$ has at most one singular point, its formal local monodromies
are quasi-unipotent, and its underlying meromorphic connection is
cohomologically rigid with prescribed formal types.
\end{theorem}
\begin{proof}
\emph{Step 1: middle extensions at the singular points.}
Let $S$ be the singular locus, let
$j:U=\mathbb P^1\setminus S\hookrightarrow\mathbb P^1$, and put
$E=M|_U$. Since $M$ is irreducible, $M=j_{!*}E$.
If $S=\varnothing$, the irreducible connection on $\mathbb P^1$ is
the trivial rank-one connection, and all the assertions hold.
Assume $S\ne\varnothing$. For $x\in S$, choose a local parameter $t_x$ at $x$ and put
$K_x=\C((t_x))$. Let $V_x$ be the formal meromorphic connection
of $E$ at $x$, and set
\[
 s_x=\dim H^0_{\mathrm{dR}}(K_x,V_x),\qquad
 h_x=\dim H^0_{\mathrm{dR}}(K_x,\mathcal{E}nd(V_x)).
\]
The local description of the extension functors gives
\begin{align}
 0\longrightarrow M\longrightarrow j_+E
 &\longrightarrow\bigoplus_{x\in S}\delta_x^{\oplus s_x}
 \longrightarrow0,\label{eq:p1-star-extension}\\
 0\longrightarrow\bigoplus_{x\in S}\delta_x^{\oplus s_x}
 &\longrightarrow j_!E\longrightarrow M\longrightarrow0.
 \label{eq:p1-shriek-extension}
\end{align}
Here $\delta_x=(i_x)_+\C$ for $i_x:\{x\}\hookrightarrow\mathbb P^1$.
To see the multiplicities, complete at $x$ and decompose the formal
connection. A regular singular unipotent Jordan block contributes
one copy of $\delta_x$ to $j_+E/j_{!*}E$; the other formal blocks
contribute none. Their number is $s_x$.
Duality gives the second sequence, since
$\dim H^0_{\mathrm{dR}}(K_x,V_x^\vee)=s_x$.
These local extension functors and their formal descriptions are
reviewed in \cite[\S\S3.1--3.3]{ArinkinFourier}.

\emph{Step 2: bound the de Rham self-extension space on $U$.}
Derived adjunction and $\delta_x|_U=0$ give
$\Ext^i(j_!E,\delta_x)=0$ for every $i$.
Applying $\Hom(-,\delta_x)$ to
\eqref{eq:p1-shriek-extension} therefore gives
\begin{equation}\label{eq:p1-point-extensions}
 \dim\Ext^1(M,\delta_x)=s_x.
\end{equation}
Apply $\Hom(M,-)$ to \eqref{eq:p1-star-extension}.
The assumed self-Ext vanishing and adjunction yield an injection
\[
 H^1_{\mathrm{dR}}(U,\mathcal{E}nd(E))
 =\Ext^1(M,j_+E)
 \lhook\joinrel\longrightarrow
 \bigoplus_{x\in S}\Ext^1(M,\delta_x)^{\oplus s_x}.
\]
Thus
\begin{equation}\label{eq:p1-open-ext-bound}
 \dim H^1_{\mathrm{dR}}(U,\mathcal{E}nd(E))\leq\sum_{x\in S}s_x^2.
\end{equation}

\emph{Step 3: compare with the middle extension of $\mathcal{E}nd(E)$.}
Put $F=j_{!*}\mathcal{E}nd(E)$ and
$b=\dim H^1_{\mathrm{dR}}(\mathbb P^1,F)$.
Applying the local calculation in Step~1 to $\mathcal{E}nd(E)$ gives
\[
 0\longrightarrow F\longrightarrow j_+\mathcal{E}nd(E)
 \longrightarrow\bigoplus_{x\in S}\delta_x^{\oplus h_x}
 \longrightarrow0.
\]
Irreducibility of $E$ gives $\End_{\mathcal D_U}(E)=\C$.
Since the trace pairing identifies $\mathcal{E}nd(E)$ with its dual, duality
on $\mathbb P^1$ gives
\[
 \dim H^0_{\mathrm{dR}}(\mathbb P^1,F)
 =\dim H^2_{\mathrm{dR}}(\mathbb P^1,F)=1.
\]
Also $H^2_{\mathrm{dR}}(U,\mathcal{E}nd(E))=0$, since $U$ is affine.
With the de Rham normalization used above, $\delta_x$ contributes
one dimension in degree one and nothing in other degrees. The
resulting exact sequence is therefore
\[
 0\longrightarrow H^1_{\mathrm{dR}}(\mathbb P^1,F)
 \longrightarrow H^1_{\mathrm{dR}}(U,\mathcal{E}nd(E))
 \longrightarrow\bigoplus_{x\in S}\C^{h_x}
 \longrightarrow\C\longrightarrow0.
\]
Consequently
\[
 \dim H^1_{\mathrm{dR}}(U,\mathcal{E}nd(E))=b+\sum_{x\in S}h_x-1.
\]
Together with \eqref{eq:p1-open-ext-bound}, this proves
\begin{equation}\label{eq:p1-local-numerical-bound}
 b+\sum_{x\in S}(h_x-s_x^2)\leq1.
\end{equation}

\emph{Step 4: there is at most one singular point.}
At a genuine singular point of a middle extension, $V_x$ is
nontrivial. Indeed, the middle extension of a trivial formal
connection is a free module over $\C[[t_x]]$ with nonsingular
connection; faithful flatness of completion then implies that
$M$ is nonsingular at $x$.
Lemma~\ref{lem:formal-small-endomorphisms} gives
$h_x-s_x^2\geq1$ for each $x\in S$. Since $b\geq0$,
\eqref{eq:p1-local-numerical-bound} forces
\[
 |S|=1,\qquad b=0,\qquad h_x-s_x^2=1.
\]
The equality $b=0$ is cohomological rigidity with prescribed formal
types; see \cite[Proposition~3.4]{ArinkinRigid}.

\emph{Step 5: the determinant forces quasi-unipotence.}
Choose the coordinate so that $S=\{\infty\}$. The determinant of the
connection $E$ on $\A^1$ has the form
\[
 \det E\simeq(\mathcal O_{\A^1},d+a(x)\,dx),
 \qquad a(x)\in\C[x].
\]
Choose $G\in\C[x]$ with $G'=a$. This is an exponential connection,
so its formal monodromy at infinity is trivial.
Lemma~\ref{lem:formal-small-endomorphisms}, applied to $V_\infty$,
now proves the required quasi-unipotence.
\end{proof}

\begin{corollary}\label{cor:kont-six-projective-line}
Conjecture~\ref{conj:kont-six} holds for every irreducible holonomic
$\mathcal D_{\mathbb P^1}$-module $M$ with $\Ext^1(M,M)=0$.
\end{corollary}
\begin{proof}
If $M$ is supported at a point, it is the direct image of the trivial
connection on that point. Otherwise,
Theorem~\ref{thm:rigid-projective-line} proves both cohomological
rigidity and quasi-unipotence of its formal monodromies.
By \cite[Property~B]{SabbahRigid}, the underlying irreducible
connection has exponential-geometric origin. More explicitly, on
some dense open subset $U'\subset\mathbb P^1$ it is a submodule
of a cohomology module of $f_+(\mathcal O_Y^{\oplus a},d+d\phi)$,
where $Y$ is smooth, $f:Y\to U'$ is a morphism, and $\phi$ is a
regular function on $Y$. This belongs to the extended
motivic-exponential class by its defining closure properties.
For $j':U'\hookrightarrow\mathbb P^1$, the original module is
$j'_{!*}(M|_{U'})$, which is a submodule of $j'_+(M|_{U'})$.
Closure under direct images and subquotients therefore puts $M$
in the same class.
\end{proof}

\subsection{A counterexample concerning quantum integrable systems}
We now turn to a different assertion, concerning modules that need
not be holonomic. Here the support includes its normalized cycle
multiplicities, as in \cite[\S5.1]{Kont}.
For a smooth variety $U$, write
$\mathcal D(U)=\Gamma(U,\mathcal D_U)$.
\begin{conjecture}[Kontsevich, {\cite[Conjecture~7]{Kont}}]\label{conj:kont-seven}
Let $X$ be a smooth affine variety of dimension $n$ in characteristic
zero and let $Q\in\mathcal D(X)$. For a generic additive constant
$\lambda$, the support cycle of
$\mathcal D(X)/\mathcal D(X)(Q+\lambda)$ in large positive
characteristic is divisible by $p^{n-1}$ if and only if $Q$ belongs
to a finitely generated commutative subalgebra of $\mathcal D(X)$
of Krull dimension $n$.
\end{conjecture}

The following example also excludes a commutative subalgebra of the
required dimension after restriction to any nonempty open subset.
\begin{theorem}\label{thm:kont-seven-counterexample}
In $A_2=\C\langle x,y,\partial_x,\partial_y\rangle$, put
\begin{equation}\label{eq:integrability-operator}
 Q=\partial_x^3+\partial_y^3
    +x\partial_x^2\partial_y+y\partial_x\partial_y^2
    +4\partial_x\partial_y.
\end{equation}
For every nonempty open subset $U\subset\A^2$,
\[
 \{P\in\mathcal D(U):PQ=QP\}=\C[Q].
\]
The same equality holds in the ring of finite-order differential
operators with coefficients in $\C(x,y)$. Nevertheless, over every
field $k$ of characteristic $p>3$ and for every $\lambda\in k$,
\begin{equation}\label{eq:integrability-norm}
 \Nrd(Q+\lambda)
 =\bigl(\Xi^3+\mathrm H^3+X\Xi^2\mathrm H
                    +Y\Xi\mathrm H^2+\lambda^p\bigr)^p,
\end{equation}
where $X=x^p$, $Y=y^p$, $\Xi=\partial_x^p$, and
$\mathrm H=\partial_y^p$. Thus Conjecture~\ref{conj:kont-seven}
fails for $X=\A^2$, even if restriction to a nonempty open subset
is allowed in its conclusion.
\end{theorem}

We first determine the Poisson centralizer of the principal symbol.
On $\C[x,y,\xi,\eta]$ use $\{\xi,x\}=\{\eta,y\}=1$.
\begin{lemma}\label{lem:Poisson-centralizer}
For
\[
 h=\xi^3+\eta^3+x\xi^2\eta+y\xi\eta^2,
\]
one has
\[
 \{q\in\C(x,y)[\xi,\eta]:\{h,q\}=0\}=\C[h].
\]
\end{lemma}
\begin{proof}
\emph{Step 1: polynomial invariants after localization.}
In $\C[x,y,\xi^{\pm1},\eta^{\pm1}]$, set
\[
 \rho=\xi/\eta,\qquad J=\xi x-\eta y,
 \qquad K=J+3(\xi^2/\eta-\eta^2/\xi).
\]
Direct calculation gives $\{h,\rho\}=\{h,K\}=0$. Moreover,
\[
 \C[x,y,\xi^{\pm1},\eta^{\pm1}]
       =\C[\xi^{\pm1},\rho^{\pm1},h,K].
\]
Indeed, $\eta=\xi/\rho$,
$\xi x+\eta y=(h-\xi^3-\eta^3)/(\xi\eta)$, and $J$ is recovered
from $K$; these equations recover $x,y$. In these coordinates
\[
 \{h,-\}=-\frac{\xi^3}{\rho}\frac{\partial}{\partial\xi},
\]
with $\rho,h,K$ fixed. Its kernel is $\C[\rho^{\pm1},h,K]$.

\emph{Step 2: regularity at $\xi=0$.}
Put $B=K^3+27h(\rho^{-3}-\rho^3)$. Every element of this kernel
has a unique finite expression
\begin{equation}\label{eq:centralizer-expansion}
 q=\sum_{\substack{0\leq i\leq2,\ j\geq0\\ v\in\Z}}
       c_{ijv}(h)\rho^vK^iB^j,
 \qquad c_{ijv}\in\C[h],
\end{equation}
because $B$ is monic of degree three in $K$. In Laurent series in
$\xi$ over $\C[\eta^{\pm1},x,y]$,
\begin{align*}
 h&=\eta^3+O(\xi),&
 K&=-3\eta^2\xi^{-1}-\eta y+O(\xi),\\
 B&=9\eta^4(6x-y^2)\xi^{-1}+O(1).
\end{align*}
A nonzero summand in \eqref{eq:centralizer-expansion} has
$\xi$-valuation $m=v-i-j$, with leading coefficient
\[
 (-3)^i9^j c_{ijv}(\eta^3)\eta^{i+3j-m}(6x-y^2)^j.
\]
Terms of minimal valuation cannot cancel: first compare powers of
$6x-y^2$, which is transcendental over $\C[\eta^{\pm1}]$; for fixed
$j$, the three possible values of $i$ give distinct exponents of
$\eta$ modulo three. Thus regularity at $\xi=0$ implies
\begin{equation}\label{eq:centralizer-first-bound}
 v\geq i+j\qquad\text{whenever }c_{ijv}\ne0.
\end{equation}

\emph{Step 3: regularity at $\eta=0$.}
The corresponding expansions in $\eta$ are
\begin{align*}
 h&=\xi^3+O(\eta),&
 K&=3\xi^2\eta^{-1}+\xi x+O(\eta),\\
 B&=9\xi^4(x^2-6y)\eta^{-1}+O(1).
\end{align*}
The valuation of the same summand is $-v-i-j$. Comparing powers of
$x^2-6y$ and exponents of $\xi$ modulo three gives
\begin{equation}\label{eq:centralizer-second-bound}
 v\leq-i-j\qquad\text{whenever }c_{ijv}\ne0.
\end{equation}
The two bounds force $i=j=v=0$. Consequently every polynomial
Poisson-commuting with $h$ belongs to $\C[h]$.

\emph{Step 4: rational coefficients introduce no further invariants.}
Suppose that $q\in\C(x,y)[\xi,\eta]$ Poisson-commutes with $h$ and
has a pole of order $m>0$ along an irreducible polynomial
$f\in\C[x,y]$. Taking the term of highest pole order in
$\{h,q\}=0$ shows that $f$ divides $\{h,f\}$. Equating the three
coefficients of this quadratic polynomial in $\xi,\eta$ gives
\[
 \begin{pmatrix}3&x\\2x&2y\\y&3\end{pmatrix}
 \begin{pmatrix}f_x\\f_y\end{pmatrix}=0\pmod f.
\]
Its maximal minors are
\[
 6y-2x^2,\qquad 9-xy,\qquad 6x-2y^2.
\]
They have no common irreducible factor: already the first two have
only finitely many common zeros. The matrix therefore has rank two
over $\Frac(\C[x,y]/(f))$, so $f$ divides both $f_x$ and $f_y$.
This is impossible for a nonconstant polynomial in characteristic zero.
Thus $q$ has no poles in its coefficients, and Step~3 applies.
\end{proof}

\begin{proof}[Proof of Theorem~\ref{thm:kont-seven-counterexample}]
\emph{Step 1: the centralizer with rational coefficients.}
Use the differential-order filtration on
$\C(x,y)\langle\partial_x,\partial_y\rangle$.
The principal symbol of $Q$ is $h$. If an operator $T$ commutes
with $Q$, its principal symbol Poisson-commutes with $h$ and hence,
by Lemma~\ref{lem:Poisson-centralizer}, belongs to $\C[h]$.
Homogeneity in $\xi,\eta$ makes this symbol $c h^m$ for some
$c\in\C$ and $m\geq0$. Subtracting $cQ^m$ lowers the order.
Induction proves $T\in\C[Q]$. Restriction to the generic point embeds
$\mathcal D(U)$ into this ring for every nonempty open $U\subset\A^2$,
which proves all the centralizer assertions.

\emph{Step 2: the $p$th-power identity.}
Work over a field $k$ of characteristic $p>3$. Introduce Weyl algebra
generators
\[
 u=\partial_x,\quad v=\partial_y,\quad
 d_u=-x,\quad d_v=-y,
\]
so $[d_u,u]=[d_v,v]=1$. Put
\[
 E=ud_u+vd_v,\qquad D=uvE,\qquad g=u^3+v^3.
\]
Expanding products gives $Q=-D+g$. Since $E(uv)=uv(E+2)$,
\[
 D^p=(uv)^p\prod_{j=0}^{p-1}(E+2j)=(uv)^p(E^p-E).
\]
The two Euler operators commute, and
\[
 E^p-E=u^p d_u^p+v^p d_v^p.
\]
Let $\delta=uv(u\,\partial/\partial u+v\,\partial/\partial v)$ be the
derivation of $k[u,v]$ induced by $D$. The formula for a vector field
plus a multiplication operator gives the identity
\[
 (-D+g)^p=-D^p+g^p+(-\delta)^{p-1}(g)
\]
in the Weyl algebra, where the last term denotes multiplication by
the polynomial $(-\delta)^{p-1}(g)$. This polynomial is zero, since
\[
 (-\delta)^{p-1}(g)=(uv)^{p-1}g
                 \prod_{j=0}^{p-2}(3+2j)=0.
\]
The factor with $j=(p-3)/2$ vanishes. Returning to the original
central coordinates, we obtain
\begin{equation}\label{eq:integrability-p-power}
 Q^p=\Xi^3+\mathrm H^3+X\Xi^2\mathrm H
                         +Y\Xi\mathrm H^2=:G.
\end{equation}
The reduced norm has degree $p^2$, so
\[
 \Nrd(Q+\lambda)^p
 =\Nrd\bigl((Q+\lambda)^p\bigr)
 =(G+\lambda^p)^{p^2}.
\]
Injectivity of Frobenius on the polynomial center proves
\eqref{eq:integrability-norm}.

\emph{Step 3: the support cycle.}
The normalized codimension-one support cycle of
$A_2(k)/A_2(k)(Q+\lambda)$ is the divisor of
$\Nrd(Q+\lambda)$. Indeed, after an \'etale splitting and
localization at the generic point of an irreducible divisor
in the center, the quotient is, as a module over the center,
a direct sum of $p^2$ copies of the cokernel of the matrix
representing $Q+\lambda$. The length of this cokernel equals
the valuation of its determinant, and normalization divides
the length of the quotient by $p^2$.
By \eqref{eq:integrability-norm}, the normalized $p$-support
cycle is therefore
\[
 p\,\operatorname{div}(G+\lambda^p).
\]
It is divisible by $p=p^{2-1}$ for every $\lambda\in k$.
Every commutative subalgebra of $\mathcal D(U)$ containing $Q$,
however, is contained in $\C[Q]$ and has Krull dimension at most one.
This contradicts the asserted conclusion of
Conjecture~\ref{conj:kont-seven}.
\end{proof}

\section*{AI methodology}
The authors used ChatGPT (OpenAI) to assist with developing and drafting
mathematical arguments and examples, literature searches, language editing,
and LaTeX preparation. The authors are responsible for the final text,
including its mathematical claims, proofs, and references.
\printbibliography[title={References}]
\end{document}